\documentclass[11pt,reqno,centertags,a4paper,noamsfonts]{amsart}
\usepackage[margin=25.4mm]{geometry}

\usepackage[latin2]{inputenc}
\usepackage{amsmath}
\usepackage{amssymb}
\usepackage{amsthm}
\usepackage{graphicx}
\usepackage{tikz}
\usepackage{esint}
\usepackage{xcolor}
\usepackage{xxcolor}
\usepackage{floatrow}
\usepackage{color}
\usepackage{epsfig}
\usepackage[english]{babel}
\usepackage{hyperref}

\usepackage{comment}
\usepackage{mathtools}
\mathtoolsset{showonlyrefs}
\usepackage{enumitem}
\usepackage{csquotes}
\usepackage{marvosym}

\makeatletter
\newcommand{\leqnomode}{\tagsleft@true}
\newcommand{\reqnomode}{\tagsleft@false}
\makeatother

\newtheorem{theorem}{Theorem}[section]
\newtheorem{assumption}[theorem]{Hypotheses}
\newtheorem*{genAssumption}{General Hypotheses}
\newtheorem{remark}[theorem]{Remark}
\newtheorem{proposition}[theorem]{Proposition}
\newtheorem{lemma}[theorem]{Lemma}

\newtheorem{definition}[theorem]{Definition}

\def\XXint#1#2#3{{\setbox0=\hbox{$#1{#2#3}{\int}$ }
		\vcenter{\hbox{$#2#3$ }}\kern-.6\wd0}}

\allowdisplaybreaks[2]

\newcommand{\supp}{\operatorname{supp}} 
\newcommand{\divv}{\operatorname{div}}

\newcommand{\diag}{\operatorname{diag}}

\newcommand{\dist}{\operatorname{dist}}

\newcommand{\dd}{\mathrm{d}}
\newcommand{\loc}{\mathrm{loc}}
\newcommand{\ve}{\varepsilon}
\newcommand{\HH}{\hyperlink{H}{\bf(H)}}
\newcommand{\HHp}{\hyperlink{H'}{\bf(H')}}
\newcommand{\ga}{\hyperlink{ga}{\textup{General Hypotheses}}}

\newcommand{\ol}{\overline}

\newcommand{\wt}{\widetilde}
\newcommand{\coleq}{\mathrel{\mathop:}=}

\allowdisplaybreaks[2]

\newcommand{\Om}{\Omega}

\newcommand{\mquad}{\mkern-18mu}
\newcommand{\mqquad}{\mkern-36mu}

\newcommand{\de}{{\delta,\varepsilon}}
\renewcommand{\de}{\varepsilon}
\newcommand{\dee}{{\varepsilon,E}}
\newcommand{\rreg}{\varrho}

\numberwithin{equation}{section}

\begin{document}
	
	\title[Existence for ERDS]
	{Global existence analysis of energy-reaction-diffusion systems}

	\begin{abstract}
		We establish global-in-time existence results for thermodynamically
		consistent reaction-(cross-)diffusion systems coupled to an equation
		describing heat transfer. Our main interest is to model species-dependent
		diffusivities, while at the same time ensuring thermodynamic consistency.
		A key difficulty of the non-isothermal case lies in the intrinsic presence of
		cross-diffusion type phenomena like the Soret and the Dufour effect: due to
		the temperature/energy dependence of the thermodynamic equilibria, a
		nonvanishing temperature gradient may drive a concentration flux even in a
		situation with constant concentrations; likewise, a nonvanishing
		concentration gradient may drive a heat flux even in a case of spatially
		constant temperature.  We use time discretisation and regularisation techniques and derive a priori estimates based on a suitable entropy and the associated entropy production.
		Renormalised solutions are used in cases where non-integrable
		diffusion fluxes or reaction terms appear.
	\end{abstract}
	
	\author[J. Fischer]{Julian Fischer}
	\address{Julian Fischer, Institute of Science and Technology Austria (IST Austria),
		Am Campus 1, 3400 Klosterneuburg, Austria, 
		E-Mail: julian.fischer@ist.ac.at}
	
	\author[K. Hopf]{Katharina Hopf}
	\address{Katharina Hopf, Weierstrass Institute for Applied Analysis and Stochastics (WIAS),
		Mohrenstra\ss{}e 39, 10117 Berlin, Germany, 
		E-Mail: hopf@wias-berlin.de}
	
	\author[M. Kniely]{Michael Kniely}
	\address{Michael Kniely, Institute of Science and Technology Austria (IST Austria),
		Am Campus 1, 3400 Klosterneuburg, Austria
		\newline
		\emph{Current address:} 
		Faculty of Mathematics, TU Dortmund University, Vogelpothsweg 87, 44227 Dortmund, Germany, 
		E-Mail: michael.kniely@tu-dortmund.de
	}
	
	\author[A. Mielke]{Alexander Mielke}
	\address{Alexander Mielke, Weierstra\ss-Institut f\"ur Angewandte  Analysis und Stochastik (WIAS),
		Mohrenstra\ss{}e 39, 10117 Berlin, Germany, and Institut f\"ur Mathematik, Humboldt-Universit\"at zu Berlin, Unter den Linden 6, 10099 Berlin, Germany, 
		E-Mail: alexander.mielke@wias-berlin.de}
	
	\date{}
	
	
	
	\maketitle
	
	\section{Introduction}
	
	The purpose of this paper is to establish global-in-time existence results for
	a class of reaction-diffusion systems arising in the   modelling of
	non-isothermal chemical reactions.  We consider thermodynamically consistent
	models  that are based on the total entropy 
	\begin{align}\label{eq:entropy.functional}
	\mathcal{S}(c,u):=\int_\Omega S(c(x),u(x))\,\dd x
	\end{align}
	as a Lyapunov functional. The point here is that the entropy density $S$
	is given in terms of the vector $c(x) \in [0,\infty)^I$ of the concentrations
	and the internal energy density $u(x)$ rather than the more commonly used
	temperature $\theta(x)$. Hence, we will be able to rely on concavity of $S$ in
	$(c,u)$, while concavity in $(c, \theta)$ does not hold in general, cf.\ e.g.\
	\cite{agh_2002_thermodynamic}. Even more, following~\cite{hhmm_2018,
		Mielke_2013_thermomechanical, MielkeMittnenzweig_2018_convergence} the total
	energy-reaction-diffusion system may be written as a gradient-flow equation for
	$-\mathcal S$.  For details on the derivation of our models, see
	Section~\ref{ssec:modelling}.
	
	Even without accounting for temperature dependence, developing an existence theory for solutions to entropy-driven reaction-diffusion systems has proven challenging: for instance, even for the simple example of a reaction-diffusion system with Fick-law diffusion
	\begin{align}
	\label{ReactionDiffusionNoCrossDiffusion}
	\dot c_i = a_i \Delta c_i + R_i(c)
	\end{align}
	(with $a_i>0$) and entropy-producing chemical reactions $R_i(c)$, the
	global-in-time existence of generalised solutions has only been shown rather
	recently \cite{Fischer_2015_renormalized} and relies on the concept of
	renormalised solutions; weak (or even strong) solutions are only known to exist
	under more restrictive assumptions on the
	reactions~\cite{Pierre_2010_GlobalExistenceSurvey,CGV_2019}.
	
	The existence analysis for thermodynamically realistic non-isothermal
	reaction-diffusion systems involves further challenges: due to the
	temperature-dependence of the thermodynamic equilibrium, a nonvanishing
	temperature gradient may drive a concentration flux even in situations with a
	vanishing concentration gradient, a phenomenon known as the \emph{Soret
		effect}.  Similarly, a nonvanishing concentration gradient may drive a heat
	flux even if the temperature is spatially constant (the \emph{Dufour effect}).
	Thermodynamically realistic models for reaction-diffusion systems must
	therefore allow for \emph{cross-diffusion} effects between internal energy
	density $u$ and the concentrations $c_i$.  The methods in
	\cite{Fischer_2015_renormalized} (respectively
	\cite{ChenJuengel_2019_renormalised}) rely heavily on the diagonal structure
	(respectively the dominantly diagonal structure) of the diffusion; without
	substantial new ideas, they do not apply to a setting with strong
	cross-diffusion.
	
	In the present work, we resolve these mathematical difficulties and provide an
	existence analysis for generalised solutions to a  nontrival class  of
	reaction-(cross-)diffusion systems modelling non-isothermal chemical
	reactions. Our class of thermodynamically consistent models is derived as a
	gradient flow in Onsager form (see Section~\ref{ssec:modelling} below); the
	resulting equations are typically of the form
	\begin{subequations}
		\label{ReactionDiffusionTemperature}
		\begin{align}
		\label{ReactionDiffusionTemperaturea}
		\dot c_i &= \divv \Big( m_i(c,u) \nabla \log \frac{c_i}{w_i(u)}
		+   a(c,u)  c_i \frac{w'_i(u)}{w_i(u)} \nabla u \bigg) + R_i(c,u),
		\\
		\label{ReactionDiffusionTemperatureb}
		\dot u &=
		\divv \!\big( \,  a(c,u)  \nabla u\big),
		\\
		&   \text{where } a(c,u)=-\pi_1(c,u)\Big(\hat \sigma''(u) + \sum_{i=1}^I c_i 
		\tfrac{w''_i(u)}{w_i(u)}\Big)>0. 
		\end{align}
	\end{subequations}
	Here, $w_i(u)$ denotes the equilibrium concentration of the chemical 
	species $\mathcal{A}_i$   (which
	may depend on the internal energy density $u$),  the function $m_i(c,u)\ge0$ 
	describes the diffusive mobility of $\mathcal{A}_i$, 
	the function  $\pi_1(c,u)$
	describes the heat conductivity, and $\hat \sigma(u)$ is related to the thermal part of
	the entropy density.
	
	It is rather immediate that the system
	\eqref{ReactionDiffusionTemperaturea}--\eqref{ReactionDiffusionTemperatureb}
	accounts for the Soret effect: a nonvanishing temperature gradient may drive a
	concentration flux even for spatially constant concentrations $c_i$.  At first
	glance, the Dufour effect may appear to not be accounted for by the model, as a
	vanishing internal energy gradient $\nabla u\equiv 0$ entails the absence of a
	heat flux. However, even for a constant temperature $\theta$ the internal
	energy density $u$ may be non-constant, as the thermodynamic relation of
	temperature and internal energy density $\frac{1}{\theta}=\partial_u S(c,u)$
	also involves the concentrations.
	
	\subsection{Modelling}\label{ssec:modelling}
	
	Following~\cite{Mielke_2013_thermomechanical,hhmm_2018,MielkeMittnenzweig_2018_convergence},
	we consider  energy-reaction-diffusion systems that are motivated by the
	thermodynamically consistent models that are obtained as gradient-flow
	equations  written in \textit{Onsager form}. Given a state space 
	$\mathbf Q$ as a convex subset of a Banach space $X$ with states
	$Z\in \mathbf Q$, an Onsager operator $\mathbb{K}=\mathbb{K}(Z)$  is a
	possibly unbounded,  symmetric and positive semi-definite operator, 
	which may be seen as a generalisation of the inverse of the Riemannian metric
	tensor in a smooth manifold $\mathbf Q$.  With a differentiable driving
	functional $-\mathcal{S}: \mathbf Q \to \mathbb R$ (typically  a convex
	functional on $X$), the associated evolution problem reads
	\begin{align}\label{eq:OnsForm}
	\dot Z = \mathbb{K}(Z)D\mathcal{S}(Z),
	\end{align}
	where here $D\mathcal{S}$ denotes the Fr\'echet derivative of the entropy
	functional~$\mathcal{S}$, and $\dot Z$ is the time derivative of the state variable
	$Z=Z(t)$.  
	
	In this framework, different physical phenomena can easily be coupled by 
	taking $\mathbb K$ as a sum of  operators corresponding to individual
	processes. In this paper, we let 
	$\mathbb{K}:=\mathbb{K}_{\mathrm{diff}}+\mathbb{K}_{\mathrm{react}}$, where
	$\mathbb{K}_{\mathrm{diff}}$ accounts for diffusion and
	$\mathbb{K}_{\mathrm{react}}$ for the reactions.  In our application, the
	vector $Z$ of state variables consists of the concentrations $(c_i)_{i=1}^I$ of
	the $I$ species $\mathcal{A}_i$ diffusing and reacting  on the bounded Lipschitz
	domain $\Omega\subset \mathbb{R}^d$ and another suitable variable modelling
	variations in temperature. The latter could be temperature $\theta(t,x)$
	itself, the internal energy density $u(t,x)$ or some other suitable scalar
	quantity~\cite{Mielke_2011_generic}.  As in
	\cite{MielkeMittnenzweig_2018_convergence, hhmm_2018} we model changes in
	temperature using the internal energy density $u\ge0$ (see
	also~\cite{agh_2002_thermodynamic, Mielke_2011_gradientStructure,
		Mielke_2013_thermomechanical}).  The major  advantage of this choice is that the entropy density  $S(c,u)$ determining the entropy
	functional~\eqref{eq:entropy.functional}
	is concave in  $(c,u)$,  which is a vector of extensive variables,  see 
	\cite{LY_1999_secondLaw, agh_2002_thermodynamic}. Also note that the total energy $\mathcal E(c,u) = \int_\Omega
	u(x)\,\dd x$ takes the most simple form, and while
	chemical reactions typically induce changes in temperature, the internal energy is left invariant and the total energy is a conserved quantity along solutions of
	\eqref{ReactionDiffusionTemperature}. This also means that the Onsager operator
	$\mathbb{K}_{\mathrm{react}}(c,u)$ has a nontrivial kernel, namely span$\{
	(0,1)^T\}$. 
	
	The basis of our model are  entropy densities 
	$S(c,u)$ of the form (cf.~\cite{hhmm_2018,MielkeMittnenzweig_2018_convergence})
	\begin{align}
	S(c, u) = \sigma(u) - B(c|w(u)) 
	=\underbrace{\sigma(u)-\sum_{i = 1}^I w_i(u)+I}_{=:\hat\sigma(u)}+\sum_{i = 1}^I 
	\Big(c_i \log w_i(u) - \lambda(c_i) \Big),
	\end{align}
	in short
	\begin{align}
	\label{eq:S}
	S(c, u) = \hat\sigma(u)+\sum_{i = 1}^I 
	\Big(c_i \log w_i(u) - \lambda(c_i) \Big),
	\end{align}
	where we continue to use the Boltzmann function 
	\begin{align}\label{eq:deflb}
	\lambda(s):= s\log(s)-s+1,
	\end{align}
	the relative Boltzmann entropy $B(c|w):=\sum_{i=1}^I w_i\lambda(c_i/w_i)$, and the
	thermal part $\sigma(u)$ of the entropy
	density when the concentrations $c=(c_i)_{i=1}^I$ are in equilibrium.  
	From the formula
	\[
	D_c S(c,u)= \big( \log w_i(u)- \lambda'(c_i) \big)_i = - \Big(
	\log\frac{c_i}{w_i(u)} \Big)_i
	\]
	we see that the vector $w(u)=(w_i(u))_i$ defines  
	the thermodynamic equilibrium $w=(w_i)_{i=1}^I$ of the concentrations as a
	function of the internal energy $u$.  We generally
	assume that $w_i\in C([0,\infty))\cap C^2((0,\infty))$ are positive,
	non-decreasing and concave. Moreover, the $C^2$ function
	$\hat\sigma(s):=\sigma(s)-\sum_{i=1}^Iw_i(s)+I$ is supposed to be strictly
	concave and increasing. These properties ensure that $S$ is
	strictly concave and that $u\mapsto S(c,u)$ is increasing (see the proof
	of~\cite[Prop.\,2.1]{MielkeMittnenzweig_2018_convergence}). The
	temperature $\theta$ can then be recovered via $\theta=\frac{1}{\partial_uS(c,u)}$ and is per se non-negative.
	
	The diffusive part $\mathbb{K}_{\mathrm{diff}}(Z)$ with $Z=(c,u)$ of the
	Onsager operator is assumed to be of the form
	\begin{align}
	\mathbb{K}_{\mathrm{diff}}(Z)W:=-\divv\big( \mathbb{M}(Z)\nabla W\big),
	\quad W(x)\in \mathbb{R}^{I+1},
	\end{align}
	for a symmetric and positive semi-definite matrix
	$\mathbb M(Z) \in \mathbb R^{(I+1) \times (I+1)}$, the so-called mobility
	matrix, satisfying suitable additional conditions. Here and below,
	$\nabla=\nabla_x$ denotes the gradient with respect to $x\in\Omega$.  The
	operator  $\mathbb{K}_\mathrm{diff}$ will be complemented with the no-flux
	boundary condition $(\mathbb{M}(Z)\nabla W)\cdot\nu=0$ on $\partial\Omega$,
	where $\nu$ denotes the outer unit normal to $\partial\Omega$.  The precise
	structure of the reactive part
	$$
	\mathbb{K}_{\mathrm{react}}(Z)=\begin{pmatrix}
	\mathbb{L}(Z) & 0 \\0&0
	\end{pmatrix}\in \mathbb{R}^{(I+1)\times(I+1)}
	$$
	will be less relevant in this work, the main hypothesis being that it leaves
	the energy equation unchanged.  We will therefore directly formulate our
	hypotheses in terms of $R(Z)=(R_i(Z))_{i=1}^I$, where
	$R_i(Z)=(\mathbb{L}(Z)D_cS(Z))_i$ for $i\in\{1,\dots,I\}$. The positive
	semi-definiteness of $\mathbb{K}_{\mathrm{react}}$ then means that
	$D_cS(Z)\cdot R(Z)\ge0$ (see condition~\eqref{eq:109} below).  We refer to
	\cite{Mielke_2011_gradientStructure, Mielke_2013_thermomechanical} for concrete
	choices of $\mathbb{L}$ for realising reactions following mass-action
	kinetics as in \eqref{ReactionRateMassAction}. 
	
	The above choices for $\mathbb{K}_{\mathrm{diff}}$ and
	$\mathbb{K}_{\mathrm{react}}$ of the  Onsager operators encode conservation of
	the total energy $\mathcal{E}(c,u)=\int_\Omega u(x)\,\dd x$ since  $D\mathcal{E}(c,u)\equiv 1$, hence  $\mathbb{K}(Z)D\mathcal{E}(Z)\equiv0$
	and thus, thanks to symmetry, $\tfrac{\dd}{\dd t}\mathcal{E}(Z)\equiv0$ along
	any curve $Z=Z(t)$ satisfying~\eqref{eq:OnsForm}.  For more information on the
	(modelling) background of such ERDS, we refer
	to~\cite{Mielke_2013_thermomechanical, MielkeMittnenzweig_2018_convergence,
		hhmm_2018}.

	Thus, letting
	$A(Z):=-\mathbb{M}(Z) D^2 S(Z)$ and $R^\circ(Z):=(R(Z),0)^T$, the
	\textit{energy-reaction-diffusion system} (ERDS) we consider takes the form
	\begin{subequations}
		\label{eq:system}
		\begin{align}
		\label{eq:I.system}
		\dot Z & = \divv \big(A(Z) \nabla Z\big)+R^\circ(Z), &  & t>0,\ x\in \Omega,
		\\
		\label{eq:I.bc}
		0   & =\nu \cdot A(Z)\nabla Z,                &  & t>0,\ x\in\partial \Omega.
		\end{align}
	\end{subequations}
	This system will be supplemented with suitable initial
	conditions 
	\[
	Z_{|t=0}=Z^0 \ \ \mbox{in} \ \Omega.
	\] 
	Without loss of generality, we assume that the bounded Lipschitz domain
	$\Omega\subset \mathbb{R}^d$ has unit volume, i.e.~$|\Omega| = 1$.

	The analytical results in the
	references~\cite{MielkeMittnenzweig_2018_convergence,hhmm_2018} primarily
	concern entropy-entropy production inequalities and address the question of
	decay rates to equilibrium assuming the existence of suitably regular
	global-in-time solutions.  To the authors' knowledge, the only global existence
	result for an ERDS of the form~\eqref{eq:I.system}--\eqref{eq:I.bc} (with semiconductor-like
	reactions) appears in~\cite[Section~6]{hhmm_2018} for the special choice
	$\mathbb{M}(c,u)=-k (D^2S(c,u))^{-1}$, where $k>0$ is a positive constant, and
	under the assumption of bounded initial data $(c^0,u^0)$ with $\inf u^0>0$.  In
	this case, species' diffusivities are all equal allowing the authors to infer
	global existence from maximum type principles.  
	
	\medskip
	
	\paragraph{\bf Mobility matrix.}
	Let us note that owing to the
	entropic coupling between $c_i$ and $u$, the Hessian of the entropy density is
	non-diagonal and takes the form
	\begin{align}
	D^2 S (c, u) =
	\begin{pmatrix}
	\ddots &                  &        & \vdots           \\
	& -\frac{1}{c_i}   &        & \frac{w'_i}{w_i} \\
	&                  & \ddots & \vdots           \\
	\dots  & \frac{w'_i}{w_i} & \dots  & \partial^2_u S
	\end{pmatrix},
	\end{align}
	where we abbreviated $w_i=w_i(u)$ and
	\begin{equation}
	\label{eq:Suu}
	\partial^2_u S (c,u) = \hat \sigma''(u) + \sum_{i=1}^I c_i
	\tfrac{w''_i(u)}{w_i(u)}-\sum_{i=1}^I c_i
	\left(\tfrac{w'_i(u)}{w_i(u)}\right)^2  < 0,
	\end{equation}
	because of the concavity of $\hat \sigma$ and $w_i$. 
	
	Except for the restrictive case $\mathbb{M}(c, u)=  - k (D^2S(c,u))^{-1}$
	considered in \cite{hhmm_2018},  
	the matrix $  A(c,u)  =-\mathbb{M}(c, u) D^2 S(c, u)$ is not diagonal,
	usually not symmetric, and positive semi-definiteness cannot be expected.
	However, thanks to the formal entropy structure and the concavity of $S$,
	problem~\eqref{eq:system} is still parabolic in a certain sense
	(cf.~Section~\ref{ssec:ev}).  In our approach, control of the solution is obtained
	via entropy estimates, coercivity bounds for the entropy production and an
	$L^2$-energy estimate for the thermal part, where we also exploit the absence
	of source terms in the energy equation  (see page~\pageref{p:strategy} for
	more details on our strategy.)

	Our existence analysis focuses on mobility matrices of the form
	\begin{subequations}
		\label{eq:mobility}
		\begin{align}
		\label{eq:mobilitya}
		&\mathbb M(c, u) := \diag \big( m_1,\dots,m_I, m_{I+1}\big) + \pi_1(c,u)\mu\otimes\mu
		\\
		\label{eqdefmui}
		&\text{with } \mu_{I+1}(c,u)=1 \text{ and } \mu_i(c,u) \coleq c_i
		\,\frac{w_i'(u)}{w_i(u)} \quad \text{for }\{1,\dots, I\}, 
		\end{align}
	\end{subequations}
	where the mobilities $m_i(c,u) \geq 0$, $i\in\{1,\dots,I\}$,  and the coupling 
	coefficient  $\pi_1(c,u)\ge 0$ are specified below,  whereas
	$m_{I+1}\equiv0$ throughout (but see Remark~\ref{rem:CrossEnergyEqn}).  
	This choice of $\mathbb{M}$ leads to equations of the form~\eqref{ReactionDiffusionTemperature} (see also~\eqref{eq:defb*}).
	
	Loosely speaking, this ansatz for $\mathbb{M}$ can be seen as a
	thermodynamically consistent generalisation of the above-mentioned choice
	$\mathbb{M}(c,u)=-k (D^2S(c,u))^{-1}$ to allow for species-dependent
	diffusivities.  In fact, when choosing $m_i \coleq c_i$ and
	$\pi_1(c,u) \coleq 1/\gamma(c,u)$ with
	\begin{equation}\label{eq:def.gamma}
	\gamma(c,u) \coleq  -\partial^2_u S(c,u) - \sum_{i=1}^I c_i \big(\tfrac{w'_i(u)}{w_i(u)}\big)^2  
	= -\hat\sigma''(u)-\sum_{i=1}^Ic_i\frac{w_i''(u)}{w_i(u)},
	\end{equation} 
	then $\mathbb{M}(c,u)$ becomes the inverse Hessian of $S(c, u)$
	(see~\cite[p.~777]{MielkeMittnenzweig_2018_convergence}). We would like to emphasize
	that a diagonal diffusion matrix
	$A=\diag(\dots)\in \mathbb{R}^{(I+1)\times(I+1)}$ cannot be produced in a
	thermodynamically consistent setting unless its entries on the diagonal are all
	equal, since otherwise $\mathbb{M}$ lacks symmetry.
	
	In this paper, $m_i(c,u)$ is assumed to take the form
	\begin{enumerate}[label=\normalfont{(M\arabic*)}]
		\item\label{it:mi1} $m_i(c,u) = c_ia_i(c,u)$ for $i\in\{1,\dots, I\}$, $m_{I+1}\equiv0$, where
		the functions $a_i\in C([0,\infty)^{I+1})$ satisfy for certain constants
		$\kappa_{1,i},\kappa_{0,i}\ge0$
		\[	      
		a_i(c,u)\sim \kappa_{0,i} + \kappa_{1,i}c_i \quad 
		\text{for }  (c,u)\in[0,\infty)^{I+1}.
		\]
	\end{enumerate}
	(See~\emph{Notations} on page~\pageref{p:notations} for the meaning of the
	symbol~$\sim$ .)  Here, $\kappa_{0,i}$ is typically called the diffusion
	coefficient of the  species $\mathcal{A}_i$,  while the $\kappa_{1,i}$ are sometimes
	referred to as self-diffusion coefficients. The motivation for this choice of
	$a_i(c,u)$ comes from the observation that, to some extent, models with
	positive coefficients $\kappa_{1,i},\kappa_{0,i}>0$ may be seen as a
	regularisation of the model where $\kappa_{1,i}=0$ and $\kappa_{0,i}>0$ for all
	$i$ since then entropy estimates may allow to control not only
	$\nabla \sqrt{c_i}$ but also $\nabla c_i$ in $L^2$, which via Sobolev
	embeddings leads to better integrability properties of $c_i$. One should caveat
	though that, owing to the entropic coupling, the problem is more complex and a
	positive $\kappa_{1,i}$ does not necessarily ensure an $L^2$ a priori bound of
	$\nabla c_i$ by entropy production estimates.  In fact, for the special choice
	$\pi_1\sim  {1}/{\gamma}$ such a regularisation is neither needed nor
	helpful, and we will construct global-in-time weak solutions provided that
	$\kappa_{1,i}=0$ for all $i$.  
	
	However, we are interested in more general choices of coefficients $\pi_1$ in~\eqref{eq:mobilitya} allowing for strong
	cross-diffusion due to energy gradients, as alluded to in the beginning.
	Here, strong cross-diffusion manifests itself in the fact that
	the flux terms generated by off-diagonal entries in the diffusion matrix $A(Z)$
	cannot be controlled in $L^1_{t,x}$ by means of natural entropy estimates associated with the system. Surprisingly (when compared to existing literature), for such problems, we are still able to show an existence result:
	we will first construct global-in-time weak solutions for a family of \enquote{regularised} models with $\kappa_{1,i}>0$ for all $i\in\{1,\dots, I\}$. These approximate solutions will then enable an existence analysis in the case of strong cross-diffusion (arising due to vanishing self-diffusion) on the basis of the concept of renormalised solutions, as carried out in Sections~\ref{sec:prelim2}--\ref{sec:renormalised}.
	
	Observe that, by the choice of $\mu_i$ in
	\eqref{eqdefmui}, cross-diffusion between $c_i$ and $c_j,$ $i\not=j,$ does not
	arise in our models.  More precisely, our definition
	$A(c,u) := -\mathbb M(c,u) D^2 S(c,u)$ and the fact that
	$-\frac{1}{c_i}\mu_i+\frac{w_i'}{w_i}\mu_{I+1}\equiv0$ for $i\in\{1,\dots,I\}$ imply
	that the submatrix $(A_{ij})^{1\le i\le I+1}_{1\le j\le I}$ only depends on the
	diagonal part $\diag(m_i)$ of $\mathbb{M}$.  
	We would like to note that in the
	self-diffusive case with $\kappa_{1,i}>0$ for all $i$, our methods extend to
	situations modelling cross-diffusion between species, see
	Remark~\ref{rem:CrossEffect}.
	
	For the above choice of~$\mathbb{M}$, a direct computation of $A$ yields 
	the coefficients 
	\begin{equation}\label{eq:defb*}
	\begin{aligned}
	A_{ii}(c,u) &= \tfrac{m_i(c,u)}{c_i} =: a_i(c,u) \qquad\text{for }i=1,\ldots,I,
	\\
	A_{i,j}(c,u)&=0 \qquad\text{for } i\neq j \text{ and } i,j\in
	\{1,\ldots,I\},
	\\  
	A_{i,I+1}(c,u)&=
	-m_i\tfrac{w_i'}{w_i}-\pi_1c_i\tfrac{w_i'}{w_i}\Big(\hat\sigma''(u)+\sum_{j=1}^Ic_j\tfrac{w_j''(u)}{w_j(u)}\Big)
	\\
	&=\big({-}a_i(c,u)+\pi_1(c,u)\gamma(c,u)\big)c_i\,\tfrac{w_i'(u)}{w_i(u)}
	\qquad\text{for }i=1,\ldots,I,
	\\
	A_{I+1,I+1}(c,u) & = -\pi_1\bigg(\hat\sigma''(u)+ 
	\sum_{i=1}^Ic_i\tfrac{w_i''(u)}{w_i(u)}\bigg)=\pi_1(c,u)\gamma(c,u)=:a(c,u).
	\end{aligned}
	\end{equation}
	In particular,
	$\sum _jA_{ij}(Z)\nabla Z_j= a_i(c, u) \nabla c_i + A_{i,I+1}(c,u)\nabla u $ if
	$i\not= I+1$, and $A_{I+1,j}(Z)=a(c,u)\delta_{I+1,j}$.
	
	\medskip
	
	\paragraph{\bf Reactions}
	For reaction-diffusion equations with physically realistic reaction rates, the
	available global existence results in the literature often rely on renormalised
	solutions. To illustrate the underlying reason, consider for example a single
	reversible chemical reaction of the form 
	\begin{align*}
	\alpha_1 \mathcal{A}_1+\ldots+  \alpha_I \mathcal{A}_I  \rightleftharpoons
	\beta_1 \mathcal{A}_1+\ldots+   \beta_I \mathcal{A}_I.   
	\end{align*}
	The reaction rates according to mass-action kinetics are then given by
	\begin{align}
	\label{ReactionRateMassAction}
	R_i(c,u) =  \kappa(c,u)\:\bigg(  \prod_{j=1}^I 
	\Big(\frac{c_j}{w_j(u)}\Big)^{\alpha_j}  
	- \prod_{k=1}^I \Big(\frac{c_k}{w_k(u)}\Big)^{\beta_k}  \bigg)
	\: (\beta_i-\alpha_i) , 
	\end{align}
	for some non-negative reaction coefficient $\kappa(c,u)\geq 0$.  At
	the same time, the only known energy estimate for the basic reaction-diffusion
	system \eqref{ReactionDiffusionNoCrossDiffusion} is the entropy estimate; it
	merely provides control of quantities of the form
	$\sup_t \int_\Omega c_i \log c_i \,\dd x$ or
	$\smash{\int_0^T \int_\Omega |\nabla \sqrt{c_i}|^2 \,\dd x\dd t}$.  Thus,
	without unphysically strong assumptions on 
	the stoichiometric coefficients $\alpha_j,\beta_k$ (or on $\kappa(c,u)$),
	the available energy estimates are not sufficient to ensure
	$L^1([0,T]\times\Omega)$ integrability of the reaction terms
	\eqref{ReactionRateMassAction}, which would be required for standard weak
	solutions concepts.
	
	Nevertheless, reactions of the type \eqref{ReactionRateMassAction} satisfy the
	entropy inequality 
	\begin{align*}
	D_cS(c,u)\cdot R(c,u)&= \kappa(c,u)\:\big( c_w^\alpha - c_w^\beta  \big)
	\: \sum_{i=1}^I
	(\beta_i{-}\alpha_i)\log\big(\frac{w_i(u)}{c_i} \big) 
	\\
	&= \kappa(c,u) \big( c_w^\alpha - c_w^\beta  \big)\big(\log c_w^\alpha  - \log
	c_w^\beta \big) \geq 0,
	\end{align*}
	where $c_w^\alpha = \prod_{j=1}^I \big(\frac{c_j}{w_j(u)}\big)^{\alpha_j} $ and
	similarly for $c_w^\beta$. Thus, it is reasonable to impose the 
	\emph{entropy inequality} \leqnomode
	\begin{align}\label{eq:109}\tag{R1}
	D_cS(c,u)\cdot R(c,u)\ge0\quad \text{ for all }(c,u)\in(0,\infty)^{I+1}.
	\end{align}
	\reqnomode 
	This condition together with $\mathbb{M}=\mathbb{M}^T \geq 0$ ensure
	that, formally, the entropy functional $\mathcal{S}(c,u)$ 
	is non-decreasing along trajectories of system~\eqref{eq:system}.  We note
	that, since $\partial_{c_i}S(c,u)=-\log\big(\frac{c_i}{w_i(u)}\big)$, the
	condition $[c_i=0\implies R_i\ge0]$, which is  necessary for the positivity of $c_i$, is
	implicitly contained in hypothesis~\eqref{eq:109}.  For more background,
	applications and specific examples for admissible choices of $R(c,u)$ we refer
	to~\cite[Sec.\,2.3]{MielkeMittnenzweig_2018_convergence} and \cite[Sec.\,3.2]{hhmm_2018}
	and references therein.

	\medskip
	
	\subsection{Main results}
	
	\begin{genAssumption}\hypertarget{ga}{}
		In all of our results, the entropy function $S_0$ is assumed to have the form
		\begin{align}\label{eq:S0}
		S_0(c, u) = \hat \sigma_0(u) + \sum_{i = 1}^I \Big(c_i \log w_i(u) - \lambda(c_i) \Big),
		\end{align}
		where $\lambda$ is as in~\eqref{eq:deflb}, and $\mathbb{M}(c,u)$ is supposed to
		be given by~\eqref{eq:mobilitya}--\eqref{eqdefmui} with~\ref{it:mi1} being satisfied.  Regarding
		the coefficient functions $\hat\sigma_0,w_i,\pi_1$, the following basic
		regularity and qualitative properties will be supposed:
		
		\begin{enumerate}[label=\normalfont{(B\arabic*)}]
			\item\label{it:sigma} Suppose that $\hat\sigma_0 \in C^2((0,\infty))$ is a strictly
			concave, increasing function.
			\item\label{it:w} Let $w_i\in C([0,\infty)) \cap C^2((0,\infty))$ with
			$w_i(0)>0$ for $i\in\{1,\dots, I\}$ be concave and non-decreasing functions.
			\item\label{it:pi} Assume that $\pi_1\in C([0,\infty)^{I+1},[0, \infty))$ with $\sqrt{\pi_1(u,c)}w_i'(u)\in C([0,\infty)^{I+1})$ for all $i$.
			\item\label{it:sig1}\label{it:H1}
			$\lim_{u\downarrow0}\hat\sigma_0'(u)=+\infty$,
			$\lim_{u\uparrow\infty}\hat\sigma_0'(u)=0$
			\item\label{it:6}\label{it:H3} $\exists\beta\in(0,1)$ such that
			$w_i(u)\lesssim  (1{+}u)^\beta  $ for all $i\in\{1,\dots,I\}$
			\item\label{it:4}\label{it:H6} $0\le\pi_1^\frac{1}{2}(c,u)\lesssim 1+u$.
		\end{enumerate}
		Finally, we generally assume the reaction rates
		$R\in C([0,\infty)^{I+1},\mathbb{R}^I)$ to be continuous and to satisfy the
		entropic production estimate~\eqref{eq:109}.
	\end{genAssumption}
	The above collection of assumptions will be referred to below as General Hypotheses.
	
	\medskip

	\paragraph{\bf The two models (H) and (H')}
	In our analysis, we consider two `models', which mainly differ in the choice of the coefficient function $\pi_1$ of the rank-one part of $\mathbb{M}$.
	The first model is determined by the collection of hypotheses~\hypertarget{H}{\textbf{\textup{(H)}}} consisting of the following three conditions~\ref{it:2}--\ref{it:1}
	\begin{enumerate}[label=\normalfont{(H\arabic*)}]
		\item\label{it:2}\label{it:H2} $\pi_1\gamma\gtrsim 1$ \qquad 
		(with $\gamma$ given by~\eqref{eq:def.gamma} for $S=S_0$, i.e.~$\gamma=-\hat\sigma_0''(u)-\sum_{i=1}^Ic_i\tfrac{w_i''(u)}{w_i(u)}>0$)
		\item\label{it:3}\label{it:H4} $w_i'(u)\lesssim-w_i''(u)\pi_1^\frac{1}{2}(c,u)$
		\item\label{it:1}\label{it:H5} $\pi_1^\frac{1}{2}(c,u)\tfrac{w_i'}{w_i}\lesssim 1$.
	\end{enumerate}
	The second model is determined by hypotheses~\hypertarget{H'}{\textbf{\textup{(H')}}} consisting of the following two conditions
	\begin{enumerate}[label=\normalfont{(H\arabic*')}]
		\item\label{it:pi.m3} $\pi_1\gamma\sim1$
		\item \label{it:w.m3} $(w_i')^2\lesssim -w_i''w_i$. 
	\end{enumerate}
	\medskip
	
	Let us briefly comment on the hypotheses concerning $\hat\sigma_0,w_i$ and $\pi_1$ formulated in the last two paragraphs.
	The assumptions on the concave functions $\hat\sigma_0$ and $w_i$ are not very restrictive and allow to include essentially all examples typically used in the modelling
	such as $\hat\sigma_0(s)=b\log(s)$ or
	$\hat\sigma_0(s)=bs^\alpha$ for $\alpha\in(0,1)$, where $b>0$ is a positive constant.
	Typical admissible choices of $w_i$ are $w_i(u)=b_{0,i}+b_{1,i}u^{\beta_i}$ or  $w_i(u)=b_{0,i}(1+b_{1,i}u)^{\beta_i}$ for $\beta_i\in(0,1)$ and $b_{0,i}>0, b_{1,i}\ge0$, in which case one can choose $\pi_1^\frac{1}{2}\sim u$ and $\pi_1^\frac{1}{2}\sim1+u$ respectively, when assuming~\HH. When considering~\HHp\ instead, we should note the compatibility of hypotheses~\ref{it:H6} and~\ref{it:pi.m3} for any power-law ansatz of $\hat\sigma_0$. Indeed, we then have $\gamma(c,u)\ge-\hat\sigma_0''(u)\gtrsim u^{\alpha-2}$ and hence $\pi^\frac{1}{2}\lesssim u+1$ whenever $\alpha\in[0,1)$ (with $\alpha=0$ corresponding to $\hat\sigma_0(u)=\log(u)$).
	Hypotheses~\ref{it:3} of~\HH\ and~\ref{it:w.m3} of~\HHp\ can be regarded as concavity conditions on the equilibria $w_i$ and rule out, for instance, that $w_i(u)=b_1u+b_0$ for $b_1>0,b_0\ge0$ for some $i\in\{1,\dots,I\}$. In view of~\ref{it:w} and~\ref{it:H6}, Hypothesis~\ref{it:1} can be shown to be always fulfilled if $u\ge 1$, and should thus be understood as a condition for small arguments $u$ close to zero.
	
	\medskip
	
	Our first main result assumes the following hypotheses. 
	
	\begin{assumption}
		\label{mainhypo1}
		
		Let the~\ga\ be satisfied and assume that either hypotheses~\HH\ or
		hypotheses~\HHp\ are fulfilled.
		
		In case~\HH, we additionally assume that $\kappa_{1,i}>0$ for all
		$i\in\{1,\dots, I\}$ and that there exists $0\le q_1<\ol q_1:=2+\tfrac{2}{d}$ and
		$0\le q_2<\ol q_2:=2+\tfrac{4}{d}$ such that $|R(c,u)|\lesssim 1+|c|^{q_1}+|u|^{q_2}.$
		
		Under hypotheses~\HHp, we assume that $\kappa_{1,i}=0$, $\kappa_{0,i}>0$ for
		all $i\in\{1,\dots, I\}$, and that there exists  $0\le q_1<\tilde q_1:=1+\tfrac{2}{d}$ and
		$0\le q_2<\tilde q_2:=2+\tfrac{4}{d}$ such that $|R(c,u)|\lesssim 1+|c|^{q_1}+|u|^{q_2}$.
	\end{assumption}
	
	The proof of the following main result will be completed in Section
	\ref{su:ProMainThm}.

	\begin{theorem}[Global existence of weak solutions]
		\label{thm:main}
		Let Hypotheses~\ref{mainhypo1} hold true.  Let $Z^0=(c^0, u^0)$ have
		non-negative components satisfying $c^0_i\in L\log L(\Omega)$,
		$i\in\{1,\dots,I\}$, $u^0\in L^2(\Omega)$ and
		$\hat\sigma_{0,-}(u^0)\in L^1(\Omega)$, where $\hat\sigma_{0,-}$ denotes the negative part of $\hat\sigma_{0}$.  Then there exist non-negative
		functions
		\begin{align}
		\label{eq:Th.est.c}
		c_i &\in L^\infty(0,\infty;L\log L(\Omega)),\quad  i\in\{1,\dots,I\}, \\
		\label{eq:Th.est.u}
		u &\in L^\infty(0,\infty;L^2(\Omega))
		\end{align}
		such that $Z:=(c_1,\dots,c_I,u)$ has the regularity
		\begin{align}
		\label{eq:Th.est.flux}
		&\mathbb{M}(Z)D^2S_0(Z)\nabla Z\in
		L^s_{\mathrm{loc}}([0,\infty) {\times} \ol\Omega)^{d(I+1)}, 
		\text{ where }  s=\tfrac{2d+2}{2d+1}, 
		\\
		\label{eq:Th.est.Zdot}
		&\partial_tZ_i\in L^r_\mathrm{loc}(0,\infty; W^{1,r'}(\Omega)^*) \;\;\text{for suitable }r>1, \;\;\tfrac{1}{r'}+\tfrac{1}{r}=1 \\&\hspace{3cm}(\text{in case }\HH{} \text{\,one may choose }  \text{$r:=\min\big\{\tfrac{2d+2}{dq_1},\tfrac{2d+4}{dq_2},s\big\}$}),
		\end{align}
		and satisfies for all $T>0$ and all
		$\phi=(\phi',\phi_{I+1})\in L^\infty(0,T;W^{1,\infty}(\Omega))^{I+1}$ the
		equation
		\begin{equation}\label{eq:00}
		\begin{split}
		\int_0^T\langle\partial_tZ,\phi\rangle\,\dd t
		-\int_0^T\!\!\int_{\Omega}(\mathbb{M}(Z)D^2S_0(Z)\nabla Z):&\nabla
		\phi\,\dd x\dd t \\&=\int_0^T\!\!\int_{\Omega}R(Z)\cdot\phi'\,\dd x\dd t
		\end{split}
		\end{equation}
		and the identity $Z(t=0,\cdot)=(c^0,u^0)$ as an equality in
		$(W^{1,\infty}(\Omega)^*)^{I+1}$.
		\smallskip
		
		Furthermore, the internal energy is conserved, i.e.~for all $t>0$,
		\begin{align}
		\int_\Omega u(t,x)\,\dd x = 	\int_\Omega u^0(x)\,\dd x,
		\end{align}
		and the solution satisfies the bound 
		\begin{equation}
		\label{eq:mainbd}
		\begin{split}
		\|c\|_{L^\infty(L\log L)}&+\|u\|_{L^\infty
			L^2}+\|\hat\sigma_{0,-}(u)\|_{L^\infty L^1}
		\\&+\int_0^\infty\!\!\int_\Om\pi_1\gamma|\nabla u|^2\,\dd x\dd
		t+\int_0^\infty\!\!\int_\Om P(c,u)\,\dd x\dd t \\&\qquad\qquad\qquad\le
		C(\|c^0\|_{L\log L}, \|u^0\|_{L^2}, \|\hat\sigma_{0,-}(u^0)\|_{L^1}),
		\end{split}
		\end{equation}
		where 
		\begin{equation}\label{eq:1000}
		\begin{split}
		P(c,u)&:=\sum_{i=1}^I\Big(\kappa_{1,i}|\nabla c_i|^2
		+\kappa_{0,i}|\nabla \sqrt{c_i}|^2\Big)
		+\pi_1\gamma^2|\nabla u|^2.
		\end{split}
		\end{equation}
		(Observe that the last term in~\eqref{eq:1000} differs from the
		integrand of the first term in the second line of~\eqref{eq:mainbd} by a factor of $\gamma$.)
	\end{theorem}
	
	We note that, by density, the equation~\eqref{eq:00} holds true for a somewhat
	larger set of test functions $\phi$, and in case~\HHp\ the regularity of
	$\partial_tZ$ can be slightly improved. We should further note that in this work we have not aimed at optimising the regularity of the initial energy density $u_0$. The choice~$u_0\in L^2$ has been made for simplicity.
	
	At this stage we can comment on the major interplay between the choice of
	the entropy density $S$ in \eqref{eq:S} and the mobility matrix $\mathbb M$ in
	\eqref{eq:mobility}. At a formal level, we obtain along solutions the 
	conservation of energy and the
	entropy entropy-production balance:
	\begin{align*}
	&\mathcal{E}(c(t),u(t)) = \int_\Omega u(t,x)\, \dd x = \int_\Omega u^0(x)\,\dd x =
	\mathcal{E}(c^0,u^0), 
	\\
	&\mathcal{S}(c(t),u(t)) =\mathcal{S}(c^0,u^0) + \int_0^t \!\!\big( 
	\mathcal{P}_{\mathrm{diff}}(c(r),u(r))
	+\mathcal{P}_{\mathrm{react}}(c(r),u(r))\big) \dd r 
	\end{align*} with
	$\mathcal{P}_{\mathrm{diff}}(Z)=\int_\Omega \!\!\nabla Z :D^2S(Z)\mathbb{M}(Z) D^2S(Z)\nabla Z \,\dd x   \text{ and }   \mathcal{P}_{\mathrm{react}}(c,u)
	=\int_\Omega \!\! D_c S(c,u){\cdot} R(c,u) \,\dd x.$
	\label{page:Pdiff}
	Using $\mathcal{P}_{\mathrm{diff}}(Z)\geq 0$ and
	$\mathcal{P}_{\mathrm{react}}(Z)\geq 0$ we conclude
	$\mathcal{S}(Z(t))\geq \mathcal{S}(Z^0)$. Combining the trivial bound of $u(t)$
	in $L^1(\Omega)$ with the bounds \ref{it:H3} on $w_i$, this implies a uniform
	a priori bound for $\int_\Omega \lambda(c_i(t))\,\dd x $, see Lemma \ref{l:S}.
	However, the main difficulty in justifying eq.~\eqref{eq:00} is to show that the flux
	$A(Z)\nabla Z$ lies in $L^1_{\mathrm{loc}}([0,\infty){\times}\ol\Omega)$, based
	on the fact that the entropy production
	$\int_0^\infty \mathcal{P}_{\mathrm{diff}} (Z(r))\dd r$ is finite.  Because of
	$A(Z)=-\mathbb{M}(Z)D^2 S(Z)$, this means that a bound of the form 
	\begin{equation}
	\label{eq:1.Bound}
	|A(Z) \nabla Z|^{\widetilde s} \leq C\Big( 1 + |Z| + |u|^\nu + \nabla Z: D^2S(Z)\mathbb{M}(Z) D^2
	S(Z)\nabla Z\Big) 
	\end{equation}
	would be desirable for some $\widetilde s \geq 1$. The trivial case would be
	$\nu=1$; however, using the special structure
	\eqref{ReactionDiffusionTemperatureb} allows us to derive simple a priori bounds
	in $L^\nu(\Omega)$, $\nu>1$, as well. We refer to the estimates \eqref{eq:402.3} and
	\eqref{eq:402} in Lemma \ref{l:efl}.  
	At the end, the situation is somewhat more involved
	and we can exploit Galiardo--Nirenberg estimates as well, which will lead us
	to the dimension-dependent exponent $s= 1 + 1/(2d{+}1)$ in
	\eqref{eq:Th.est.flux}.

	\begin{remark}[Cross-diffusion between species]
		\label{rem:CrossEffect}
		In case~\HH, where $\kappa_{1,i}>0$ for all~$i$, Theorem~\ref{thm:main} can be
		extended to a situation where cross-diffusion between species does
		occur. Indeed, in this case the coercivity estimate for the entropy production
		in Lemma~\ref{l:EP} remains valid when replacing the Onsager matrix
		$\mathbb{M}$ by $\mathbb{M}+\tilde\mu\otimes\tilde\mu$, where the continuous function
		$\tilde\mu=(\tilde\mu_1,\dots,\tilde\mu_I,0)$ is supposed to satisfy
		$|\tilde\mu_i(c,u)|\lesssim c_i$ for all $i$.  In this case,
		$A_{ij}=\delta_{ij}\tfrac{m_i}{c_i}+\tilde\mu_i\tfrac{\tilde\mu_j}{c_j},
		i,j\in\{1,\dots,I\}$, implying that $|A_{ij}|\lesssim c_i$.  The additional
		thermodiffusion-type coefficient in front of $\nabla u$ in the $i$-th component
		is given by $-\tilde\mu_i\sum_{j=1}^I\tilde\mu_j\tfrac{w_j'(u)}{w_j(u)}$.  One
		can now see that, under the assumption $\kappa_{1,i}\gtrsim 1$, a flux bound of
		the form~\eqref{eq:402} can still be guaranteed.
	\end{remark}
	
	\begin{remark}[Cross terms in the energy equation]
		\label{rem:CrossEnergyEqn}
		It is possible in Theorem~\ref{thm:main} to allow for non-trivial, non-negative
		continuous coefficients $m_{I+1}(Z)$ in the definition of $\mathbb{M}$ (cf.~\eqref{eq:mobility})
		satisfying suitable growth
		conditions. Note that while non-trivial coefficients $m_{I+1}$ may alter heat
		conductivity, the Dufour coefficients remain unchanged since
		$\partial_uS(Z)=1/\theta$.  For general $m_{I+1}$, the energy equation
		in the $Z$-variables takes the form
		\begin{align*}
		\dot u & =\divv\bigg(\big(a(Z)-m_{I+1}\partial_u^2S(Z)\big)\nabla u+\sum_{j=1}^Id_j(Z)\nabla c_j\bigg),
		\end{align*}
		where $d_j(Z)=-m_{I+1}\tfrac{w_j'}{w_j}$. The non-negativity of $m_{I+1}$
		ensures that the coercivity bounds for the entropy production are preserved.
		Besides the entropy production estimate, the main a priori estimate for $u$ is
		the energy estimate~\eqref{eq:121} in Lemma~\ref{l:energy}.  As will become
		clear in the proof of Theorem~\ref{thm:main}, our analysis can deal with
		generalisations of estimate~\eqref{eq:121} of the form
		\begin{align}\label{eq:221}
		\epsilon_1\tau\int_\Omega a_\delta(Z)|\nabla u|^2\dd x+\int_\Omega u^2\,\dd x
		\leq \int_\Omega (u^{k-1})^2\,\dd x +C \tau\int_\Om P(Z)\,\dd x,
		\end{align}	
		provided $\epsilon_1>0.$      
		For instance, under the assumptions of Theorem~\ref{thm:main}, in case~\HH\ an
		a priori estimate of the form~\eqref{eq:221} can be obtained if
		$m_{I+1}(Z)\ge0$ satisfies the bound
		$\tfrac{w_j'}{w_j}\:m_{I+1}\lesssim \sqrt{\pi_1}\gamma+\sqrt{\pi_1\gamma}.$
		For the case~\HHp\ the corresponding inequality~\eqref{eq:221} can be obtained
		if $m_{I+1}\lesssim 1$. 
		(Some additional conditions may have to be imposed to make the full
		construction work.)
	\end{remark}
	
	\begin{remark}
		\label{rem:thm}
		The hypothesis $\kappa_{1,i}>0$ for all $i\in\{1,\dots,I\}$ in case~\HH\ is
		essential in our proof of Theorem~\ref{thm:main}.  If $\kappa_{1,i}=0$, entropy
		production bounds and Sobolev type estimates do not generally provide
		$L^2$-integrability of the concentrations, which we need in case~\HH\ in order
		to ensure at least $L^1$-integrability of the flux term associated with
		thermodiffusion (see Lemmas~\ref{l:EP} and~\ref{l:efl}).
	\end{remark}
	
	Our second result is motivated by the question of global existence in the
	absence of self-diffusion, i.e.\ in the case when $\kappa_{1,i}=0$ for all
	$i\in\{1,\dots,I\}$.
	In the setting of~\HH, choosing $\kappa_{1,i}=0$ leads to strong cross-diffusion effects and is not covered by Theorem~\ref{thm:main}.
	Here, entropy (and energy) estimates in general fail to ensure integrability of the thermodiffusive flux terms. We therefore use a weaker concept of solution similar to the notion of renormalised solutions introduced in~\cite{Fischer_2015_renormalized}, which allows us at the same time to drop the growth condition on the reaction rates in Theorem~\ref{thm:main}. Weak or no growth restrictions on $|R_i(\cdot)|$ are often desirable when interested in physically realistic reactions.
	The concept of renormalised solutions utilized
	in~\cite{Fischer_2015_renormalized} originates from the studies of DiPerna and
	Lions on the global existence of solutions to Boltzmann and transport equations
	\cite{DiPernaLions_1988_FokkerPlanckBoltzmann,
		DiPernaLions_1989_CauchyProblemBoltzmann,
		DiPernaLions_1989_ODETransportSobolevSpaces}. During the last decades,
	various notions of renormalised solutions have been employed in the literature;
	see e.g.\ \cite{DalMaso_etal_1999_RenormalizedSolutionsMeasureData,
		Desvillettes_etal_2007_GlobalExistQuadrSystems,
		Villani_1996_CauchyProblemLandauEquation}. In
	\cite{Desvillettes_etal_2007_GlobalExistQuadrSystems} the authors present some
	classes of mass-action kinetics models which admit global weak solutions for
	reaction rates with at most quadratic growth and which allow for global
	renormalised solutions with defect measure for at most quartic growth.
	
	As pointed out in Remark \ref{rem:thm}, we are generally lacking an $L^2$ a priori bound on $c_i$ in the case $\kappa_{1,i} = 0$. 
	For certain classes of models, $L^2$ bounds for the species can be obtained using duality estimates.
	See, e.g., the references~\cite{Pierre_2010_GlobalExistenceSurvey,DLM_2014_entropy,CanizoDesvillettesFellner_2014_ImprovedDualityEstimates,DesvillettesFellner_2015_DualityMethodsDegenerateDiffusion,LM_2017_entropic}, which include cross-diffusive models.
	However, without very specific assumptions on the structure of the diffusion operator such duality arguments are not applicable in our setting.
	
	Our definition of renormalised solutions adapts~\cite{Fischer_2015_renormalized,ChenJuengel_2019_renormalised}.
	
	\begin{definition}[Renormalised solution]
		\label{def:RenormSol}
		Let $J=(0,\infty)$ and let $Z_0=(c_0, u_0):\Om\to\mathbb{R}_{\ge0}^{I+1}$ be
		measurable. We call a function $Z=(c,u)$ with non-negative components a
		(global) renormalised solution of \eqref{eq:system} with initial data $Z^0$ if
		$Z_i\in L^2_\loc(\ol J; H^1(\Omega))$ or
		$\sqrt{Z_i}\in L^2_\loc(\ol J; H^1(\Omega))$ for each $i\in\{1,\dots, I+1\}$,
		if furthermore
		\begin{equation}\label{eq:2.100}
		\begin{split}
		\chi_{\{|Z|\le E\}}A(Z)\nabla Z\in L^2_\loc(\ol J; L^2(\Omega))
		\end{split}
		\end{equation}
		for all $E\ge 1$, and if for every $\xi \in C^\infty([0,\infty)^{I+1})$ with
		compactly supported derivative $D\xi$, every $T>0$, and every
		$\psi \in C^\infty([0, T] \times \ol \Omega)$ with $\psi(\cdot, T) = 0$ the
		following identity is satisfied:
		\begin{align}\label{eq:def.renorm}
		- \int_\Omega & \xi(Z^0) \psi(\cdot, 0) \,\dd x -
		\int_0^T\!\!\int_\Omega\!\!\xi(Z) \frac{d}{dt} \psi \, \dd x\dd t
		\\
		&=  - \sum_{i,j,k=1}^{I+1} \int_0^T\!\!\int_\Omega \psi \, \partial_i
		\partial_k \xi (Z) A_{ij}(Z)\nabla Z_j \cdot \nabla Z_k \, \dd x\dd t 
		\\ 
		& \quad -  \sum_{i,j=1}^{I+1} \int_0^T\!\!\int_\Omega   \partial_i \xi
		(Z)
		A_{ij}(Z)\nabla   Z_j \cdot   \nabla \psi \,   \dd x\dd t
		+ \sum_{i=1}^I \int_0^T\!\!\int_\Omega \psi \, \partial_i \xi (Z) R_i(Z) \, \dd x\dd t.
		\end{align}
	\end{definition}
	
	Our second main result assumes the following conditions.
	
	\begin{assumption}
		\label{mainhypo2}
		Let the~\ga\ be satisfied and assume hypotheses~\HH. Further suppose that
		$\kappa_{1,i}=0$ and $\kappa_{0,i}>0$ for all $i\in\{1,\dots, I\}$.
	\end{assumption}
	Note that Hypotheses~\ref{mainhypo2} do not impose any growth conditions on
	$R(c,u)$.  Let us further point out that while the restriction to the (arguably
	more interesting) case~$\kappa_{1,i}=0$ for all $i$ is not necessary and we
	could have equally treated Onsager matrices of the form considered in
	Theorem~\ref{thm:main}, our second result appears to require that either
	$\kappa_{1,i}>0$ for all~$i$ or $\kappa_{1,i}=0$ for all $i$ if one wants to
	admit reactions with arbitrarily fast growth.

	\begin{theorem}[Global existence of renormalised solutions]
		\label{theoremexistence}
		
		Let Hypotheses \ref{mainhypo2} hold true. Let $Z^0=(c^0, u^0)$ have
		non-negative components satisfying $c^0_i\in L\log L(\Omega)$,
		$i\in\{1,\dots,I\}$, $u^0\in L^2(\Omega)$ and
		$\hat\sigma_{0,-}(u^0)\in L^1(\Omega)$.  Then, there exists a global
		renormalised solution $Z=(c, u)$ to \eqref{eq:system} with initial data $Z^0$
		having the additional regularity
		\begin{align*}
		c_i &\in L^\infty(0, \infty; L \log L (\Omega)), \quad i \in \{1, \dotsc, I\}, \\
		u &\in L^\infty(0, \infty; L^2(\Omega)).
		\end{align*}
		For all $T > 0$, $u$ satisfies
		$\partial_t u \in L^s(0, T; W^{1,s'}(\Omega)^\ast)$ (with
		$s = \tfrac{2d+2}{2d+1}$, $s' = 2d+2$) and
		\[
		\int_0^T \big\langle \partial_t u, \phi \big\rangle \, \dd t + \int_0^T\!\!
		\int_{\Omega} a(c, u) \nabla u \cdot \nabla \phi \,\dd x\dd t = 0
		\]
		for all $\phi \in L^\infty(0, T; W^{1,\infty}(\Omega))$. Moreover, the internal
		energy is conserved, i.e.\ for all $t > 0$,
		\[
		\int_\Omega u(t,x) \, \dd x = \int_\Omega u^0(x) \, \dd x,
		\]
		and the following bounds are valid:
		\begin{align}
		\|c\|_{L^\infty(0,\infty;L\log L)} + \|u\|_{L^\infty(0,\infty;L^2)}
		+\int_0^\infty\!\!\int_\Omega P(c,u)\, \dd x \dd t + \int_0^\infty\!\!\int_\Om \pi_1
		\gamma |\nabla u|^2\,\dd x\dd t \le C(\mathrm{data}),
		\end{align}
		 where $P(c,u)$ is given by~\eqref{eq:1000}
		and $\mathrm{data} = (\|c^0\|_{L\log L}, \|u^0\|_{L^2}, \|\hat\sigma_{0,-}(u^0)\|_{L^1})$.
	\end{theorem}
	
	For results on the global existence and the asymptotic behaviour of
	reaction--diffusion systems coupled to Poisson's equation, which is the typical
	setting for semiconductor models, we refer to
	\cite{Gajewski_Groeger_1996_ReactionDiffusionElectrically,
		GlitzkyGroegerHuenlich_1996_FreeEnergyDissipationRate,
		GlitzkyHuenlich_1997_ElectroReactionDiffusion}.  While uniqueness is in
	general a difficult question for coupled systems, weak-strong uniqueness of
	renormalised solutions to entropy-dissipating reaction-diffusion systems has
	been established in~\cite{Fischer_2017} in the weakly coupled case, and
	in~\cite{CJ_2019_wkstruni} for a class of population models featuring weak
	cross-diffusion.  Within the last couple of years, various kinds of entropy
	estimates have been successfully applied for studying the long-time behavior of
	reaction--diffusion equations and to typically obtain exponential convergence
	to equilibrium
	\cite{ArnoldMarkowichToscani_2000_LargeTimeDriftDiffusionPoisson,
		DesvillettesFellner_2006_ExponentialDecayViaEntropy,
		DesvillettesFellner_2007_EntropyMethods,
		DesvillettesFellner_2008_EntropyMethodsAPrioriBounds,
		DiFrancescoFellnerMarkowich_2008_EntropyDissipationInhomogeneous}.  Finally,
	we point out that exponential convergence to equilibrium has also been shown
	for renormalised solutions in the framework of detailed and complex-balanced
	chemical reaction networks \cite{DesvillettesFellnerTang_2017_ComplexBalance,
		FellnerTang_2017_DetailBalance,
		FellnerTang_2018_ConvergenceOfRenormalizedSolutions}.
	
	\smallskip
	
	\paragraph{\bf Strategy of the proof}
	\label{p:strategy}
	
	The main mathematical difficulty of the problem lies in the fact that the
	equilibria $w_i$ of the concentrations $c_i$ depend on the internal energy
	leading to a strong coupling in the entropy and in the associated evolution
	system.  Let us point out that for strongly coupled systems the design of
	suitable (structure-preserving) approximation schemes may in general be quite
	tricky, even if suitable formal a priori estimates are available.  Our main a
	priori estimates are obtained from the entropy structure of the coupled system
	together with the scalar-like structure of the heat equation (we choose $L^2$ for simplicity). Our approximation
	scheme for Theorem~\ref{thm:main} adapts ideas developed
	in~\cite{Juengel_2015_boundednes-by-entropy,ChenDausJuengel_2018_global} and
	references therein for reaction-diffusion systems with cross-diffusion. A key
	aspect of this scheme lies in a transformation to the so-called entropy
	variables, upon which the semi-definiteness of the mobility matrix $\mathbb{M}$
	is exploited to construct approximate solutions to an elliptically regularised
	problem. In our case, a slight complication arises due to the circumstance that
	regularisations in the entropy variables interfere with the $L^2$-structure of
	the heat equation. We by-pass this issue by suitable additional
	approximation/regularisation procedures.
	
	The proof of Theorem~\ref{theoremexistence} uses the global weak solutions
	obtained in Theorem~\ref{thm:main} as approximate solutions and adapts the
	construction of renormalised solutions to entropy-dissipating
	reaction-diffusion systems in~\cite{Fischer_2015_renormalized,
		ChenJuengel_2019_renormalised}. In these two references, the main difficulty
	comes from the lack of control of the reaction rates, whereas in our situation
	new difficulties arise owing to thermodiffusive cross effects. While the
	article~\cite{ChenJuengel_2019_renormalised} does consider cross-diffusion, it
	strongly relies on the presence of self-diffusion, and for suitably controlled
	reaction rates the existence of global-in-time weak solutions has previously
	been established in~\cite{ChenDausJuengel_2018_global}.  In contrast, in our
	situation a mere regularisation of the reaction rates is insufficient for
	constructing approximate weak solutions, which was one of our motivations for
	establishing Theorem~\ref{thm:main}, case~\HH.
	
	A new aspect of our approximation scheme is the idea of approximating the
	problem with vanishing self-diffusion and uncontrolled flux (treated in
	Theorem~\ref{theoremexistence}) by (thermodynamically consistent) models
	featuring self-diffusion (as considered in case~\HH\ of
	Theorem~\ref{thm:main}).  This construction relies on a second key ingredient:
	a stability result for the thermal part, namely the strong convergence of the
	gradient in $L^2$.  Such stability results for elliptic and parabolic equations
	are classical, and have been used, for instance, in the existence analysis of
	elliptic equations with measure data~\cite{BG_1992, BM_1992, BBGGPV_1995,
		DalMaso_etal_1999_RenormalizedSolutionsMeasureData}.
	
	\smallskip
	
	\paragraph{\bf Outline and notations}
	The remaining part of this paper is devoted to the proofs of
	Theorems~\ref{thm:main} and~\ref{theoremexistence}.  In
	Section~\ref{sec:entropy} we introduce the transformation to the entropy
	variables and a regularised entropy (Sec.~\ref{ssec:ev}), establish formal
	coercivity bounds for the entropy production as well as an estimate on the flux
	term.  In Section~\ref{sec:delta} we use elliptic methods and a fixed point
	theorem to construct a solution to a time-discrete nonlinear regularised
	system, remove the elliptic regularisation (Prop.~\ref{prop:eps=0}) and perform
	an $L^2$ estimate at the level of $u$ (Lemma~\ref{l:energy}). In
	Section~\ref{sec:time} we construct a global weak solution to the
	time-continuous problem.
	
	Some preliminary technical tools for proving Theorem \ref{theoremexistence} are gathered
	in Section \ref{sec:prelim2}. In particular, we provide a weak chain rule for
	the time derivative of truncated solutions. The construction of a global
	renormalised solution is carried out in Section \ref{sec:renormalised} by first
	proving some compactness properties to obtain a limiting candidate.
	This limit is then shown to be a global renormalised
	solution by deriving an approximate equation for appropriately truncated solutions and subsequently passing to the limit of infinite truncation height.
	Some auxiliary results from the literature are recalled in the
	appendix (see page~\pageref{sec:app}).
	
	\smallskip
	
	\small
	\paragraph{Notations} We use the following notations and conventions.
	\label{p:notations}
	\begin{itemize}
		\item For $s\in \mathbb{R}$, we use the convention $s_+=\max\{s,0\}$ and
		$s_-=\min\{s,0\}$ so that $s=s_++s_-$.  Moreover, we let $\hat\sigma_{0,-}$ denote the negative part of $\hat\sigma_{0}$.
		\item In the notation of $L^p$ and Sobolev spaces we usually do not explicitly
		state the underlying domain $\Omega$. For a Banach space $X$, we sometimes
		write $L^p(X)$ to denote the Bochner space $L^p(0,T;X)$.
		\item We abbreviate $\Omega_T:=(0,T)\times\Omega$.
		\item Unless otherwise stated, $\nabla$ denotes the gradient with respect to
		the space variable $x\in \Omega$, while for a general function
		$F(a_i,\dots,a_N)$ (in several variables) we denote by $DF$ its total
		derivative.
		\item The notation $A\lesssim B$ for non-negative quantities $A,B$ means that
		there exists a constant $C\in(0,\infty)$ (only depending on fixed parameters)
		such that $A\le CB$; $A\gtrsim B$ is defined as $B\lesssim A$, and by
		$A\sim B$ we mean that both $A\lesssim B$ and $A\gtrsim B$ hold true. 
		 For indicating a possible dependence of the constant $C$ in the estimate $A\le CB$ on a certain family of parameters $p$, we write $\lesssim_{p}$. Similar notations will be used for the relations $\gtrsim$ and $\sim$.
		\item Unless specified otherwise, $C$ denotes a positive constant that may
		change from line to line.
		\item We usually neither indicate nor mention the dependence of constants on
		fixed parameters such as the number $I$ of species or the space
		dimension. 
		\item $(u,v)_{L^2(\Omega)}=\int_\Omega uv\,\dd x$ denotes the standard
		$L^2(\Omega)$ inner product of $u,v\in L^2(\Omega)$.
		\item Given a Banach space $V$, $V^*$ denotes its topological dual and
		$\langle u,v\rangle_{V^*,V}$ the dual pairing between $u\in V^*$ and
		$v\in V$.
		\item For real matrices $A,B\in \mathbb{R}^{k\times l}$ we let
		$A:B=\sum_{i=1}^k\sum_{j=1}^lA_{ij}B_{ij}$.
		\item For a vector $\phi=(\phi_1,\dots,\phi_I,\phi_{I+1})\in \mathbb{R}^{I+1}$
		we let $\phi'=(\phi_1,\dots,\phi_I)$.
		\item Abbreviate $\ol\kappa\coleq \max_{i,j}\kappa_{j,i}\in(0,\infty)$. The
		dependence of our estimates on $\ol\kappa$ is typically not indicated since
		$\ol\kappa$ remains uniformly bounded throughout our analysis. To indicate
		this, we occasionally also write $\kappa_{j,i}\lesssim 1$.
	\end{itemize}
	\normalsize

	\section{Entropy tools}
	\label{sec:entropy}
	
	\subsection{Entropy variables}
	\label{ssec:ev}
	
	Here, we consider entropies of the form~\eqref{eq:S}, where the equilibria
	$w_i,i\in\{1,\dots,I\}$, are as specified in~\ref{it:w}, and
	$\hat\sigma \in C^2((0,\infty))$ is supposed to satisfy $\hat\sigma''(s)<0$ for
	all $s>0$.  This property ensures that the matrix $D^2S(Z)$ is negative
	definite for all $Z\in(\mathbb{R}_{_>0})^{I+1}$, see the proof
	of~\cite[Prop.~2.1]{MielkeMittnenzweig_2018_convergence}.  We can then define a
	change of variables to the so-called entropy variables
	\begin{equation}\label{eq:114}
	W:=\begin{pmatrix}y \\ v\end{pmatrix}:=-\begin{pmatrix}D_cS \\ \partial_uS\end{pmatrix}_{|(c,u)}=-DS(Z).
	\end{equation}
	The regularity and strict concavity property of the entropy density $S$ ensure that the transformation
	\begin{align}
	-DS:\quad \mathcal{U}:=(\mathbb{R}_{>0})^{I+1}\to -DS(\mathcal{U}), \qquad Z\mapsto -DS(Z)=W
	\end{align}
	is well-defined and invertible.  
	The choice of this transformation is motivated by the fact that 
	it allows us to rewrite the system
	$\dot Z = -\divv ( \mathbb{M}(Z) D^2 S(Z) \nabla Z)$
	as 
	\begin{align*}
	\dot Z = \divv \big( \mathbb{M}(Z)\nabla W \big), 
	\end{align*}
	where here $(c,u)=(-DS)^{-1}(y,v)$ is to be understood as a function of
	$(y,v)$.  Such a transformation has proved useful in the construction of
	solutions to cross-diffusion systems with a formal entropy structure (see
	e.g.~\cite{GGJ_2003,CJ_2004_strongcross,CJ_2006_noselfdiff,Dreher_2008_population,DLM_2014_entropy,Juengel_2015_boundednes-by-entropy,ChenDausJuengel_2018_global,ChenJuengel_2019_renormalised}).
	The main motivation for this transformation lies in the fact that, thanks to
	the positivity of $(-D^2S)$ and the semi-positivity of $\mathbb{M}$, the system
	in the entropy variables is parabolic (in the sense of~\cite{amann_1992}):
	letting $H=H(W)$ denote the Legendre transform of the convex function $-S$, we
	have $Z=DH(W)$, and hence the above system can be written as
	$\tfrac{\dd}{\dd t}(DH(W))=\divv(\mathbb{M}(Z)\nabla W)$ with convex
	$H$. 
	
	Definition~\eqref{eq:114} can be written more explicitly as
	\begin{align*}
	y_i & =\log\big(\tfrac{c_i}{w_i(u)}\big),\qquad v= -\hat\sigma'(u)-\sum_{i=1}^Ic_i\tfrac{w_i'(u)}{w_i(u)}.
	\end{align*}
	If the range of $-DS$ equals $\mathbb{R}^{I+1}$, we can compute for any given $W\in \mathbb{R}^{I+1}$
	\begin{align}\label{eq:307}
	Z=Z(W)=(-DS)^{-1}(W).
	\end{align}
	In fact, the densities $c_i,u$ can be  
	recovered from  $y_i,v$ fairly explicitly:
	substituting $\frac{c_i}{w_i(u)}$ for $\exp(y_i)=\frac{c_i}{w_i(u)},$ we deduce
	$v= -\hat\sigma'(u)-\sum_{i=1}^I\exp(y_i)w_i'(u).$
	Thanks to strict concavity, for given $y\in\mathbb{R}^I$ the function 
	\begin{align}\label{eq:0psi}
	(0,\infty)\ni s\mapsto\Psi(s,y):=-\hat\sigma'(s)-\sum_{i=1}^I\exp(y_i)w_i'(s)
	\end{align}
	is strictly increasing. Denoting by $\Phi(\cdot,y)$ its inverse, we have $u=\Phi(v,y)$ and
	\begin{align}
	c_i = w_i(u)\exp(y_i) = w_i(\Phi(v,y))\exp(y_i).
	\end{align}
	
	In our application (where $S=S_0$ as in~\ga), 
	the range of $-DS_0$ equals $\mathbb{R}^I\times\mathbb{R}_{<0}$
	(cf.~\ref{it:sig1}) and thus fails to coincide with $\mathbb{R}^{I+1}$. To
	circumvent this issue we introduce an approximation $S_\delta$ of $S_0$ such
	that im$(DS_\delta)=\mathbb{R}^{I+1}$. This can be achieved by setting
	\begin{align}
	\label{eq:Seps}
	S_\delta(c, u) = S_0(c,u)-\delta\lambda(u)\qquad\delta\in(0,1],
	\end{align}
	where $\lambda\ge0$ is as in~\eqref{eq:deflb}. 
	Since im$(\lambda')=\mathbb{R}$,
	$-DS_\delta:\mathbb{R}^{I+1}_{>0}\mapsto\mathbb{R}^{I+1}$ is indeed onto.
	Similar approximations have been used in the analysis of population
	models~\cite{DLM_2014_entropy, Juengel_2015_boundednes-by-entropy,
		ChenDausJuengel_2018_global} to deal with sublinear transition rates.  For
	the choice~\eqref{eq:Seps} of the entropy, the associated function
	$\Phi=\Phi_\delta$ is locally bounded, i.e.
	$(0\le)\Phi_\delta(v,y)\le C(|(v,y)|)$ for suitable $C=C_\delta$ and hence
	\begin{align}\label{eq:310}
	0\le Z_i(W)\le C(|W|)
	\end{align}
	for all $i\in\{1,\dots,I+1\}$.  Here, we can see another advantage of this
	approach: non-negativity of the original variables is guaranteed by
	construction. Finally, note that since $D_cS_0=D_cS_\delta$, the Lyapunov type
	hypothesis~\eqref{eq:109} is preserved
	\begin{align}\label{eq:109b}
	D_cS_\delta(c,u)\cdot R(c,u)\ge0\quad \text{ for all }(c,u)\in(0,\infty)^{I+1}.
	\end{align}

	\subsection{A priori estimates}\label{ssec:ep}
	
	Here, we gather a class of a priori estimates fundamental to our problem. Since
	we want to apply these estimates also to certain regularised models to gain
	control on our approximate solutions, the results will be stated in a somewhat
	more general form.

	\subsubsection{Entropy production estimates}\label{sssec:algQ}
	
	\begin{lemma}\label{l:EP}
		Let $S$ be of the form~\eqref{eq:S} for some strictly concave function
		$\hat\sigma\in C^2((0,\infty))$ and equilibrium functions $w_i$ as in
		hp.~\ref{it:w}. Let $\mathbb{M}$ be given by~\eqref{eq:mobility} and
		fulfill~\ref{it:mi1} and~\ref{it:pi}. Further suppose hypotheses~\HH\ or,
		alternatively, suppose hypothesis~\HHp\ with $\kappa_{1,i}=0$ for all $i$.
		Abbreviate $M(Z):=D^2S(Z)\mathbb{M}(Z)D^2S(Z)$.  There exists $\epsilon>0$
		depending only on fixed parameters such that for all
		$\zeta:=(\xi,\eta)\in\mathbb{R}^{I+1}$ and all
		$(c,u)\in \mathbb{R}_{\ge0}^{I+1}$
		\begin{equation}
		\zeta^TM(c,u)\zeta \ge\epsilon\Big[\sum_{i=1}^I\Big(\kappa_{1,i}|\xi_i|^2+\kappa_{0,i}|\tfrac{1}{\sqrt{c_i}}\xi_i|^2\Big)
		+	\pi_1\gamma^2|\eta|^2\Big],
		\end{equation}
		where, as before,  $\gamma(c,u):=-\big(\hat\sigma''(u) +\sum_{l=1}^I\tfrac{w_l''}{w_l}c_l\big)$.
	\end{lemma}
	\begin{proof}We compute
		\begin{align}\nonumber
		\zeta^TM(c,u)\zeta & = \sum_{i=1}^{I} m_i|\tfrac{1}{c_i}\xi_i-\tfrac{w_i'}{w_i}\eta|^2+
		\pi_1\Big(\hat\sigma''(u) +\sum_{i=1}^I\tfrac{w_i''}{w_i}c_i\Big)^2|\eta|^2,
		\end{align}
		where $w_i=w_i(u),\pi_1=\pi_1(c,u)$.
		
		Using the bound $|n+\tilde n|^2\ge\tfrac{1}{2}|n|^2-|\tilde n|^2$, we estimate
		\begin{align}
		m_i|\tfrac{1}{c_i}\xi_i-\tfrac{w_i'}{w_i}\eta|^2 & \ge
		\tfrac{1}{2}m_i|\tfrac{1}{c_i}\xi_i|^2-m_i|\tfrac{w_i'}{w_i}\eta|^2.
		\end{align}
		By~\ref{it:mi1}, we have on the one hand
		\begin{equation*}
		m_i|\tfrac{1}{c_i}\xi_i|^2\gtrsim \kappa_{1,i}|\xi_1|^2+\kappa_{0,i}|\tfrac{1}{\sqrt{c_i}}\xi_i|^2.
		\end{equation*}
		On the other hand, supposing~\HHp\ with $\kappa_{1,i}=0$ for all $i$ allows us to estimate
		\begin{equation*}
		\sum_{i=1}^I m_i|\tfrac{w_i'}{w_i}\eta|^2\lesssim -\sum_{i=1}^I\kappa_{0,i}\tfrac{w''_i}{w_i}c_i|\eta|^2 
		\lesssim \pi_1\gamma^2|\eta|^2\le \zeta^TM(c,u)\zeta,
		\end{equation*}
		while under Hypotheses~\HH\ we find using~\ref{it:H4} resp.~\ref{it:H5} and~\ref{it:H2}
		\begin{align}
		|c_i\tfrac{w_i'}{w_i}\eta|^2 & \lesssim|\pi_1^\frac{1}{2}\tfrac{w_i''}{w_i}c_i\eta|^2,
		\\	|\sqrt{c_i}\tfrac{w_i'}{w_i}\eta|^2&\lesssim|\pi_1^\frac{1}{2}\tfrac{w_i''}{w_i}c_i\eta|^2
		+|\pi_1^\frac{1}{2}\gamma\eta|^2,
		\end{align}
		where the right-hand sides of the last two estimates are bounded above by~$|\pi_1^\frac{1}{2}\gamma\eta|^2$. 
		
		By our convention that $\max_{j,i}\kappa_{j,i}=\ol\kappa\lesssim1$, this completes the proof of the assertion.
	\end{proof}
	
	Let us now suppose that $Z=(c_1,\dots,c_I,u)$ is a measurable function and that the gradients $\nabla Z$ and $\nabla W$ are defined in a suitable sense, where $W=-DS(c,u)$.
	Then, defining (formally)
	\begin{equation}\label{eq:001i}
	\begin{split}
	&Q(c,u):=\nabla W:(\mathbb{M}(Z)\nabla W), \text{ where }W=-DS(c,u),
	\\\quad&\mathcal{Q}(c,u):=\int_\Omega Q(c,u)\,\dd x, 
	\end{split}
	\end{equation} and (cf.~\eqref{eq:1000})
	\begin{equation}\label{eq:401}
	\begin{split}
	P(c,u)&:=\sum_{i=1}^I\Big(\kappa_{1,i}|\nabla c_i|^2
	+\kappa_{0,i}|\nabla \sqrt{c_i}|^2\Big)
	+\pi_1\gamma^2|\nabla u|^2,
	\\\mathcal{P}(c,u)&:=\int_\Om P(c,u)\,\dd x,
	\end{split}
	\end{equation}
	Lemma~\ref{l:EP} implies the existence of $\epsilon_*>0$ such that
	\begin{align}\label{eq:400}
	\epsilon_* \mathcal{P}(c,u)\le \mathcal{Q}(c,u).
	\end{align}
	(Observe that $\mathcal{Q}$ agrees with $\mathcal{P}_{\mathrm{diff}}$ on page~\pageref{page:Pdiff}.)
	\subsubsection{Entropy and flux bounds}
	\label{suu:EntFluxBd} 
	To derive an upper bound for $c_i$ using the total entropy
	$\mathcal{S}(c,u)=\int_\Omega S(c,u) \,\dd x$, 
	we define for $p>1$ the $p$-entropy function
	$U_p(w)=\frac1{p(p{-}1)} \big( w^p - p w + p-1\big) \geq 0$, which is
	characterised by $U''_p(w)=w^{p-2}$ and $U_p(1)=U'_p(1)=0$. An elementary calculation shows that  for all $p>1$ and all $c\ge 0,w> 0$
	\begin{align}
	w \,\lambda \big(\tfrac cw \big) \ =	
	\ \tfrac{p{-}1}p \lambda (c) - (p{-}1) U_p(w) +\tfrac1p w^p\lambda(\tfrac c{w^p}),
	\end{align}
	where  as before $\lambda(z)=z \log z - z +1$. 
	As a consequence,
	\begin{equation}
	\label{eq:ERE.Bz}
	w \,\lambda \big(\tfrac cw \big) \ \geq \ \tfrac{p{-}1}p \lambda (c) - (p{-}1)
	U_p(w)  \ \geq \ \tfrac{p{-}1}p \lambda (c) - \tfrac1p\, w^p - 1,
	\end{equation}
	where the second inequality follows from the positivity of $w$ and the fact that $p>1$.
	\begin{lemma}[Upper and lower entropy bounds]
		\label{l:S} 
		Let $S=S(c,u)$ have the form~\eqref{eq:S} and satisfy
		Hypotheses~\ref{mainhypo1} with $S=S_0$, $\hat\sigma=\hat\sigma_0$, and
		$0<w_0\leq w_i(u)\leq C_w^\beta(1{+}u)^\beta$, where $\beta\in (0,1)$ is from \ref{it:H3}. Then, 
		for all $(c,u)\in[0,\infty)^{I+1}$,
		\begin{align}
		\label{eq:112ptw}
		S(c,u)&\leq \hat\sigma(u)  + 2I\beta C_w u + I(2\beta C_w{+}1) -(1{-}\beta)
		\sum_{i=1}^I \lambda(c_i) ,
		\\
		\label{eq:112ptwii} 
		S(c,u)&\ge \hat\sigma(u)+I- \frac I{\min\{w_0,1\}} -2\sum_{i=1}^I \lambda(c_i) .
		\end{align}
	\end{lemma}
	\begin{proof}
		The upper estimate follows by applying \eqref{eq:ERE.Bz} with
		$p=1/\beta>1$ and adding over $i=1,\ldots,I$, namely
		\begin{align*}
		S(c,u)& \le \sigma(u)-\sum\nolimits_1^I\big(\tfrac{p{-}1}p \lambda(c_i)- (p{-}1)
		U_p(w_i(u))\big)
		\\
		&\le \hat\sigma(u) -I+\sum\nolimits_1^Iw_i(u) - \sum_1^I\big((1{-}\beta) \lambda(c_i)
		-\beta C_w(1{+}u)-1 \big) ,
		\end{align*}
		which gives \eqref{eq:112ptw} if we invoke $w_i(u)\leq \beta
		C_w(1{+}u)+1-\beta$.  
		
		The lower estimate~\eqref{eq:112ptwii} is obtained by using $w_i(u)\geq w_0>0$, which leads to 
		\begin{align*}
		S(c,u)& \geq \hat\sigma(u)+\sum\nolimits_1^I\big(c_i \log w_0 -\lambda(c_i)\big).
		\end{align*}	
		For $w_0\geq 1$ we may simply drop the term $c_i\log w_0\geq 0$ and obtain
		\eqref{eq:112ptwii}. For $w_0 \in (0,1)$ we use the the Young--Fenchel inequality
		\[
		-c_i \log w_0 = c_i \log\big(\tfrac1{w_0}\big) \leq \lambda(c_i) +
		\lambda^*\big( \log\big(\tfrac1{w_0}\big)\big)= \lambda(c_i) + \tfrac1{w_0}-1,
		\] 
		where $\lambda^*$ denotes the Legendre transform of $\lambda$.
		Estimate~\eqref{eq:112ptwii} now easily follows.  
	\end{proof} 
	
	In the following lemma, we establish general bounds on the flux term. Before
	stating the assertion, let us recall our convention that~$0\le
	\kappa_{j,i}\lesssim 1$ for all $j,i$. 
	
	\begin{lemma}[Control of the flux]
		\label{l:efl}
		Let $A(Z):=-\mathbb{M}(Z)D^2S(Z)$, where $S$ is of the general
		form~\eqref{eq:S} for some strictly concave function
		$\hat\sigma\in C^2((0,\infty))$ and equilibrium functions $w_i$ as in
		hp.~\ref{it:w}. Let $\mathbb{M}$ have the form~\eqref{eq:mobility}
		with~\ref{it:mi1},~\ref{it:pi} and~\ref{it:H6}, and assume that $Z=(c,u)$ has
		non-negative components and is such that $P(c,u)$ given by~\eqref{eq:401} is a
		well-defined function.
		
		If~\HHp\ holds true and $\kappa_{1,i}=0$ for all $i$, then
		\begin{equation}\label{eq:402.3}
		\begin{split}
		|A(Z)\nabla Z|&\lesssim \max_{i=1,\dots,I}\Big(
		\sqrt{c_i}+\sqrt{\pi_1(Z)}\Big) P^\frac{1}{2}(Z).
		\end{split}
		\end{equation}
		If instead hypotheses~\HH\ are fulfilled, then
		\begin{equation}\label{eq:402}
		\begin{split}
		|A(Z)\nabla Z|&\lesssim \max_{i=1,\dots,I}\Big(c_i
		+\kappa_{0,i}+\sqrt{\pi_1(Z)}\Big) P^\frac{1}{2}(Z).
		\end{split}
		\end{equation}
		Thus, in both cases, for any $\xi\in C([0,\infty)^{I+1})$ with
		$\supp\xi\subseteq \{|Z|\le E\}$
		\begin{equation}\label{eq:105b}
		|\xi(Z)A(Z)\nabla Z|\lesssim (E+1) P^\frac{1}{2}(Z).
		\end{equation}
	\end{lemma}
	\begin{proof}
		We first note that
		\begin{equation}\label{eq:101c}
		\begin{split}
		|A_{ii}(c,u)\nabla c_i|&\lesssim \kappa_{1,i}c_i|\nabla c_i|+\kappa_{0,i}\sqrt{c_i}|\nabla \sqrt{c_i}|
		\\&\lesssim \big(\kappa_{1,i}^\frac{1}{2}c_i
		+\kappa_{0,i}^\frac{1}{2}\sqrt{c_i}\big)P^\frac{1}{2}(Z).
		\end{split}
		\end{equation}
		Let us next turn to the coefficients $A_{i,I+1}(Z)$ (given by~\eqref{eq:defb*}). Supposing~\HHp, we recall  that $\kappa_{1,i}=0$ in this case, and
		estimate using~\ref{it:w.m3} and~\ref{it:pi.m3}
		\begin{equation}\label{eq:101.3}
		\begin{split}
		|A_{i,I+1}(Z)| 	&\lesssim \kappa_{0,i}\sqrt{c_i}\big(\tfrac{w_i''}{w_i}c_i\big)^\frac{1}{2}
		+\sqrt{c_i}\big(\tfrac{w_i''}{w_i}c_i\big)^\frac{1}{2}
		\\&\lesssim (\kappa_{0,i}+1)\sqrt{c_i}\pi_1^\frac{1}{2}\gamma
		\end{split}
		\end{equation}
		and hence 
		\begin{align}\label{eq:101b.3}
		|A_{i,I+1}(Z)\nabla u|\lesssim (\kappa_{0,i}+1)\sqrt{c_i}P^\frac{1}{2}(Z).
		\end{align}
		By our convention that $\kappa_{j,i}\lesssim 1$ for all $j,i$, we see that the RHS is bounded above by the RHS of~\eqref{eq:402.3}.
		
		Supposing instead~\HH, we estimate using~\ref{it:3},~\ref{it:1}
		\begin{align}\label{eq:101}
		|A_{i,I+1}(Z)| 	&\lesssim
		-a_i\pi_1^\frac{1}{2}\tfrac{w_i''}{w_i}c_i+\pi_1^\frac{1}{2}\gamma c_i
		\\&\lesssim (a_i(c,u)+c_i)\pi_1^\frac{1}{2}\gamma 
		\end{align}
		and deduce 
		\begin{align}\label{eq:101b}
		|A_{i,I+1}(Z)\nabla u|\lesssim ((\kappa_{1,i}+1)c_i+\kappa_{0,i})P^\frac{1}{2}(Z).
		\end{align}
		
		Finally, we recall (cf.\;eq.~\eqref{eq:defb*}) that $0\le A_{I+1,I+1}(Z)= \pi_1\gamma$ to infer
		\begin{align}\label{eq:102b}
		|A_{I+1,I+1}(Z)\nabla u|\lesssim \pi_1^\frac{1}{2}(Z)P^\frac{1}{2}(Z).
		\end{align}
		Put together, we obtain estimate~\eqref{eq:402.3} resp.~\eqref{eq:402}. 
	\end{proof}

	\section{Approximation scheme}
	\label{sec:delta}
	
	In this and the subsequent section (Sections~\ref{sec:delta}--\ref{sec:time}),
	we assume that all hypotheses of Theorem~\ref{thm:main} hold true, and the main
	purpose of these sections is to prove Theorem~\ref{thm:main}.  In the current
	section, we construct solutions to an approximate time-discrete, regularised
	problem.  In our approximate equations we replace the time derivative by a
	backward difference quotient $\dot Z\approx\tau^{-1}(Z-Z^{k-1})$ with
	$0<\tau\ll1$ denoting the size of the time step and $Z^{k-1}$ corresponding to
	the solution at the previous time step, and introduce an elliptic
	regularisation at the level of the entropy variables.  Schemes of that type
	have been proposed in~\cite{Juengel_2015_boundednes-by-entropy} (see
	also~\cite{BDPS_2010,J_2000_thermo,DGJ_1997} for a selection of earlier works).
	
	Throughout this section, $\delta>0$ is kept fixed and we assume that
	$S=S_\delta$ is given by~\eqref{eq:Seps}.  Then, the associated
	$\gamma=\gamma_\delta$ (see def.~\eqref{eq:def.gamma}) depends on $\delta$, and
	we have $\gamma_\delta\ge \gamma$. To obtain good a priori estimates, we will
	also consider a modified coefficient $\pi_{1,\delta}$ converging, as
	$\delta\to0$, to the function $\pi_1$ considered in Theorem~\ref{thm:main}.
	The specific choice of $\pi_{1,\delta}$ depends on whether~~\HHp\ or~\HH\ are
	considered:
	\begin{itemize}
		\item Under assumption~\HHp, we choose
		$\pi_{1,\delta}=\pi_1\gamma\tfrac{1}{\gamma_\delta}$, which ensures
		that~\ref{it:pi.m3} is preserved for the
		$\delta$-model. Condition~\ref{it:w.m3} remains obviously valid.
		\item Let us now instead suppose~\HH. In this case, we simply let
		$\pi_{1,\delta}(c,u)=\pi_1(c,u)+\delta u^2$. This ensures that
		$\pi_{1,\delta}^\frac{1}{2}(c,u)\lambda''(u)\gtrsim_\delta 1$, where
		$\lambda$ denotes the Boltzmann function (so that
		$\lambda''(u)=\tfrac{1}{u}$). Here, we need to point out that hp.~\ref{it:H5}
		remains true since $u\tfrac{w_i'(u)}{w_i(u)}\lesssim 1$ as a consequence of
		hp.~\ref{it:w}. (Indeed,
		$w_i(u)\ge w_i(u)-w_i(0)= \int_0^uw'_i(v)\,\dd v\ge uw'_i(u)$ since $w_i$ is
		concave.) The bounds~\ref{it:H2},~\ref{it:3} for the corresponding
		$\delta$-quantities are obvious.
	\end{itemize}
	Of course, the $\delta$-dependence of $\pi_1=\pi_{1,\delta}$ implies that
	$\mathbb{M}=\mathbb{M}_\delta$ is also $\delta$-dependent. Crucial for our
	subsequent analysis is the observation that thanks to the properties listed
	above, hypotheses~\HH\ resp.~\HHp\ and thus the a priori estimates in
	Lemma~\ref{l:EP} and Lemma~\ref{l:efl} are equally true for the
	$\delta$-regularised model.  The $\delta$-dependencies will therefore not be
	explicitly indicated in the current section.  We also assume, in
	this section, that the reaction rates $R=(R_i(Z))_{i=1}^I$ are globally
	bounded. Possible dependencies on $\|R_i\|_{L^\infty}$ will always be
	indicated.  These hypotheses regarding $S$ and $R$ will be removed in
	Section~\ref{sec:time} by approximation.
	
	Given an approximate solution $Z^{k-1}=(c^{k-1},u^{k-1})$ (at the previous time
	step) we are concerned with constructing a solution at the subsequent time
	step. This amounts to solving a problem of elliptic type obtained by the
	backward time discretisation.  We therefore suppose in this section that
	\begin{align}\label{eq:2023}
	Z^{k-1}=(c^{k-1},u^{k-1})\in (L\log L(\Omega))^I\times L^2(\Omega)
	\end{align}
	is a given vector-valued function with non-negative components.  
	We will further require  the hypothesis
	\begin{align}\label{eq:2024}
	\hat\sigma_{-}(u^{k-1})\in L^1(\Omega), 
	\end{align}
	which is non-redundant if $\hat\sigma(0+)=-\infty$. It ensures that the entropy $\int_\Omega S(Z^{k-1})\,\dd x$  is finite (see Lemma~\ref{l:S}).
	
	To construct a solution to the nonlinear elliptic problem in the entropy variables, it is convenient to introduce a higher order regularisation of order $m>\frac{d}{2}$, so that $H^m(\Omega)\overset{\mathrm{c}}{\hookrightarrow} C_b(\Omega)$. The associated regularisation parameter is denoted by $\varepsilon>0$.
	To simplify notation we will often drop superscripts like $(\cdot)^{I+1}$ etc.\ when denoting spaces of vector-valued functions.
	
	\begin{lemma}\label{l:weakLin}
		Given $\tilde W\in L^\infty$ and letting $\tilde Z:=(-DS)^{-1}(\tilde W)$,
		there exists a unique
		$W=(y,v)\in (H^m)^{I+1}$ satisfying  for all $\phi=(\phi',\phi_{I+1})\in (H^m)^{I+1}$
		\begin{align*}
		\tau\int_\Omega (\mathbb{M}(\tilde Z) & \nabla W):\nabla \phi \,\dd x 
		+ \tau\varepsilon\left(\sum_{|\alpha|=m}\int_\Omega \partial^\alpha W\cdot \partial^\alpha \phi\,\dd x+\int_\Omega W\cdot\phi\,\dd x \right) \\&
		=-\int_\Omega (\tilde Z-Z^{k-1})\cdot\phi\,\dd x +\tau\int_\Omega R(\tilde Z)\cdot\phi'\,\dd x.
		\end{align*}
	\end{lemma}
	\begin{proof}
		This follows from the Lax--Milgram theorem using the local boundedness of the functions $\mathbb{M}(\cdot)$ and $R(\cdot)$, and the fact that by~\eqref{eq:310}, $\|\tilde Z\|_{L^\infty}\le C_\delta(\|\tilde W\|_{L^\infty})$. (See e.g.~\cite{Juengel_2015_boundednes-by-entropy,ChenJuengel_2019_renormalised} for the proofs of rather similar assertions).
	\end{proof}
	
	We now construct a solution to the nonlinear problem.
	\begin{lemma}\label{l:weakNl1}
		There exists $W\in H^m$ such that for all $\phi\in H^m:$
		\begin{equation}
		\begin{split}
		\tau\int_\Omega (\mathbb{M}(Z) \nabla W)&:\nabla \phi \,\dd x 
		+ \tau\varepsilon\big(\sum_{|\alpha|=m}\int_\Omega \partial^\alpha W\cdot \partial^\alpha \phi \,\dd x+\int_\Omega W\cdot\phi\,\dd x \big)
		\\&=-\int_\Omega (Z-Z^{k-1})\cdot\phi\,\dd x+ \tau\int_\Omega R(Z)\cdot\phi'\,\dd x. 
		\end{split}
		\label{eq:regweakNL}
		\end{equation}
		Here $Z=(-DS)^{-1}(W)$.
	\end{lemma}
	\begin{proof}We want to apply the Leray--Schauder theorem (see~\cite[Theorem~11.3]{GilbargTrudinger_book}).
		For this purpose, we define a fixed point map 
		\begin{align*}
		\Gamma: L^\infty\to L^\infty,\qquad \tilde W\mapsto W,
		\end{align*}
		where $W\in H^m\overset{c}{\hookrightarrow} L^\infty$ denotes the solution obtained in Lemma~\ref{l:weakLin}. 
		
		We proceed in three steps.
		
		\underline{Step~1:} $\Gamma$ is a continuous operator.
		Suppose that $\tilde W_j\to \tilde W$ in $L^\infty$ and let $W_j=\Gamma(\tilde W_j)$ and $\tilde Z_j:=(-DS)^{-1}(\tilde W_j)$. Then 
		$L:=\sup_j\|\tilde W_j\|_{L^\infty}<\infty$ and hence
		\begin{align*}
		\|\tilde Z_{j} \|_{L^\infty}\le C(L).
		\end{align*}
		Choosing $\phi=W_j$ in the equation for $W_j$ we therefore get for $\theta>0$
		\begin{align}\label{eq:208}
		\tau\varepsilon \|W_j\|_{H^m}^2
		\le C(L,\theta,\|Z^{k-1}\|_{L^1}, \|R\|_{L^\infty})+\theta\|W_j\|_{L^\infty}^2.
		\end{align}
		Using the embedding $H^m\hookrightarrow L^\infty$, an absorption argument gives
		\begin{align*}
		\|W_j\|_{H^m}^2 \le C(L,\tau,\varepsilon,\|Z^{k-1}\|_{L^1}, \|R\|_{L^\infty}).
		\end{align*}
		Hence, there exists $W_\infty\in H^m$ such that, along a subsequence, $W_j\rightharpoonup W_\infty$ in $H^m$, 
		$W_j\to W_\infty$ in $C_b(\Omega)$. Observe that, after possibly passing to a subsequence,  we can assume that $\tilde Z_j \to \tilde Z$ a.e.\;in $\Omega$ and, by dominated convergence, also $\tilde Z_{j} \to \tilde Z$ in $L^p$ for any $p\in[1,\infty)$. It is now easy to see that $W_\infty$ satisfies for all $\phi\in H^m$
		\begin{equation*}
		\begin{split}
		&\tau\int_\Omega (\mathbb{M}(\tilde Z) \nabla W_\infty):\nabla \phi \,\dd x 
		+ \tau\varepsilon\left(\sum_{|\alpha|=m}\int_\Omega \partial^\alpha W_\infty\cdot \partial^\alpha \phi' \,\dd x+\int_\Omega W_\infty\cdot\phi\,\dd x \right)
		\\&=-\int_\Omega (\tilde Z-Z^{k-1})\cdot\phi\,\dd x+\tau\int_\Omega R(\tilde Z)\cdot\phi'\,\dd x.
		\end{split}
		\end{equation*}
		By the uniqueness of solutions to the linear equation, we infer that $W_\infty=\Gamma(\tilde W)$.
		
		The above reasoning shows that any subsequence of $(\tilde W_j)$ has a subsequence along which $\Gamma$ converges to $\Gamma(\tilde W)$ in $L^\infty$.
		This implies that $\Gamma(\tilde W_j)\to \Gamma(\tilde W)$ in $L^\infty$. Hence $\Gamma$ is continuous.
		
		\underline{Step~2:} $\Gamma$ is a compact operator. Arguments similar to the proof of Step~1 show that the image of any bounded set in $L^\infty$ under the map $\Gamma$ is bounded in $H^m$ (cf.~\eqref{eq:208}).
		Taking also into account the compactness of the embedding
		$H^m\hookrightarrow L^\infty$, we infer that the operator $\Gamma$ is compact.
		
		\underline{Step~3:}
		A priori bound. Suppose that $W=\gamma_1 \Gamma(W)$ for some $\gamma_1\in(0,1]$.
		By hypothesis, $W$ satisfies for all $\phi\in H^m$ the equation
		\begin{align*}
		& \tau\int_\Omega (\mathbb{M}(Z) \nabla W):\nabla \phi \,\dd x 
		+ \tau\varepsilon\bigg(\sum_{|\alpha|=m}\int_\Omega \partial^\alpha W\cdot \partial^\alpha \phi \,\dd x+\int_\Omega W\cdot\phi\,\dd x \bigg)
		\\&=-\gamma_1\int_\Omega (Z-Z^{k-1})\cdot\phi\,\dd x+\gamma_1\tau\int_\Omega R(Z)\cdot\phi'\,\dd x.
		\end{align*}
		Choosing $\phi=W=-DS(Z)$ we obtain
		\begin{equation}
		\begin{split}
		\tau\int_\Omega (\mathbb{M}(Z) &\nabla W):\nabla W\,\dd x 
		+ \tau\varepsilon\bigg(\sum_{|\alpha|=m}\int_\Omega |\partial^\alpha W|^2 \,\dd x+\int_\Omega |W|^2\,\dd x\bigg)
		\\&=\gamma_1\int_\Omega (Z-Z^{k-1})\cdot DS(Z)\,\dd x
		-\gamma_1\tau\int_\Omega R(Z)\cdot D_cS(Z)\,\dd x
		\\& \le \gamma_1 (\mathcal{S}(Z)-\mathcal{S}(Z^{k-1})).
		\end{split}
		\label{eq:01}
		\end{equation}
		Choosing $\phi=(0,\dots,0,1)^T$ further yields
		\begin{align}\label{eq:energy1}
		\gamma_1\int_\Omega u \,\dd x =\gamma_1 \int_\Omega u^{k-1} \,\dd x-\varepsilon\tau\int_\Omega v\,\dd x\le  \int_\Omega u^{k-1} \,\dd x + \varepsilon\tau C
		+\tfrac{1}{2}\varepsilon\tau\|v\|_{L^2}^2.
		\end{align}
		Adding the last inequality to~\eqref{eq:01} yields upon absorption
		\begin{multline}
		\gamma_1\int u\,\dd x+\tau\mathcal{Q}(c,u)
		+ \tfrac{1}{2}\tau\varepsilon\big(\sum_{|\alpha|=m}\|\partial^\alpha W\|^2+\|W\|_{L^2}^2 \big)
		\\\label{eq:0010}\le \gamma_1 \big(\mathcal{S}(Z)-\mathcal{S}(Z^{k-1})\big)+\int_\Omega u^{k-1} \,\dd x +\varepsilon\tau C,
		\end{multline}
		where $\mathcal{Q}(c,u)$ is defined in~\eqref{eq:001i}. Using Lemma~\ref{l:S} to estimate the RHS and applying standard H\"older's and Young's inequality (to deal with integrals of sublinear functions of $u$), we can infer after absorption (possibly increasing $\beta<1$ in Lemma~\ref{l:S})
		\begin{equation}\label{eq:305}
		\begin{split}
		\gamma_1 c_\beta\sum_{i=1}^I\int_\Omega c_i\log_+(c_i)\,\dd x  &+\tfrac{\gamma_1}{2}\int_\Omega u\,\dd x+\gamma_1\|\hat\sigma_{-}(u)\|_{L^1}+ \tau \mathcal{Q}(c,u)
		\\&+ \tfrac{1}{2}\tau\varepsilon\left(\sum_{|\alpha|=m}\|\partial^\alpha W\|^2_{L^2}+\|W\|_{L^2}^2\right)
		\\\le \|\hat\sigma_-(u^{k-1})\|_{L^1}+&C\sum_{i=1}^I\int_\Omega c_i^{k-1}\log_+(c_i^{k-1})\,\dd x +\|u^{k-1}\|_{L^2}^2+C.
		\end{split}
		\end{equation}
		Here, we used the rough estimate $\delta\|\lambda(u^{k-1})\|_{L^1}\lesssim \|u^{k-1}\|_{L^2}^2+1$. 
		
		In particular, we have obtained the bound 
		\begin{align}\label{eq:apridelta}
		\|W\|_{L^\infty}\le C\|W\|_{H^m}\le  
		C(\|c^{k-1}\|_{L\log L},\|u^{k-1}\|_{L^2}, \|\hat\sigma_-(u^{k-1})\|_{L^1},\tau,\varepsilon).
		\end{align}
		Theorem~11.3 in~\cite{GilbargTrudinger_book} now yields the existence of a fixed point $W=\Gamma(W)$.
	\end{proof}
	
	Letting $\gamma_1=1$ in Step~3 of the proof of Lemma~\ref{l:weakNl1} (see~\eqref{eq:01}, \eqref{eq:305}), we infer 
	using the entropy production estimate~\eqref{eq:400}
	\begin{align}
	\epsilon_*\tau\mathcal{P}(c,u)
	+\tau\varepsilon\Big(\sum_{|\alpha|=m}\|\partial^\alpha W\|_{L^2}^2+\|W\|_{L^2}^2 \Big)
	\label{eq:110b}\le \mathcal{S}(Z)-\mathcal{S}(Z^{k-1}),
	\end{align}
	and
	\begin{multline}
	\|(c,u)\|_{L^1}+\|\hat\sigma_-(u)\|_{L^1}+\epsilon_*\tau\mathcal{P}(c,u)\label{eq:1005}
	+\tau\varepsilon\big(\sum_{|\alpha|=m}\|\partial^\alpha W\|_{L^2}^2+\|W\|_{L^2}^2\big)
	\\\le C(\|c^{k-1}\|_{L\log L},\|u^{k-1}\|_{L^2},\|\hat\sigma_{-}(u^{k-1})\|_{L^1}).
	\end{multline}
	To proceed, we need to distinguish the cases~\HH\ and~\HHp.
	
	Let us first assume that hypotheses~\HH\ hold true.
	Then, by the Sobolev embedding, the fact that $\pi_1^\frac{1}{2}=\pi_{1,\delta}^\frac{1}{2}\gtrsim_\delta u$, $\gamma_\delta\gtrsim_\delta\tfrac{1}{u}$ and the definition~\eqref{eq:401} of $\mathcal{P}(c,u)=\mathcal{P}_\delta(c,u)$, we have
	\begin{align}
	\|u\|_{L^2} & \lesssim_\delta \|\nabla u\|_{L^2}+\|u\|_{L^1}\label{eq:120}
	\\&\lesssim C(\|c^{k-1}\|_{L\log L},\|u^{k-1}\|_{L^2},\|\hat\sigma_{-}(u^{k-1})\|_{L^1},\tau,\epsilon_*,\delta).
	\end{align}
	Under hypotheses~\HHp\ it suffices to note that 
	\begin{equation}\label{eq:120.3}\tag{\text{\ref*{eq:120}'}}
	|\nabla\sqrt{u}|\lesssim_\delta\sqrt{\gamma_\delta}|\nabla u|\lesssim P_\delta^\frac{1}{2}(c,u).
	\end{equation} 
	
	We are now in a position to pass to the limit $\varepsilon\to0$ in problem~\eqref{eq:regweakNL}.
	
	\begin{proposition}\label{prop:eps=0}
		Let $A(Z)=-\mathbb{M}(Z)D^2S(Z)$.	There exists $Z=(c_1,\dots,c_I,u)$ with $c_i,u\ge0$, $\kappa_{0,i}\sqrt{c_i}\in H^1(\Omega),$ $\kappa_{1,i}c_i\in H^1(\Omega)$ and $u\in H^1(\Omega)$ in case \HH, $\sqrt{u}\in H^1(\Omega)$ in case \HHp\ such that for all $\phi=(\phi',\phi_{I+1})\in W^{1,\infty}(\Omega)^{I+1}$
		\begin{align}\label{eq:eps=0}
		\tau\int_\Omega A(Z) \nabla Z:\nabla \phi \,\dd x =-\int_\Omega (Z-Z^{k-1})\cdot\phi\,\dd x + \tau\int_\Omega R(Z)\cdot\phi'\,\dd x.
		\end{align}
		Furthermore, 
		\begin{align}\label{eq:1010}
		\tau\epsilon_*\mathcal{P}(Z)
		\le \mathcal{S}(Z)-\mathcal{S}(Z^{k-1}).
		\end{align}
	\end{proposition}
	\begin{proof}
		By~\eqref{eq:1005},~\eqref{eq:120} resp.~\eqref{eq:120.3}, and Sobolev embeddings there exists a (non-negative) vector-valued function $Z=(c,u)$ and a sequence $\varepsilon_j\searrow 0$ such that the associated solutions $W_j=-DS(Z_j)=-DS(c_j,u_j)$ of Lemma~\ref{l:weakNl1} (with $\varepsilon=\varepsilon_j$) satisfy for suitable $\wt\epsilon=\wt\epsilon(d)>0:$
		\begin{enumerate}[label=(\roman*)]
			\item\label{it:c1} $c_{j,i}\rightharpoonup c_i$ in $H^1(\Omega)$, $c_{j,i}\to c_i$ in $L^{2+\wt\epsilon}(\Omega)$ and a.e.\ in $\Omega$ provided $\kappa_{1,i}>0$
			\item\label{it:c1*} $\sqrt{c_{j,i}}\rightharpoonup \sqrt{c_i}$ in $H^1(\Omega)$, $\sqrt{c_{j,i}}\to \sqrt{c_i}$ in $L^{2+\wt\epsilon}(\Omega)$ and a.e.\ in $\Omega$ if $\kappa_{0,i}>0$
			\item\label{it:c2} $\nabla u_j\rightharpoonup \nabla u$ in $L^2(\Omega)$,
			$u_j\to u$ in $L^{2+\wt\epsilon}(\Omega)$ and a.e.\ in $\Omega$ under hp.~\HH\\
			\item\label{it:c2.3}   $\nabla \sqrt{u_j}\rightharpoonup \nabla \sqrt{u}$ in $L^2(\Omega)$,
			$\sqrt{u_j}\to \sqrt{u}$ in $L^{2+\wt\epsilon}(\Omega)$ and a.e.\ in $\Omega$ under hp.~\HHp.
		\end{enumerate}
		Thanks to Lemma~\ref{l:efl}, hp.~\ref{it:4} and the above convergence results (using in case~\HH\ also the elementary Lemma~\ref{l:1}), we further deduce
		\begin{align}
		A_{\varepsilon_j}(Z_j)\nabla Z_j\rightharpoonup A(Z)\nabla Z\text{ in }L^{1+\epsilon}(\Omega)\label{eq:213}
		\end{align}
		for some $\epsilon=\epsilon(d)>0$. Observe that in case~\HH, owing to estimate~\eqref{eq:402}, we here require the hypothesis that~$\kappa_{1,i}>0$ for all $i$ to ensure $L^{2+\wt\epsilon}(\Om)$-integrability of $c_i$ (and not only of $\sqrt{c_i}$). 
		By boundedness and continuity, the passage to the limit in the reaction rates is immediate.
		The fact that $\varepsilon_j\big(\sum_{|\alpha|=m}\|\partial^\alpha W\|_{L^2}+\|W\|_{L^2}\big)\lesssim\varepsilon_j^{\frac{1}{2}}$ and the above convergence results allow us to pass to the limit in eq.~\eqref{eq:regweakNL} to obtain~\eqref{eq:eps=0} for all $\phi\in H^m(\Omega)^{I+1}$. A density argument then allows to extend this identity in particular to all $\phi\in W^{1,\infty}(\Omega)^{I+1}$. 
		
		Inequality~\eqref{eq:110b} and weak lower semi-continuity in $L^2$ yield~\eqref{eq:1010}.
		Here, we also used the fact that $\limsup_{j\to\infty}\mathcal{S}(c_j,u_j)\le\mathcal{S}(c,u)$, which follows from~\ref{it:c1}--\ref{it:c2.3}) and 
		Fatou's lemma applied to  $\hat\sigma_{-}(u_j)$ giving
		$\limsup_{j\to\infty}\int_\Om \hat\sigma_{-}(u_j)\,\dd x\le \int_\Om \hat\sigma_{-}(u)\,\dd x$.
	\end{proof}
	
	Having removed the regularisation in the entropy variables,
	we derive an $L^2$-energy estimate to upgrade the regularity of $u$. 
	
	\begin{lemma}[Energy estimate]\label{l:energy}
		The solution $Z=(c,u)$ obtained in Prop.~\ref{prop:eps=0} satisfies the bound
		\begin{align}\label{eq:121}
		2\tau\int_\Omega a(c,u)|\nabla u|^2\dd x+\int_\Omega u^2\,\dd x \le \int_\Omega (u^{k-1})^2\,\dd x.
		\end{align}
	\end{lemma}
	\begin{proof}
		It follows from~\eqref{eq:eps=0} that for all $\psi\in W^{1,\infty}(\Omega)$
		\begin{align}\label{eq:1007}
		\tau\int_\Omega a(c,u)\nabla u\cdot\nabla \psi \,\dd x =-\int_\Omega (u-u^{k-1})\psi\,\dd x.
		\end{align}
		For $L\in \mathbb{N}$ consider $u_L:=\min\{u,L\}$.
		Then, in view of estimate~\eqref{eq:102b},~\eqref{eq:1010} and hp.~\ref{it:4}, we have $\int_\Omega a(c,u)|\nabla u_L|^2\dd x <\infty$, both in case~\HH\ and in case~\HHp.
		Using an approximation argument, one can now show that the above identity also holds for $\psi=u_L$. 
		(See for instance the proof of the $L^2$-energy identity~\eqref{eq:2.wlt} for a detailed argument in a related, but somewhat more complex situation.)
		Thus,
		\begin{align}\label{eq:1008}
		\tau\int_\Omega a(c,u)|\nabla u_L|^2\dd x +\int_\Omega uu_L\,\dd x=\int_\Omega u^{k-1}u_L\,\dd x.
		\end{align}
		Sending $L\to\infty$, we infer 
		\begin{align}\label{eq:1009}
		\tau\big\|a^\frac{1}{2}(c,u)\nabla u\big\|^2_{L^2} +\|u\|^2_{L^2} & \le \int_\Omega u^{k-1}u\,\dd x
		\le \tfrac{1}{2}\|u^{k-1}\|_{L^2}^2+\tfrac{1}{2}\|u\|_{L^2}^2,
		\end{align}
		which yields the asserted bound~\eqref{eq:121}.
	\end{proof}

	\section{Global weak solutions}
	\label{sec:time}
	
	We recall that $\mathcal{S}=\mathcal{S}_\delta$ (see eq.~\eqref{eq:Seps}),
	$\mathcal{P}=\mathcal{P}_\delta$, $\mathbb{M}=\mathbb{M}_\delta$ etc.\ (see the
	second paragraph of the introductory part in Section~\ref{sec:delta}), and
	apply Proposition~\ref{prop:eps=0} with $R=R_\rreg$ defined by
	\begin{align}\label{eq:defRrho}
	R_\rreg(Z)=\frac{ R(Z)}{\rreg|R(Z)|+1}.
	\end{align}
	In order to emphasise the $\delta$-dependence of quantities like
	$\mathcal{S}(Z), \mathcal{P}(Z)$, we will add the subscript $(\cdot)_\delta$
	and write $\mathcal{S}_\delta(Z), \mathcal{P}_\delta(Z)$ etc.
	
	Given a vector $Z^{k-1}=(c^{k-1},u^{k-1})\in (L\log L)^I\times L^2$,
	$k\in\mathbb{N}$, with non-negative components and such that
	$\|\hat\sigma_{0,-}(u^{k-1})\|_{L^1}<\infty$, we let $(c^k,u^k)=Z^k$ denote the
	solution $Z$ of eq.~\eqref{eq:eps=0} (with $S=S_\delta$, $R=R_\rreg$) obtained
	in Prop.~\ref{prop:eps=0}. We also use the notation
	$W^k=(y^k,v^k)=-DS_\delta(c^k,u^k)$.  By construction and Lemma~\ref{l:energy},
	we have $(c^{k},u^{k})\in (L\log L)^I\times L^2$ with non-negative components.
	As a consequence of estimate~\eqref{eq:1010}, we further have
	$\|\hat\sigma_{0,-}(u^k)\|_{L^1}<\infty$.  Hence, given an initial datum
	$Z^0\in (L\log L)^I\times L^2$ with $\|\hat\sigma_{0,-}(u^0)\|_{L^1}<\infty$
	and a time step size $\tau>0$, this allows us to iteratively construct a
	sequence $(Z^k)_{k\in \mathbb{N}}$, where $Z^k_i\ge0$ for all $i$ and
	$Z^k\in (L\log L)^I\times L^2$ with $\|\hat\sigma_{0,-}(u^k)\|_{L^1}<\infty$
	for all $k\in\mathbb{N}$.  Our goal is now to send $\tau,\rreg,\delta\to0$ in
	order to construct a solution to the original time-continuous problem.
	
	\subsection{Uniform estimates}
	Summing estimate~\eqref{eq:1010} over the time steps from $k=1$ to $k=n\in\mathbb{N}$
	yields
	\begin{align}
	\epsilon_*\tau\sum_{k=1}^n\mathcal{P}_\delta(Z^k)\le \mathcal{S}_\delta(Z^n)-\mathcal{S}_\delta(Z^{0}). \label{eq:1011}
	\end{align}
	By estimate~\eqref{eq:121} we further have
	\begin{align}\label{eq:122}\qquad
	2\tau\sum_{k=1}^n\int_\Omega a_\delta(c^k,u^k)|\nabla u^k|^2\dd x+\int_\Omega (u^n)^2\,\dd x \le \int_\Omega (u^0)^2\,\dd x. 
	\end{align}
	Adding the previous two inequalities and recalling estimates~\eqref{eq:112ptw},~\eqref{eq:112ptwii} (with $S=S_0$, treating $\delta\lambda(u^{0})\lesssim 1+(u^0)^2$ separately), we deduce upon absorption 
	\begin{equation}
	\begin{split}
	\sup_{n\in\mathbb{N}}\bigg(\sum_{i=1}^I&\int_\Omega c_i^n\log c_i^n\,\dd x+\|u^n\|_{L^2}^2+\|\hat\sigma_{0,-}(u^n)\|_{L^1}\bigg)
	+
	\epsilon_*\tau\sum_{k=1}^\infty\mathcal{P}_\delta(Z^k)\\&+\tau\sum_{k=1}^\infty\int_\Omega a_\delta(c^k,u^k)|\nabla u^k|^2\dd x
	\le C\bigg(\sum_{i=1}^I\|c^0\|_{L\log L}, \|u^0\|_{L^2},\|\hat\sigma_{0,-}(u^0)\|_{L^1}\bigg). 
	\end{split}
	\label{eq:1012}
	\end{equation}
	To simplify notation, we henceforth abbreviate $$\mathcal{N}(Z^0):=\sum_{i=1}^I\|c^0_i\|_{L\log L}+ \|u^0\|_{L^2}+\|\hat\sigma_{0,-}(u^0)\|_{L^1}.$$
	
	\begin{lemma} \label{l:tmpdq}
		Given $T\in(0,\infty)$ and $N\in \mathbb{N}$ suppose that $\tau =T/N$. Let further  $\delta\in(0,1]$ and $s=\tfrac{2d+2}{2d+1}$. Then there exists a finite constant $C=C(\mathcal{N}(Z^0))$ (depending also on $\epsilon_*$ and $\kappa_{j,i}$) such that
		\begin{align}\label{eq:201}
		\tau\sum_{k=1}^N\|A_\delta(Z^k)\nabla Z^k\|_{L^s(\Omega)}^s & \le C(\mathcal{N}(Z^0))(1+T).
		\end{align}
	\end{lemma}
	\begin{proof}Below we abbreviate $\tfrac{1}{s'}+\tfrac{1}{s}=1$, $q=\tfrac{2s}{2-s}$,
		$\theta=\tfrac{1-\frac{1}{q}}{\frac{1}{2}+\frac{1}{d}}$.
		By Lemma~\ref{l:efl} (with $A=A_\delta$), hp.~\ref{it:4}, estimate~\eqref{eq:1012} and the Gagliardo--Nirenberg inequality (Lemma~\ref{l:GN} with $q,\theta$ as defined here and $p=1$),
		\begin{align}\label{eq:105*}
		\|A_\delta(Z^k)\nabla Z^k\|_{L^s(\Omega)}\le (\mathcal{P}_\delta)^\frac{1}{2}(c^k,u^k)
		\big((\mathcal{P}_\delta)^\frac{\theta}{2}(c^k,u^k)+ \|\nabla u^k\|_{L^2}^\theta+1\big)C\big(\mathcal{N}(Z^0)\big).
		\end{align}
		When considering \HHp, GNS will be applied to $\sqrt{c_i}$ instead of $c_i$, and thanks to the time-uniform $L^1(\Om)$ control of $c_i$ we could have chosen $p=2$ in Lemma~\ref{l:GN}, thus somewhat improving the last estimate. Since this would only lead to a minor improvement of the integrability of the flux term in case~\HHp, we content ourselves with this somewhat suboptimal bound.
		\\
		Since $\tfrac{s(1+\theta)}{2}=1$, 
		the previous estimate yields
		\begin{align}\label{eq:205}
		\|A_\delta(Z^k)\nabla Z^k\|_{L^s(\Omega)}^s\le C(\mathcal{N}(Z^0))\bigg(\mathcal{P}_\delta(c^k,u^k)
		+\|\nabla u^k\|_{L^2}^2+1\bigg).
		\end{align}
		Taking the discrete time integral up to time $T$ and using once more~\eqref{eq:1012}, we infer~\eqref{eq:201}.
	\end{proof}
	
	\subsection{Limit \texorpdfstring{$(\tau,\rreg,\delta)\to0$}{}}
	\label{su:ProMainThm}
	
	\begin{proof}[Proof of Theorem~\ref{thm:main}]
		To pass to the limit $\tau\to0$ we can follow the approach in~\cite{ChenLiu_2013_globalWeak,Juengel_2015_boundednes-by-entropy,ChenDausJuengel_2018_global,ChenJuengel_2019_renormalised}. 	
		Given the sequence $(Z^k)_{k\in \mathbb{N}}$ constructed above, we define the piecewise constant interpolant function 
		\begin{align*}
		Z^{(\tau)}(t,\cdot) = Z^k\qquad\text{ if }t\in((k-1)\tau,k\tau].
		\end{align*}
		We further let $(c^{(\tau)},u^{(\tau)}):=Z^{(\tau)}$, $W^{(\tau)}=-DS(Z^{(\tau)})$ and define the discrete time derivative 
		\begin{align*}
		\partial^{(\tau)}_tZ^{(\tau)}(t,\cdot)=\tfrac{1}{\tau}(Z^k-Z^{k-1})\qquad\text{ if }t\in((k-1)\tau,k\tau].
		\end{align*}
		
		Let now $T\in (0,\infty)$ be fixed but arbitrary.
		Then, by~\eqref{eq:1012},
		\begin{equation}\label{eq:1014}
		\begin{multlined}
		\|c^{(\tau)}\|_{L^\infty(0,T;L\log L)}+\|u^{(\tau)}\|_{L^\infty(0,T;L^2)}+\|\hat\sigma_{0,-}(u^{(\tau)})\|_{L^\infty(0,T;L^1)}
		\\+\int_{(0,T)}\mathcal{P}_\delta(Z^{(\tau)}(t,\cdot))\,\dd t
		+\|\sqrt{a_\delta(Z^{(\tau)})}\nabla u^{(\tau)}\|^2_{L^2(\Omega_T)}
		\le  C(\mathcal{N}(Z^0)).
		\end{multlined}
		\end{equation}
		Observe that since $a_\delta\gtrsim 1$ (cf.~\ref{it:2} resp.~\ref{it:pi.m3}),
		estimate~\eqref{eq:1014} implies the bound
		\begin{align}\label{eq:210}
		\|\nabla u^{(\tau)}\|^2_{L^2(\Omega_T)}\le  C(\mathcal{N}(Z^0)).
		\end{align}
		
		Next, for $s$ as in Lemma~\ref{l:tmpdq}, we have by~\eqref{eq:201}
		\begin{align}
		\|A_\delta(Z^{(\tau)})\nabla Z^{(\tau)}\|_{L^s(\Omega_T)} & \le C(\mathcal{N}(Z^0),T).\label{eq:1019}
		\end{align}	
		In the remaining reasoning, we need to distinguish the cases~\HH\ and~\HHp. We provide the details only in the case~\HH, and then briefly describe how to modify the arguments to deal with case~\HHp. Thus, let us first suppose~\HH. 
		
		In order to control the reaction term, 
		we need to upgrade the space-time integrability of $(c,u)$. 
		By the Gagliardo--Nirenberg inequality (Lemma~\ref{l:GN} with $p=1$, $q=\ol q_1=2+\tfrac{2}{d}$, $\theta$ as in~\eqref{eq:103} $\leadsto \ol q_1\theta=2$), we have
		for $c=c^{(\tau)}$
		\begin{equation}
		\begin{split}
		\|c\|^{\ol q_1}_{L^{\ol q_1}(\Omega_T)}&\le C_1\int_0^T\|\nabla c\|_{L^2(\Omega)}^{\ol q_1\theta}\|c\|_{L^1(\Omega)}^{\ol q_1(1-\theta)}\,\dd t+C_2T\|c\|_{L^\infty(L^1)}^{\ol q_1}
		\\&\le C(\mathcal{N}(Z^0),T).
		\end{split}
		\label{eq:214}
		\end{equation}
		Similarly, applying Lemma~\ref{l:GN} with $p=2$, $q=\ol q_2=2+\tfrac{4}{d}$, $\theta$ as in~\eqref{eq:103} (so that again $\ol q_2\theta=2$), we find for 
		$u=u^{(\tau)}$
		\begin{equation}
		\begin{split}
		\|u\|^{\ol q_2}_{L^{\ol q_2}(\Omega_T)} &\le C_1\int_0^T\|\nabla u\|_{L^2(\Omega)}^{{\ol q_2}\theta}\|u\|_{L^2(\Omega)}^{{\ol q_2}(1-\theta)}\,\dd t+C_2T\|u\|_{L^\infty(L^1)}^{\ol q_2}
		\\&\le C(\mathcal{N}(Z^0),T).
		\end{split}
		\label{eq:215}
		\end{equation}
		
		We next assert that for $r:=\min\Big\{\tfrac{2+\frac{2}{d}}{q_1},\tfrac{2+\frac{4}{d}}{q_2},s\Big\}=\min\Big\{\tfrac{\ol q_1}{q_1},\tfrac{\ol q_2}{q_2},s\Big\}>1$ and $r'$ given by
		$\tfrac{1}{r'}+\tfrac{1}{r}=1$,
		\begin{align}
		\|\partial^{(\tau)}_tZ^{(\tau)}\|_{L^r(0,T;(W^{1,r'}(\Omega))^*)} & \le  C(\mathcal{N}(Z^0),T). \label{eq:107}
		\end{align}
		
		To show~\eqref{eq:107}, we take $\phi\in W^{1,r'}(\Omega)$ with $\|\phi\|_{W^{1,r'}(\Omega)}\le 1$. The choice of $r$  implies that $r'>d$ and hence $W^{1,r'}(\Omega)\hookrightarrow L^\infty(\Omega)$.
		Thus, for any $k\in\mathbb{N}$, 
		\begin{align}\nonumber
		\big|\int_\Omega\tau^{-1}(Z^k-Z^{k-1})\cdot\phi\,\dd x\big| & \lesssim 	\|A_\delta(Z^k)\nabla Z^k\|_{L^s(\Omega)}+\|R_\rreg(Z^k)\|_{L^1(\Omega)}
		\end{align}
		and hence  $\|\partial^{(\tau)}_tZ^{(\tau)}\|_{(W^{1,r'}(\Omega))^*}\lesssim 	\|A_\delta(Z^{(\tau)})\nabla Z^{(\tau)}\|_{L^s(\Omega)}+\|R_\rreg(Z^{(\tau)})\|_{L^1(\Omega)}$.
		Estimate~\eqref{eq:107} then follows upon taking the $L^r_t$ norm and 
		recalling~\eqref{eq:1019}, \eqref{eq:214} and~\eqref{eq:215}.
		
		We can therefore apply the Aubin--Lions lemma in the version of \cite[Theorem~1]{DreherJuengel_2012_compactFamilies} for any $T<\infty$. Choosing a sequence  $T\to\infty$ and using a diagonal argument, then allows one to infer the existence of a sequence  $(\tau,\rreg,\delta)\to0$ and
		\begin{align}
		& c\in L^\infty(0,\infty;L\log L) \text{ with }\nabla c\in L^2((0,\infty)\times\Omega),
		\\&u\in L^\infty(0,\infty;L^2)\text{ with }\nabla u\in L^2((0,\infty)\times\Omega)\label{eq:211}
		\end{align}
		such that for any $T>0$
		\begin{align}
		& (c^{(\tau)},u^{(\tau)}) \to (c,u) \quad\text{ in }L^{2}(\Omega_T),\label{eq:1013}
		\\&(\nabla c^{(\tau)},\nabla u^{(\tau)})\rightharpoonup(\nabla c,\nabla u)\quad\text{ in }L^2(\Omega_T),\label{eq:1013b}
		\\&(c^{(\tau)},u^{(\tau)})\to (c,u)\quad \text{a.e. in }\Omega_T,\label{eq:1013d}
		\end{align}
		\begin{align}\label{eq:204}\qquad
		\lim_{\tau\to0}\int_0^T\!\!\int_\Omega \partial^{(\tau)}_tZ^{(\tau)}\cdot\phi \,\dd x\,\dd t
		= -\int_0^T\!\!\int_\Omega Z\cdot\partial_t\phi \,\dd x\,\dd t-\int_\Omega Z^0\cdot\phi(0,\cdot)\,\dd x\qquad 
		\end{align}
		for any $\phi\in C^\infty_c([0,T)\times\ol\Omega)$, and 
		\begin{align}\label{eq:203}
		\partial^{(\tau)}_tZ^{(\tau)} \overset{\ast}{\rightharpoonup} \partial_tZ\quad\text{ in }L^r(0,T;(W^{1,r'})^*),
		\end{align}
		where the last two assertions are obtained as in~\cite[p.~2792f.]{ChenLiu_2013_globalWeak}.
		
		Arguing as in the proof of Prop.~\ref{prop:eps=0} (using in addition~\eqref{eq:1019}), one further has
		\begin{align}
		A_\delta(Z^{(\tau)})\nabla Z^{(\tau)}\rightharpoonup A(Z)\nabla Z\quad\text{ in }L^s(\Omega_T).\label{eq:1013f}
		\end{align}
		Also, by~\eqref{eq:1014} and~\eqref{eq:1013},
		$\nabla \sqrt{c_i^{(\tau)}}\rightharpoonup\nabla \sqrt{c_i}\text{ in }L^2(\Omega_T)$
		for all $i\in\{1,\dots,I\}$ with $\kappa_{0,i}>0$,
		and (using also~\eqref{eq:1013f} and Lemma~\ref{l:weakConv})
		$\sqrt{a_\delta(Z^{(\tau)})}\nabla u^{(\tau)}\rightharpoonup \sqrt{a(Z)}\nabla u\quad\text{ in } L^2(\Omega_T).$
		
		Concerning the reaction rates, the pointwise convergence~\eqref{eq:1013d} combined with the continuity of $R$ implies that $R_\varrho(c^{(\tau)},u^{(\tau)})\to R(c,u)$ a.e.\ in $\Om_T$ along the chosen sequence $(\tau,\varrho,\delta)\to0$.
		At the same time, the uniform bounds~\eqref{eq:214}, \eqref{eq:215} and the growth control for $R$ under Hypothesis~\ref{mainhypo1} (in case~\HH{}) guarantee that the family $\{R_\varrho(c^{(\tau)},u^{(\tau)})\}_{(\tau,\varrho,\delta)}\subset L^1(\Om_T)$ is equi-integrable. As a consequence,
		\begin{align}
		R_\varrho(c^{(\tau)},u^{(\tau)})\to R(c,u)\quad \text{ in }L^1(\Om_T).
		\end{align}
		
		By~\eqref{eq:1014}, the above convergence results, and weak(-star) lower semi-continuity,
		we obtain in the limit $(\tau,\rreg,\delta)\to0$
		\begin{multline}
		\|c\|_{L^\infty(0,T;L\log L)}+\|u\|_{L^\infty(0,T;L^2)}+\|\hat\sigma_{0,-}(u)\|_{L^\infty(0,T;L^1)}
		\\+\|a^\frac{1}{2}(c,u)\nabla u\|_{L^2(\Omega_T)}^2+\int_0^T\mathcal{P}(Z)\,\dd t
		\le C(\mathcal{N}(Z^0)).\label{eq:1014*}
		\end{multline}
		Here, the bound for $\|c\|_{L^\infty(0,T;L\log L)}$ and $\|\hat\sigma_{0,-}(u)\|_{L^\infty(0,T;L^1)}$ was obtained using~\eqref{eq:1013d} and Fatou's lemma.
		
		To infer equation~\eqref{eq:00}, we sum eq.~\eqref{eq:eps=0} (with $Z=Z^k, S=S_\delta$ and $R=R_\rreg$) from $k=1$ to $k=N$ (where $\tau N=T$). The resulting equation can be written in the form
		\begin{equation}
		\begin{split}
		\int_0^T\!\!\int_\Omega A_\delta(Z^{(\tau)}) \nabla Z^{(\tau)}:\nabla \phi \,\dd x\dd t =-\int_0^T\!\!\int_\Omega& \partial_t^{(\tau)}Z^{(\tau)}\cdot\phi\,\dd x\dd t\\&+\int_0^T\!\!\int_\Omega R_\rreg(Z^{(\tau)})\cdot\phi'\,\dd x\dd t.
		\end{split}
		\label{eq:111}
		\end{equation}
		Since this equation holds true for any $T'\in(0,T]$, the density of functions in $L^{r'}(0,T;W^{1,r'}(\Omega))$ which are piecewise constant in time (see~\cite[Prop.~1.36]{Roubicek_2013_npde}) allows us to extend eq.~\eqref{eq:111} to time-dependent test functions $\phi=(\phi',\phi_{I+1})\in  L^{r'}(0,T;W^{1,r'}(\Omega))^{I+1}$.
		Sending $(\tau,\rreg,\delta)\to0$, we then obtain in particular for all $\phi\in L^\infty(0,T;W^{1,\infty}(\Omega))^{I+1}$ the equation
		\begin{align}\label{eq:112}
		\int_0^T\langle\partial_tZ,\phi\rangle\,\dd t+\int_0^T\!\!\int_\Omega A(Z) \nabla Z:\nabla \phi \,\dd x\dd t = 
		\int_0^T\!\int_\Omega R(Z)\cdot\phi'\,\dd x\dd t,
		\end{align}  
		where we used the convergence properties established before.
		Since $W^{1,r}(0,T;(W^{1,r'})^*)\hookrightarrow C([0,T];(W^{1,r'})^*)$, we have $Z\in C([0,T];(W^{1,r'})^*)$ and~\eqref{eq:204},~\eqref{eq:203} then imply that $Z(t=0,\cdot)=Z^0$ in $(W^{1,r'}(\Omega))^*$ and thus in particular in the $(W^{1,\infty}(\Omega))^*$-sense.

		It remains to prove the conservation of the internal energy.
		Set $V:=W^{1,r'}(\Omega)$ and let $T<\infty$ be arbitrary. The fact that $\partial_tu\in L^{r}(0,T;V^*)$ (where $\partial_tu$ is understood as the distributional derivative of the Bochner function $u$) implies that for any $\psi\in V$ the function 
		$t\mapsto (u(t,\cdot),\psi)_{L^2(\Omega)}=\langle u(t,\cdot),\psi\rangle_{V^*,V}$ is absolutely continuous and for all $0\le t_1\le t_2\le T$
		\begin{align}
		(u(t_2,\cdot),\psi)_{L^2(\Omega)}-(u(t_1,\cdot),\psi)_{L^2(\Omega)}=\int_{t_1}^{t_2}\langle \partial_tu(t,\cdot),\psi\rangle_{V^*,V}\,\dd t.
		\label{eq:212}
		\end{align}
		On the other hand, choosing $\phi=(0,\dots,0,\phi_{I+1})$ with $\phi_{I+1}(t,\cdot)\equiv\chi_{[t_1,t_2]}(t)$ in equation~\eqref{eq:112} shows that the RHS of~\eqref{eq:212} vanishes for $\psi\equiv1$.  Hence, for any $t_2>0$
		$(u(t_2,\cdot),1)_{L^2(\Omega)} =(u(0,\cdot),1)_{L^2(\Omega)}$, which concludes the proof of Theorem~\ref{thm:main} under hypotheses~\HH.
		\medskip
		
		Let us now sketch the necessary modification under hypotheses~\HHp.
		We first note that, like in the derivation of~\eqref{eq:214},~\eqref{eq:215}, we may appeal to the 
		Gagliardo--Nirenberg inequality (cf.\ Lemma~\ref{l:GN}) to infer that
		\begin{align}
		\|c^{(\tau)}\|_{L^{\tilde q_1}(\Omega_T)} + \|u^{(\tau)}\|_{L^{\tilde q_2}(\Omega_T)}
		\le C(T,\mathcal{N}(Z^0))
		\end{align}
		for the exponents $\tilde q_1=1+\tfrac{2}{d}$ and $\tilde q_2=2+\tfrac{4}{d}$. Here, we have also used the bound~\eqref{eq:1014}, which provides us among others with $(\tau,\varrho,\delta)$-uniform bounds for $\nabla\sqrt{c_i^{(\tau)}}$ and $\nabla u^{(\tau)}$ in $L^2(\Om_T)$.
		Since $\kappa_{1,i}=0$, only $\nabla\sqrt{c_i^{(\tau)}}$ (but not $\nabla c_i^{(\tau)}$) is controlled in $L^2(\Om_T)$.
		In order to obtain compactness, we therefore have to apply a nonlinear version of the Aubin--Lions lemma. Such a result covering our situation has been provided in~\cite{CJL_2014}. 
		Under the growth hypotheses on the reactions for model~\HHp\ (cf.\ Hypothesis~\ref{mainhypo1}) and the flux bound~\eqref{eq:1019}, 
		we easily obtain a $(\tau,\rreg,\delta)$-uniform bound of the form
		\begin{equation*}
		\|\sqrt{c_i^{(\tau)}}\|_{L^2(0,T;H^1(\Om))}+\|\partial_tc_i^{(\tau)}\|_{L^1(0,T;H^n(\Om)^*)} \le C(T,\mathrm{data})
		\end{equation*}
		for sufficiently large $n\in \mathbb{N}$ (using the embedding $W^{1,s'}(\Om)^*\hookrightarrow H^n(\Om)^*$, which holds true e.g.\ for $n>\tfrac{d}{2}+1$). Applying~\cite[Theorem~3]{CJL_2014}, we deduce 
		that, along a subsequence,   $c^{(\tau)}_i\to c_i$ in $L^1(0,T;L^{1+\epsilon}(\Om))$ for all $\epsilon\in[0,\tfrac{1}{(1-\frac{2}{d})_+})$ and a.e.\ in~$\Om_T$.
		The remaining arguments are as before.
	\end{proof}

	\section{Preliminaries for Theorem~\ref{theoremexistence}}
	\label{sec:prelim2}
	
	From now on, we suppose Hypotheses~\ref{mainhypo2} (in place of
	Hypotheses~\ref{mainhypo1}).  We then let $\varepsilon,\rreg>0$ denote small
	positive parameters, consider the Onsager matrix
	\begin{align}\label{eq:mobility2reg"}
	\mathbb M^\varepsilon(c, u) := \diag \big( m_1^\varepsilon,\dots, 
	m_I^\varepsilon,0 \big) + \pi_{1}(Z)\mu\otimes\mu,
	\end{align}
	where $m_i^\ve=c_ia_i^\ve(c,u)$, $a_i^\ve(c,u)\sim \ve c_i+\kappa_{0,i}$,
	satisfy~\ref{it:mi1} with $\kappa_{1,i}=\ve>0$ for all $i\in\{1,\dots, I\}$,
	and choose the regularised reactions $R_\rreg(Z)=\tfrac{1}{1+\rreg |R(Z)|}R(Z)$
	(cf.~\eqref{eq:defRrho}).  Since the choice $\mathbb{M}=\mathbb M^\varepsilon$,
	$R=R_\rreg$ clearly satisfies Hypotheses~\ref{mainhypo1} (with $q_1=q_2=0$),
	Theorem~\ref{thm:main} provides us with a family of functions $Z=(c,u)$
	(depending on $\varepsilon,\rreg>0$) with the regularity
	\begin{itemize}
		\item[] $A^\varepsilon(Z)\nabla Z\in L^s((0,\infty)\times\Om)$ for $s \coleq \tfrac{2d+2}{2d+1}$, \;
		where $A^\varepsilon:=-\mathbb M^\varepsilon D^2S_0$,
		\item[] $\partial_tZ\in L^s_\mathrm{loc}(0,\infty; W^{1,s'}(\Omega)^*)$ for 
		$s' \coleq 2d + 2$
	\end{itemize}
	emanating from the initial datum $Z^0$ (in the $W^{1,\infty}(\Om)^*$ sense) that
	satisfy for all $T>0$ the weak formulation 
	\begin{equation}\label{eq:00.ii}
	\int_0^T\langle\partial_tZ,\phi\rangle\,\dd t
	+\int_0^T\!\!\int_\Omega(A^\varepsilon(Z)\nabla Z):\nabla \phi\,\dd x\dd t=\int_0^T\!\!\int_\Omega R_\rreg(Z)\cdot\phi'\,\dd x\dd t
	\end{equation}
	for any $\phi=(\phi',\phi_{I+1})\in L^\infty(0,T;W^{1,\infty}(\Om))^{I+1}$, enjoy the $\rreg$-independent bound
	\begin{equation}\label{equnifbounds}
	\begin{split}
	\|c\|_{L^\infty(0,\infty;L\log L)}& + \|u\|_{L^\infty(0,\infty;L^2)} \\
	&+\int_0^T\!\!\int_\Om P_\ve(Z)\,\dd x\dd t + \int_0^T\!\!\int_\Om a(Z)|\nabla u|^2\,\dd x\dd t
	\le C(\mathrm{data}),
	\end{split}
	\end{equation}
	where $\mathrm{data} := (\|c^0\|_{L\log L}, \|u^0\|_{L^2}, \|\hat\sigma(u^0)\|_{L^1})$, and where $P_\ve(Z)$ denotes the quantity $P(Z)$ defined in~\eqref{eq:401} with $\kappa_{1,i}=\ve$.
	These solutions conserve the internal energy, i.e.~$\int u(t,x)\,\dd x=\int u^0(x)\,\dd x$ for any $t>0$, and for later usage we further note the $(\ve,\rreg)$-uniform bound
	\begin{equation}\label{eq:ut2}
	\|\partial_tu\|_{L^s(0,T;W^{1,s'}(\Om)^*)}\le C(\mathrm{data}),
	\end{equation}
	which is a consequence of~\eqref{equnifbounds} and the bound~\eqref{eq:102b} on the heat flux.
	The functions $Z=Z^{\ve,\rreg}$ now serve as approximate solutions to the model considered in Theorem~\ref{theoremexistence}.
	
	In this section, we gather some technical tools and derive a renormalised formulation for the weak solutions $Z=(c,u)$. Throughout this section, $\varepsilon,\rreg>0$ will be kept fixed, and the dependence of  $Z$ and of $P(Z):=P_\ve(Z)$ on $\ve$ and $\rreg$ will not be explicitly indicated.

	\subsection{Truncation functions}
	
	As in \cite{ChenJuengel_2019_renormalised} and \cite{Fischer_2015_renormalized}, we employ the family of special cut-off functions
	\begin{equation}
	\label{eqphiie}
	\varphi_i^E(Z) \coleq Z_i \phi \bigg( \frac{\sum_{k=1}^{I+1} Z_k - E}{E} \bigg) + 3E \bigg( 1 - \phi \bigg( \frac{\sum_{k=1}^{I+1} Z_k - E}{E} \bigg) \bigg)
	\end{equation}
	for $i \in \{1, \dotsc, I\}$ and $Z \in \mathbb R^{I+1}$ where $\phi \in C^\infty(\mathbb R, [0,1])$ is a fixed non-increasing function with the property $\phi = 1$ on $(-\infty,0]$ and $\phi = 0$ on $[1,\infty)$.

	The following properties of $\varphi^E_i$ are easily verified.
	\begin{lemma}
		The truncations $\varphi_i^E(Z)$ fulfill the following conditions.
		\begin{enumerate}[label=\normalfont{(C\arabic*)}]
			\item\label{it:C1} $\varphi_i^E \in C^\infty((\mathbb R_0^+)^{I+1})$.
			\item\label{it:C2}
			There exists a constant $K_1 > 0$ such that for all $E>0$ and $Z \in (\mathbb R_0^+)^{I+1}$
			\begin{equation*}
			|Z||D^2\varphi^E(Z)|\le K_1.
			\end{equation*}
			\item\label{it:C3} For all $E>0$, the set $\supp D \varphi_i^E$ is compact.
			More specifically, $(D\varphi_i^E)(Z) = 0$ for all $Z\in \mathbb R^{I+1}$ satisfying $\sum_{j=1}^{I+1} Z_j \geq 2E$.
			\item\label{it:C4} For all $i,j \in \{1,\dotsc,I+1\}$ and $Z \in (\mathbb R_0^+)^{I+1}$, we have $\lim_{E \rightarrow \infty} \partial_j \varphi_i^E(Z) = \delta_{ij}$.
			\item\label{it:C5} There exists a constant $K_2 > 0$ such that $|D\varphi_i^E(Z)| \leq K_2$ holds true for all $E>0$ and $Z \in (\mathbb R_0^+)^{I+1}$.
			\item\label{it:C6} For all $Z \in (\mathbb R_0^+)^{I+1}$ with $\sum_{j=1}^{I+1} Z_j < E$, we have $\varphi_i^E(Z) = Z_i$.
			\item\label{it:C7} For all $j,k \in \{1,\dotsc,I+1\}$ and $K > 0$, we have $$\lim_{E \rightarrow \infty} \sup_{|Z| \leq K} |\partial_j \partial_k \varphi_i^E(Z)| = 0.$$
			\item\label{it:C8} For all $Z \in (\mathbb R_0^+)^{I+1}$, and $E \in \mathbb N$, we have
			\[
			\varphi_i^E(Z) \leq Z_i + 3 \sum_{j=1}^{I+1} Z_j.
			\]
			\item\label{it:C9} For all $E \geq E_0 \geq 0$ and $Z \in (\mathbb R_0^+)^{I+1}$, we have 
			\[
			\sum_{i=1}^I Z_i \geq E_0 \ \Rightarrow \ \sum_{i=1}^I \varphi_i^E(Z) \geq E_0.
			\]
		\end{enumerate}
	\end{lemma}
	
	\subsection{A weak chain rule for truncated solutions}
	
	The remaining part of this section is devoted to deriving an evolution equation for $\varphi^E_i(Z)$, as asserted in the following proposition. 
	\begin{proposition}
		\label{propchainrule}
		For all $T > 0$, $E \in \mathbb N$, $i \in \{1, \dotsc, I\}$, and $\psi \in C^\infty_c([0,T)\times\ol\Omega)$, 
		the weak solution $Z=(c, u)$ of eq.~\eqref{eq:00.ii} satisfies the following equation:
		\begin{align}\label{eqchainrule}   
		\begin{aligned}
		- \int_\Omega \varphi_i^E(Z^0) & \psi(0, \cdot) \, \dd x - \int_0^T\!\!\int_\Omega\!\!\varphi_i^E(Z) \frac{d}{dt} \psi \,\dd x\dd t                      \\
		=                              & - \sum_{j,k,l=1}^{I+1}\int_0^T\!\!\int_\Omega \psi \, \partial_j \partial_k \varphi_i^E (Z)A^\ve_{jl}(Z)\nabla Z_l\cdot\nabla Z_k\,\dd x\dd t
		\\& - \sum_{j=1}^{I+1} \int_0^T\!\!\int_\Omega \partial_j \varphi_i^E (Z) A^\ve_{jl}(Z)\nabla Z_l\cdot\nabla \psi\,\dd x\dd t
		\\& + \sum_{j=1}^I \int_0^T\!\!\int_\Omega \psi \, \partial_j \varphi_i^E (Z)R_{\rreg,j}(Z)\,\dd x\dd t.
		\end{aligned}
		\end{align}
	\end{proposition}

	 The proof of this proposition is based on the techniques used in the proof of \cite[Lemma~11]{ChenJuengel_2019_renormalised}, which in turn employs the approximate chain rule established in \cite[Lemma~4]{Fischer_2015_renormalized} and stated below in a form adapted to our  notations. While in the present section we only use a special case of this result, the more general version formulated below will be needed in Section~\ref{ssec:proof.renorm} for the proof of Theorem~\ref{theoremexistence}.
	
	\begin{lemma}[{\cite[p.~578, Lemma~4]{Fischer_2015_renormalized}}]
		\label{lemmachainrulefischer}
		Let $\Omega \subset \mathbb R^d$ be a bounded domain with Lipschitz
		boundary. Let $T > 0$, $v \in L^2(0,T; H^1(\Omega)^{I+1})$, and
		$v_0 \in L^1(\Omega)^{I+1}$. Moreover, let
		$\nu_j \in \mathcal M([0,T) \times \ol \Omega)$ be a Radon measure,
		$w_j \in L^1(0,T; L^1(\Omega))$, and $z_j \in L^2(0,T; L^2(\Omega)^d)$ for
		$j \in \{1, \dotsc, I+1\}$.
		
		Suppose that $v$ satisfies for all $\psi \in C^\infty_c([0,T)\times\ol\Omega)$
		and all $j \in \{1, \dotsc, I+1\}$ the identity \vspace*{-1ex}
		\begin{multline}
		\label{eqchainrulefischerassump}
		-\int_0^T \int_\Omega v_j \frac{d}{dt} \psi \, \dd x \dd
		t  - \int_\Omega (v_0)_j \psi(0, \cdot) \, \dd x \\
		= \int_{\ol \Omega \times [0,T)} \psi \, \dd\nu_j + \int_0^T \int_\Omega \psi
		\, w_j \, \dd x \dd t + \int_0^T \int_\Omega z_j \cdot \nabla \psi \, \dd x
		\dd t.
		\end{multline}
		Let $\xi : \mathbb R^{I+1} \rightarrow \mathbb R$ be a smooth function with
		compactly supported first derivatives. Then, there exists a constant
		$C(\Omega) > 0$ such that for all $\psi \in C^\infty_c([0,T)\times\ol\Omega)$,
		\begin{align*}
		\bigg| &-\int_0^T \int_\Omega \xi(v) \frac{d}{dt} \psi \, \dd x \dd t 
		- \int_\Omega \xi(v_0) \psi(0, \cdot) \, \dd x 
		- \sum_{j=1}^{I+1} \sum_{k=1}^{I+1} \int_0^T \int_\Omega \psi \partial_j 
		\partial_k \xi (v) z_j \cdot \nabla v_k \, \dd x \dd t \\
		&- \sum_{j=1}^{I+1} \int_0^T \int_\Omega \partial_j \xi (v) z_j \cdot \nabla
		\psi \, \dd x \dd t - \sum_{j=1}^{I+1} \int_0^T \int_\Omega 
		\psi \partial_j \xi (v) w_j \, \dd x \dd t \bigg| \\
		&\qquad\qquad \leq C(\Omega) \| \psi \|_{L^\infty} \sup_{\tilde v \in \mathbb R^{I+1}} 
		|D\xi(\tilde v)| \, \sum_{j=1}^{I+1} \| \nu_j \|_{\mathcal M([0,T) \times \ol \Omega)}.
		\end{align*}
	\end{lemma}

Unfortunately, Lemma~\ref{lemmachainrulefischer} is not directly applicable for our proof of Prop.~\ref{propchainrule}, and we need to repeat the arguments from the proof of Lemma~\ref{lemmachainrulefischer} in~\cite{Fischer_2015_renormalized}, slightly extending them. We only provide the argument away from the boundary, as the argument near the boundary is analogous but technical. In the course of the proof of Prop.~\ref{propchainrule}, we make use of the following standard chain-rule lemma (see e.g.\ \cite[p.~583, Lemma 5]{Fischer_2015_renormalized}). 

	\begin{lemma}[Cf.~{\cite{Fischer_2015_renormalized}}]
	\label{lemmachainlemmafischer}
	Let $\Omega \subset \mathbb R^d$ be a bounded domain. Let $T > 0$, $v \in L^1(0,T; L^1(\Omega)^{I+1})$, and
	$v_0 \in L^1(\Omega)^{I+1}$. Moreover, let
	$w_j \in L^1(0,T; L^1(\Omega))$ for $j \in \{1, \dotsc, I{+}1 \}$.
	
	Suppose that $v$ satisfies for all $\psi \in C^\infty_c([0,T)\times\Omega)$
	and all $j \in \{1, \dotsc, I{+}1 \}$ the identity 
	\[
	-\int_0^T \int_\Omega v_j \frac{d}{dt} \psi \, \dd x \dd
	t  - \int_\Omega (v_0)_j \psi(0, \cdot) \, \dd x \\
	= 	\int_0^T \int_\Omega \psi\, w_j \, \dd x \dd t.
	\]
	Let $\xi : \mathbb R^{I+1} \rightarrow \mathbb R$ be a smooth function with
	compactly supported first derivatives.
	Then, 
	\begin{align*}
		-\int_0^T \int_\Omega \xi(v) \frac{d}{dt} \psi \, \dd x \dd t 
		- \int_\Omega \xi(v_0) \psi(0, \cdot) \, \dd x - \sum_{j=1}^{I+1} \int_0^T \int_\Omega 
		\psi \partial_j \xi (v) w_j \, \dd x \dd t =0
	\end{align*}
 for all $\psi \in C^\infty_c([0,T)\times\Omega)$.
\end{lemma}
	
	Observe that in contrast to Lemma~\ref{lemmachainrulefischer} the test functions $\psi=\psi(t,\cdot)$ in Lemma~\ref{lemmachainlemmafischer} are compactly supported in $\Om$. This will be needed in the mollification-based regularisation procedure employed for proving Proposition~\ref{propchainrule}.
Before turning to the proof of this proposition, let us introduce the specific mollifier $\rho_\eta$ we are going to use.	
	\begin{definition}
		\label{defmollofier}
		Denote by $\wt \rho_\eta \in C_c^\infty(\mathbb R^d)$ the standard mollifier, i.e.
		\begin{align}
		\label{eqdefmollifier}
		\wt \rho_\eta (x) \coleq C \eta^{-d} \exp \bigg( \frac{1}{\eta^{-2} |x|^2 - 1} \bigg) \ \text{ for }|x| < \eta 
		\end{align}
		and $\wt \rho_\eta (x) \coleq 0$ for $|x| \geq \eta$, which satisfies $\int_{\mathbb R^d} \wt \rho_\eta(x) \, \dd x = 1$, and set $\rho_\eta \coleq \wt \rho_\eta \ast \wt \rho_\eta$. Further define 
		\[
		\Omega^\vartheta \coleq \{ x \in \Omega \ | \ \dist(x, \partial \Omega) > \vartheta\} \ \text{ for }\vartheta > 0.
		\]
	\end{definition}
	We will make use of the following basic properties satisfied by $\rho_\eta$ and $\wt \rho_\eta$.
	\begin{lemma}[\cite{ChenJuengel_2019_renormalised}, Lemma 10]
		\label{lemmamollifier}
		Let $\rho_\eta$ for $\eta > 0$ be the special mollifier from Definition \ref{defmollofier}, let $u$, $v$ be locally integrable functions on $\Omega$ where $\supp u \subset \Omega^{4\eta}$, and let $x \in \Omega^{3\eta}$ and $\wt x \in B(x,\eta)$. Then,
		\begin{align}
		& \int_\Omega u \, \rho_\eta \ast v \, \dd x = \int_{\Omega^{3\eta}} \big( \wt \rho_\eta \ast u \big) \big( \wt \rho_\eta \ast v \big) \, \dd x, \label{eqmollifierdistribute} \\
		& \int_{B(x,3\eta)} \rho_\eta(\wt x - y) u(y) \, \dd y = \big( \rho_\eta \ast u \big) (\wt x), \qquad \int_{B(x,3\eta)} \rho_\eta (\wt x - y) \, \dd y = 1.
		\end{align}
	\end{lemma}
 We are now in a position to prove Proposition~\ref{propchainrule}.
	\begin{proof}[Proof of Proposition~\ref{propchainrule}]
		As in \cite{ChenJuengel_2019_renormalised} we first observe that
		for any $j\in\{1,\dots,I+1\}$ the expression $\int_0^T\langle\partial_tZ_j,\psi\rangle\,\dd t$ in~\eqref{eq:00.ii} (with $\phi_j=\psi$ vanishing at $t=T$) can be rewritten in the distributional form $- \int_0^T \int_\Omega Z_j \frac{\dd}{\dd t}\psi \,\dd x\dd t - \int_\Omega Z_{0,j} \psi(0, \cdot) \, \dd x $.
		Let us further recall the bound~\eqref{equnifbounds}, which implies among others the regularity $u \in L^2(0,T; H^1(\Omega))$ and $c_j \in L^2(0,T; H^1(\Omega))$ for all $j\in\{1,\dots,I\}$.
		We now specify for $j \in \{1,\dotsc,I\}$ in Lemma \ref{lemmachainrulefischer} the variables $v_j \coleq c_j \in L^2(0,T; H^1(\Omega))$,
		$v_{I+1} \coleq u \in L^2(0,T; H^1(\Omega))$, $(v_0)_j \coleq c_{0,j} \in L^1(\Omega)$, $(v_0)_{I+1} \coleq u_{0} \in L^1(\Omega)$, $\xi \coleq \varphi_i^E$, $\nu_j \coleq 0$, $w_j \coleq R_{\rreg,j}(Z) \in L^1(0,T; L^1(\Omega))$, $w_{I+1} \coleq 0$,
		\[
		z_j \coleq -\sum_{l=1}^{I+1}A^\ve_{jl}(Z)\nabla Z_l 
		\in L^s(0,T; L^s(\Omega, \mathbb R^d)),
		\]
		and $z_{I+1} \coleq -A^\ve_{I+1,I+1}(Z) \nabla Z_{I+1} \in L^s(0,T; L^s(\Omega, \mathbb R^d))$. This choice ensures that \eqref{eqchainrulefischerassump} holds true as a consequence of \eqref{eq:00.ii}. In the case of the better regularity $z \in L^2(0,T; L^2(\Omega, \mathbb R^d))^{I+1}$, we directly arrive at \eqref{eqchainrule} by applying Lemma \ref{lemmachainrulefischer}.
		
		In the general case, we proceed as in \cite{ChenJuengel_2019_renormalised} and 
		\cite{Fischer_2015_renormalized}. In our situation, this amounts to choosing a partition of unity on $\ol \Omega$, where we first treat the case of a smooth $\psi$ being compactly supported in $[0,T) \times \Omega^{4\eta}$. We now take $\rho_\eta \ast \psi$ as a test function in \eqref{eqchainrulefischerassump} and integrate by parts the last term on the right-hand side. For $j \in \{1,\dotsc,I\}$, this results in 
		\vspace{-1ex}
		\begin{multline*}
		-\int_0^T \int_\Omega (\rho_\eta \ast c_j) \frac{d}{dt} \psi \, \dd x \dd
		t  - \int_\Omega (\rho_\eta \ast c_{0,j}) \psi(0, \cdot) \, \dd x \\
		= \int_0^T \int_\Omega \psi
		\, (\rho_\eta \ast R_{\rreg, j}(Z)) \, \dd x \dd t + \int_0^T \int_\Omega \psi \, \divv \bigg( \rho_\eta \ast \sum_{l=1}^{I+1} A_{jl}^\varepsilon(Z) \nabla Z_l \bigg) \dd x
		\dd t,
		\end{multline*}
		while for $j = I+1$, we obtain 
		\vspace{-1ex}
		\begin{multline*}
		-\int_0^T \int_\Omega (\rho_\eta \ast u) \frac{d}{dt} \psi \, \dd x \dd
		t  - \int_\Omega (\rho_\eta \ast u_0) \psi(0, \cdot) \, \dd x \\
		= \int_0^T \int_\Omega \psi \, \divv \big(\rho_\eta \ast (A_{I+1,I+1}^\varepsilon(Z) \nabla Z_{I+1}) \big) \, \dd x
		\dd t.
		\end{multline*}
		This enables us to apply Lemma \ref{lemmachainlemmafischer} with the choices $\xi \coleq \varphi_i^E$, 
		\begin{align*}
		v_j     & \coleq \rho_\eta \ast c_j,                                  \qquad                                          
		v_{I+1}  \coleq \rho_\eta \ast u,                                                                           \\
		w_j     & \coleq \rho_\eta \ast R_{\rreg,j}(Z)
		+ \divv \bigg( \rho_\eta \ast \sum_{l=1}^{I+1} A^\ve_{jl}(Z) \nabla Z_l \bigg), \\
		w_{I+1} & \coleq \divv \big( \rho_\eta \ast ( A^\ve_{I+1,I+1}(Z) \nabla Z_{I+1} ) \big).
		\end{align*}
		Using the notation $Z = (c,u)$, integrating by parts, and keeping in mind that $A_{I+1,l} = 0$ for $l < I+1$, we arrive at 
		\begin{align} \label{eqchainrulemollified}
		\begin{aligned}
		-\int_0^T \int_\Omega & \varphi_i^E(\rho_\eta \ast Z) \frac{d}{dt} \psi \, \dd x \dd t 
		- \int_\Omega \varphi_i^E(\rho_\eta \ast Z^0) \psi(0, \cdot) \, \dd x \\
		=\ & - \sum_{j,k,l=1}^{I+1} \int_0^T \int_\Omega \psi \partial_j \partial_k \varphi_i^E (\rho_\eta \ast Z) \rho_\eta \ast (A^\ve_{jl}(Z) \nabla Z_l) \cdot \nabla (\rho_\eta \ast Z_k) \, \dd x \dd t \\ 
		& - \sum_{j,l=1}^{I+1} \int_0^T \int_\Omega \partial_j \varphi_i^E (\rho_\eta \ast Z) \rho_\eta \ast (A^\ve_{jl}(Z) \nabla Z_l) \cdot \nabla \psi \, \dd x \dd t \\ 
		& + \sum_{j=1}^{I} \int_0^T \int_\Omega \psi \partial_j \varphi_i^E (\rho_\eta \ast Z) \rho_\eta \ast R_{\rreg,j}(Z) \, \dd x \dd t.
		\end{aligned}
		\end{align}
		Passing to the limit $\eta \rightarrow 0$ in \eqref{eqchainrulemollified} is elementary for the left-hand side and the last two lines on the right-hand side. The convergence of the first term on the right-hand side is proven in Lemma \ref{lemmachainrulelimit} below, which settles \eqref{eqchainrule} for the case of test functions $\psi$ being compactly supported in $[0,T) \times \Omega$. The corresponding situation of test functions $\psi$ being compactly supported in $[0,T) \times U$ where $U$ is a coordinate patch of $\Omega$ at the boundary $\partial \Omega$ is treated as in \cite[p.~580--583]{Fischer_2015_renormalized} leading to analogous expressions.
	\end{proof}
	
	\begin{lemma}
		\label{lemmachainrulelimit}
		Let $Z=Z^{\ve,\rreg}$ denote the approximate solutions as introduced in the beginning of Section~\ref{sec:prelim2}, which satisfy in particular the weak formulation~\eqref{eq:00.ii}
		and obey the bound~\eqref{equnifbounds}.
		Let further $\psi\in C^\infty_c([0,T) \times \Omega^{4\eta})$,  with $\Omega^{4\eta}$ and $\rho_\eta = \wt \rho_\eta \ast \wt \rho_\eta$ as in Definition~\ref{defmollofier}. Then, we have 
		\begin{multline}
		\lim_{\eta\to0}\sum_{j,k,l=1}^{I+1}\int_0^T\int_\Omega \psi \, \partial_j \partial_k \varphi_i^E(\rho_\eta \ast Z) \, \rho_\eta \ast \big( A^\ve_{jl}(Z) \nabla Z_l \big) \cdot \nabla ( \rho_\eta \ast Z_k ) \, \dd x\dd t
		\\ =\sum_{j,k,l=1}^{I+1}\int_0^T\int_\Omega \psi \, \partial_j \partial_k \varphi_i^E (Z)A^\ve_{jl}(Z)\nabla Z_l\cdot\nabla Z_k\,\dd x\dd t.
		\end{multline}
	\end{lemma}
	\begin{proof}
		We start by rewriting the first line on the right-hand side of \eqref{eqchainrulemollified} using \eqref{eqmollifierdistribute} as
		\begin{align}
		& \sum_{j=1}^{I+1} \int_0^T\!\!\int_\Omega \psi \, \nabla \big( \partial_j \varphi_i^E(\rho_\eta \ast Z) \big) \cdot \Big( \rho_\eta \ast \sum_{l=1}^{I+1} A^\ve_{jl}(Z) \nabla Z_l \Big) \, \dd x\dd t                                                           \\
		& = \sum_{j=1}^{I+1} \int_0^T \int_{\Omega^{3\eta}} J^j_\eta (t,x) \cdot K^j_\eta (t,x) \, \dd x\dd t
		\end{align}
		where
		\begin{align}
		J_\eta^j(t,x)     & \coleq \wt \rho_\eta \ast \big( \psi \, \nabla \big( \partial_j \varphi_i^E(\rho_\eta \ast Z) \big) \big), \qquad
		K_\eta^j(t,x)     \coleq \wt \rho_\eta \ast \sum_{l=1}^{I+1} A^\ve_{jl}(Z) \nabla Z_l.
		\end{align}
		The limit $\eta \rightarrow 0$ is now performed by following the steps in \cite{ChenJuengel_2019_renormalised}.
		\paragraph{Step 1: Bound on $K_\eta^j$} We first calculate for $(t,x) \in [0,T) \times \Omega^{3\eta}$ by using Lemma \ref{l:efl}, the hypothesis \ref{it:H6}, and the Cauchy--Schwarz inequality
		\begin{align}
		|K_\eta^j(t,x)| & = \bigg| \int_{B(x,\eta)} \wt \rho_\eta(x-y) \Big( \sum_{l=1}^{I+1} A^\ve_{jl}(Z) \nabla Z_l \Big)(y) \, \dd y \bigg|                                                                                                                                       \\
		& \mqquad \lesssim \sum_{i=1}^I \int_{B(x,\eta)} \wt \rho_\eta(x-\cdot) \big( 1 + u + c_i \big) P^{\frac12}(Z) \, \dd y \\
		& \mqquad \leq C \eta^{-d} \bigg[ \bigg( \int_{B(x,\eta)} (1+u)^2 \, \dd y \bigg)^\frac12 + \sum_{i=1}^I \bigg( \int_{B(x,\eta)} c_i^2 \, \dd y \bigg)^\frac12 \bigg]                                                                                                                                           \\
		& \mquad\times \bigg( \int_{B(x,\eta)} P(Z) \, \dd y \bigg)^\frac12.
		\end{align}
		By reusing the arguments on page 5921 in \cite{ChenJuengel_2019_renormalised}, we obtain
		\begin{align}
		\eta^{-\frac{d}{2}} \bigg( \int_{B(x,\eta)} (1+u)^2 \, \dd y \bigg)^\frac12 & \leq C \eta^{1-\frac{d}{2}} \bigg( \int_{B(x,3\eta)} |\nabla u|^2 \, \dd y \bigg)^\frac12 \\
		& \quad + C \inf_{\wt x \in B(x,\eta)} |(\rho_\eta \ast (1+u))(\wt x)|,
		\end{align}
		and an analogous relation involving $c_k$. In particular,
		\begin{align}
		|K_\eta^j(t,x)| & \leq C \eta^{1-d} \int_{B(x,3\eta)} \big( |\nabla u|^2 + P(Z) \big) \, \dd y                                                                                            \\
		& \quad + C \eta^{-\frac{d}{2}} \Big( \int_{B(x,3\eta)} \big( |\nabla u|^2 + P(Z) \big) \, \dd y \Big)^\frac12 \inf_{\wt x \in B(x,\eta)} |(\rho_\eta \ast (1+Z))(\wt x)|
		\end{align}
		
		\paragraph{Step 2: Bound on $J_\eta^j$} The estimate on $J_\eta^j$ can be inferred in the same way as in \cite{ChenJuengel_2019_renormalised}; in particular, the bound on page 5923 therein entails:
		\[
		|J_\eta^j(t,x)| \leq C \min \bigg\{ \eta^{-1}, \eta^{-\frac{d}{2}} \Big( \int_{B(x,3\eta)} \Big( |\nabla u|^2 + P(Z) \Big) \, \dd y \Big)^\frac12 \bigg\}.
		\]
		
		\paragraph{Step 3: Bound on $J_\eta^j \cdot K_\eta^j$} For $(t,x) \in [0,T) \times \Omega^{3\eta}$, we first consider the case that $\sum_{k=1}^{I+1} (\rho_\eta \ast Z_k)(\wt x) \geq 2E$ holds true for all $\wt x \in B(x,\eta)$. Then, $\partial_j \varphi_i^E(\rho_\eta \ast Z) = 0$ on $B(x,\eta)$ and, hence, $J_\eta^j(t,x) = 0 = J_\eta^j(t,x)K_\eta^j(t,x)$. In the opposite case, there exists some $x^\ast \in B(x,\eta)$ with $\sum_{k=1}^{I+1} (\rho_\eta \ast Z_k)(x^\ast) < 2E$. As a result, we deduce
		\[
		\inf_{\wt x \in B(x,\eta)} |(\rho_\eta \ast Z)(\wt x)| \leq |(\rho_\eta \ast Z)(x^\ast)|
		< 2E.
		\]
		The estimates on $J_\eta^j(t,x)$ and $K_\eta^j(t,x)$ 
		now lead to
		\begin{align}
		\label{eqboundjetaketa}
		|J_\eta^j(t,x) \cdot K_\eta^j(t,x)| \leq C E \eta^{-d} \int_{B(x,3\eta)} \Big( |\nabla u|^2 + P(Z) \Big) \, \dd y.
		\end{align}
		
		\paragraph{Step 4: Bound on $J_0^j \cdot K_0^j$} The pointwise limits of $J_\eta^j$ and $K_\eta^j$ read as follows:
		\begin{align}
		J_0^j(t,x)     & \coleq \psi \, \nabla \big( \partial_j \varphi_i^E(Z) \big), \\
		K_0^j(t,x)     & \coleq \sum_{l=1}^{I+1} A^\ve_{jl}(Z) \nabla Z_l.
		\end{align}
		We infer from Lemma \ref{l:efl} that for all $(t,x) \in [0,T) \times \Omega^{3\eta}$,
		\begin{align*}
		  |J_0^j(t,x) \cdot K_0^j(t,x)|                                                                                                                            
		&  \leq C\bigg|\psi \sum_{k=1}^{I+1} \partial_j \partial_k \varphi_i^E(Z) \nabla Z_k \bigg| \sum_{i=1}^I (1 + u + c_i) P^\frac12(Z) \\
		 & \leq C(E) \Big( |\nabla u|^2 + P(Z) \Big),
		\end{align*}
		where the $E$-dependence of the constant $C(E)$ arises from estimating $1 + u + c_i \leq C(E)$ (note that $\partial_j \partial_k \varphi_i^E(Z)$ is nonzero only for $|Z| \lesssim E$). 
		
		\paragraph{Step 5: Limit $\eta \rightarrow 0$} We are now able to prove that
		\begin{align}
		\int_0^T \int_{\Omega^{3\eta}} \big( J_\eta^j \cdot K_\eta^j - J_0^j \cdot K_0^j  \big) (t,x) \, \dd x\dd t \rightarrow 0
		\end{align}
		for all $j \in \{1,\dotsc,I+1\}$ as $\eta \rightarrow 0$. Following \cite{ChenJuengel_2019_renormalised}, we decompose the domain of integration into the sets $\{|J_\eta^j \cdot K_\eta^j - J_0^j \cdot K_0^j| > \vartheta\}$ and $\{|J_\eta^j \cdot K_\eta^j - J_0^j \cdot K_0^j| \leq \vartheta\}$ for arbitrary $\vartheta > 0$, 
		and define
		\begin{align}
		g_\eta(t,y) \coleq \eta^{-d} \int_\Omega \chi_{\Omega^{3\eta}}(x) \chi_{B(x,3\eta)}(y) \chi_{\{|J_\eta^j \cdot K_\eta^j - J_0^j \cdot K_0^j| > \vartheta\}} (t,x) \, \dd x.
		\end{align}
		Using the fact that $\chi_{B(x,3\eta)}(y)=0$ whenever $|x-y|\ge3\eta$,
		it is easy to see that $g_\eta$ is $\eta$-uniformly bounded in $L^\infty((0,T)\times\Om)$.	
		We now infer from \eqref{eqboundjetaketa} the bound
		\begin{multline}
		\int_0^T \int_{\Omega^{3\eta}} | J_\eta^j(t,x) \cdot K_\eta^j(t,x) | \chi_{\{|J_\eta^j \cdot K_\eta^j - J_0^j \cdot K_0^j| > \vartheta\}}(t,x) \, \dd x\dd t \\
		\leq C E \int_0^T \int_{\Omega} g_\eta \big( |\nabla u|^2 + P(Z) \big) \, \dd y \dd t.
		\end{multline}
		Next, we calculate
		\[
		\int_0^T \int_\Omega |g_\eta(t,y)| \, \dd y \dd t \leq C \mathcal L^{d+1} \big( \{|J_\eta^j \cdot K_\eta^j - J_0^j \cdot K_0^j| > \vartheta\} \big).
		\]
		As the right-hand side tends to zero for $\eta \rightarrow 0$ (since $J_\eta^j \cdot K_\eta^j \rightarrow J_0^j \cdot K_0^j$ a.e.\ in $(0,T) \times \Omega$), we know that $g_\eta \rightarrow 0$ in $L^1((0,T) \times \Omega)$ and, hence, (up to a subsequence) $g_\eta \rightarrow 0$ a.e.\ in $(0,T) \times \Omega$. 
		Recalling the $\eta$-uniform boundedness of $g_\eta$ as well as the uniform bounds~\eqref{equnifbounds}, we further find
		\[
		g_\eta \big( |\nabla u|^2 + P(Z) \big) \leq C \big( |\nabla u|^2 + P(Z) \big) \in L^1((0,T) \times \Omega).
		\]
		Lebesgue's dominated convergence theorem now guarantees that for $\eta \rightarrow 0$,
		\begin{multline*}
		\int_0^T \int_{\Omega^{3\eta}} | J_\eta^j(t,x) \cdot K_\eta^j(t,x) | \chi_{\{|J_\eta^j \cdot K_\eta^j - J_0^j \cdot K_0^j| > \vartheta\}}(t,x) \, \dd x\dd t \\
		\leq C E \int_0^T \int_{\Omega} g_\eta \big( |\nabla u|^2 + P(Z) \big) \, \dd y \dd t \rightarrow 0.
		\end{multline*}
		
		We now apply Lebesgue's dominated convergence theorem to the parallel situation involving $J_0^j \cdot K_0^j$ for any $j \in \{1, \dotsc, I+1\}$. To this end, we use the convergence $|(J_0^j \cdot K_0^j)| \chi_{\{|J_\eta^j \cdot K_\eta^j - J_0^j \cdot K_0^j| > \vartheta\}} \rightarrow 0$ a.e.\ in $(0,T) \times \Omega$ for $\eta \rightarrow 0$ and the uniform bound $|(J_0^j \cdot K_0^j)| \chi_{\{|J_\eta^j \cdot K_\eta^j - J_0^j \cdot K_0^j| > \vartheta\}} \leq C(E) \big( |\nabla u|^2 + P(Z) \big) \in L^1((0,T) \times \Omega)$. This yields
		\[
		\int_0^T \int_{\Omega^{3\eta}} | J_0^j(t,x) \cdot K_0^j(t,x) | \chi_{\{|J_\eta^j \cdot K_\eta^j - J_0^j \cdot K_0^j| > \vartheta\}}(t,x) \, \dd x\dd t \rightarrow 0
		\]
		for $\eta \rightarrow 0$. Finally, we conclude that
		\begin{align}
		& \bigg| \int_0^T \int_{\Omega^{3\eta}} \big( J_\eta^j \cdot K_\eta^j - J_0^j \cdot K_0^j  \big) (t,x) \, \dd x\dd t \bigg|                                                                      \\
		& \quad \leq \int_0^T \int_{\Omega^{3\eta}} | (J_\eta^j \cdot K_\eta^j - J_0^j \cdot K_0^j)(t,x) | \chi_{\{|J_\eta^j \cdot K_\eta^j - J_0^j \cdot K_0^j| > \vartheta\}}(t,x) \, \dd x\dd t + C \vartheta \\
		& \quad \rightarrow C\vartheta
		\end{align}
		for $\eta \rightarrow 0$. This completes the proof since $\vartheta > 0$ can be chosen arbitrarily small.
	\end{proof}
	
	\section{Construction of renormalised solutions}\label{sec:renormalised}
	To simplify notation, we henceforth let $\rreg=\varepsilon$, where we recall that by $\rreg>0$ we have denoted the regularisation parameter associated with the reaction rates. The global weak solutions of $\dot Z = \divv(A^\ve(Z)\nabla Z)+R_\ve^\circ$, where $R_\ve^\circ:=(R_\ve, 0)^T$, introduced in Section~\ref{sec:prelim2} will henceforth be denoted by $Z^\ve$, and in this section we aim to study the limit $\ve\to0$.
	
	\subsection{Convergence of a subsequence}
	In the following lemma, we establish a fundamental compactness result, allowing to deduce weak and pointwise convergence of $(c^\ve,u^\ve)$ along a subsequence. As in~\cite[Lemma 2]{Fischer_2015_renormalized}, pointwise convergence of $c^\ve$ is obtained by applying a version of the Aubin--Lions lemma to $(\varphi_i^E(Z^\de))_\ve$, $i=1,\dots,I$, for every $E\in \mathbb{N}$.
	In contrast to~\cite{Fischer_2015_renormalized}, 
	this is, however, not sufficient to deduce the renormalised formulation for the limiting candidate obtained upon $\ve\to0$. 
	This issue is due to the strong cross-diffusion effects driven by gradients of the internal energy density, which result in a lack of a priori estimates of the first term on the RHS of eq.~\eqref{eqchainrule} that are uniform in $\varepsilon>0$ and $E\in \mathbb{N}$.
	The key to resolve this problem lies in a stability result,
	namely the strong convergence of $\nabla u^\de$ in $L^2(\Omega_T)$, which will be deduced from $L^2$-energy identities for the internal energy density (cf.~\eqref{eq:2.wt} and~\eqref{eq:2.wlt} below).
	
	\begin{lemma}
		\label{lemmasubsequence}
		Let Hypotheses~\ref{mainhypo2} hold true and denote by $Z^\de=(c^\de, u^\de)$ the global solution of eq.~\eqref{eq:00.ii} introduced above.
		Then, along a sequence $\de \searrow 0$, $Z^\de$ converges a.e.\ on $[0,\infty) \times \Omega$ to some limit $Z=(c,u)$ with $c_i \log c_i  \in L^\infty(0,\infty; L^1(\Omega))$ for all $i\in \{1, \dotsc, I\}$ and $u\in L^\infty(0,\infty; L^2(\Om)).$
		Furthermore, the weak convergence $\sqrt{c_i^\de} \rightharpoonup \sqrt{c_i}$ and the strong convergence $u^\de \rightarrow u$ in $L^2(0,T; H^1(\Omega))$ hold true for all $T>0$.
	\end{lemma}
	\begin{proof}
		In a first step, we recall that due to \eqref{equnifbounds}--\eqref{eq:ut2}, $u^\de$ is uniformly bounded in $L^2(0, T; H^1(\Omega))$, while $\partial_t u^\de$ is uniformly bounded in $L^s(0, T; W^{1,s'}(\Omega)^\ast)$. Employing the Aubin--Lions Lemma then gives rise to a subsequence $(u^\varepsilon)_\varepsilon$ converging strongly in $L^2(0, T; L^2(\Omega))$ and pointwise a.e.\ in $\Omega_T$ to some $u \in L^2(0, T; H^1(\Omega))$. The additional regularity $u \in L^\infty(0, \infty; L^2(\Omega))$ follows with Fatou's Lemma.
		
		We now keep $E$ fixed and aim to apply the Aubin--Lions Lemma as stated in \cite[Corollary 7.9]{Roubicek_2013_npde} for proving the existence of a subsequence $(\varphi_i^E(Z^\de))_\ve$ converging strongly in $L^2(0,T; L^2(\Omega))$.
		To this end, we will verify that $\varphi_i^E(Z^\de)$ is $\ve$-uniformly bounded in the Bochner space $L^2(0,T; H^1(\Omega))$ and that the distributional time derivative $\tfrac{\dd}{\dd t} \varphi_i^E(Z^\de)$ is $\ve$-uniformly bounded in the space of Radon measures $\mathcal M([0,T], H^p(\Omega)^*)$ if $p > d/2 + 1$.
		
		The uniform boundedness of $\varphi_i^E(Z^\de)$ in $L^2(0,T; H^1(\Omega))$ follows from the boundedness of $\varphi_i^E$ itself, the $\ve$-uniform bound on $\nabla \sqrt{c_i^\de}$ in $L^2(\Om_T)$ and the identity
		\[
		\nabla \big( \varphi_i^E (Z^\de) \big) = \sum_{j=1}^{I+1} \partial_j \varphi_i^E (Z^\de) \nabla Z_j^\de.
		\]
		In fact, keeping in mind that $\supp D \varphi_i^E$ is a compact subset of $\mathbb R^{I+1}$, we see that $\varphi_i^E(Z^\de)$ is uniformly bounded w.r.t.\ $\ve$ in $L^2(0,T; H^1(\Omega))$ for all fixed $T>0$, $E>0$, and $i \in \{1,\dots,I\}$.
		
		We are left to establish an $\ve$-uniform bound on the time derivative of $\varphi_i^E(Z^\de)$ in $\mathcal M([0,T], H^p(\Omega)^*)$, the topological dual space of~$C^0([0,T], H^p(\Omega))$. For this purpose, we first extend \eqref{eqchainrule} in Proposition \ref{propchainrule} to allow for test functions $\psi \in C^\infty([0,T] \times \ol \Omega)$, which do not necessarily vanish at time $t = T$. This is achieved by a standard approximation procedure carried out on pages 5926--5927 in \cite{ChenJuengel_2019_renormalised} and leads to
		\begin{align}
		\int_\Omega \varphi_i^E&(Z^\de(T, \cdot)) \psi(T, \cdot) \, \dd x \label{eqchainrulewithfinalt}
		- \int_\Omega \varphi_i^E(Z^0) \psi(0, \cdot) \, \dd x - \int_0^T\!\!\int_\Omega \varphi_i^E(Z^\de) \tfrac{\dd}{\dd t}\psi \, \dd x\dd t                                      \\
		= & - \sum_{j,k,l=1}^{I+1} \int_0^T\!\!\int_\Omega \psi \, \partial_j \partial_k \varphi_i^E (Z^\de) A^\ve_{jl}(Z^\de)\nabla Z^\de_l\cdot\nabla Z^\de_k\,\dd x\dd t
		\\& - \sum_{j,l=1}^{I+1} \int_0^T\!\!\int_\Omega \partial_j \varphi_i^E (Z^\de)A^\ve_{jl}(Z^\de)\nabla Z^\de_l\cdot\nabla \psi\,\dd x\dd t
		+ \sum_{j=1}^{I} \int_0^T\!\!\int_\Omega \psi \, \partial_j \varphi_i^E (Z^\de)R_{\ve,j}(Z^\ve) \,\dd x\dd t.
		\end{align}
		Observe that the LHS of eq.~\eqref{eqchainrulewithfinalt} is just the action of the distribution $\frac{\dd}{\dd t} \varphi_i^E(Z^\de) $ on $\psi \in C^\infty([0,T] \times \ol \Omega)$.
		On the other hand, thanks to Lemma~\ref{l:efl} (and in particular~\eqref{eq:105b}), the compactness of $\supp D\varphi_i^E$ and the continuous embedding $H^p(\Omega) \hookrightarrow W^{1,\infty}(\Omega)$ for $p>\tfrac{d}{2}+1$, the RHS of~\eqref{eqchainrulewithfinalt} can be bounded above in modulus by 
		\begin{equation*}
		C(E) \bigg( 1 + \int_0^T\!\!\int_\Omega P_\ve(Z^\ve)\,\dd x\dd t \bigg) \|\psi\|_{C([0,T],H^p(\Om))},
		\end{equation*}
		where the first term is controlled by the initial data (cf.~\eqref{equnifbounds}).
		Since $C^\infty([0,T] \times \ol \Omega)$ is dense in $C^0([0,T], H^p(\Omega))$,
		this implies the asserted $\ve$-uniform bound for $\big\| \frac{\dd}{\dd t} \varphi_i^E(Z^\de) \big\|_{\mathcal M([0,T], (H^p)^*)}$.
		
		Having proven the existence of a subsequence $\varphi_i^E(Z^\de)$ converging in $L^2(0,T;$ $L^2(\Omega))$ (for both $E \in \mathbb N$ and $T>0$ fixed), there exists a further subsequence converging pointwise a.e.\ to a measurable function $v_i^E$ for all $i \in \{1,\dotsc,I\}$ and all $E \in \mathbb N$. This can be shown by applying a diagonal sequence argument.
		
		We can now follow the reasoning in~\cite{Fischer_2015_renormalized}. Thanks to the uniform bounds on $\sum_{j=1}^{I} c_j^\de \log c_j^\de$ 
		in $L^\infty(0,\infty; L^1(\Omega))$ as well as~\ref{it:C5} and~\ref{it:C6}, the functions $\varphi_i^E(Z^\de) \log \varphi_i^E (Z^\de)$ 
		are uniformly bounded in $L^\infty(0,\infty; L^1(\Omega))$ w.r.t.\ $\varepsilon > 0$ and $E \in \mathbb N$. The pointwise a.e.\ convergence of $\varphi_i^E(Z^\de)$ to $v_i^E$ and Fatou's Lemma now entail an $E$-uniform bound on $v_i^E \log v_i^E$ 
		in $L^\infty(0,\infty; L^1(\Omega))$.
		We further proceed as in \cite{Fischer_2015_renormalized} to establish the pointwise a.e.\ convergence of $v^E=(v_i^E)_i$ to some measurable $c=(c_i)_i$ for $E \rightarrow \infty$ satisfying $c_i \log c_i\in L^\infty(0,\infty; L^1(\Omega))$. We start with the elementary observation
		\[
		\sum_{j=1}^{I} \varphi_j^E(Z^\de) + u^\varepsilon < E \ \Rightarrow \ \sum_{j=1}^{I} Z_j^\de + u^\varepsilon < E \ \Rightarrow \ \varphi_j^E(Z^\de) = Z_j^\de = \varphi_j^{\wt E}(Z^\de)
		\]
		for all $j \in \{1,\dotsc,I\}$ and $\wt E > E$, which follows from \ref{it:C9} and \ref{it:C6}. Moreover, if $\sum_{j=1}^{I} v_j^E + u = \lim_{\ve \rightarrow 0} \sum_{j=1}^{I} \varphi_j^E(Z^\de) + u^\varepsilon < E$, then $\sum_{j=1}^{I} \varphi_j^E(Z^\de) + u^\varepsilon < E$ for $\ve$ sufficiently small, and hence, $v_j^E = v_j^{\wt E}$ for all $\wt E > E$. By the uniform bound on $\sum_{j=1}^{I} v_j^E + u$ in $L^1([0,T] \times \Omega)$ for arbitrary but fixed $T>0$ we know that $\sum_{j=1}^{I} v_j^E + u \geq E$ can hold true only on a set of points with vanishing measure in the limit $E \rightarrow \infty$. As a consequence, $v^E=(v^E_i)_{i}$ converges a.e.\ in $\Omega_T$ to a measurable function $c = (c_i)_i$ satisfying, thanks to Fatou's lemma, $c_i \log c_i\in L^\infty(0,\infty; L^1(\Omega))$.
		
		Next, we prove that a subsequence $c_i^\de$ pointwise a.e.\ converges to $c_i$ for $\ve \rightarrow 0$. The uniform bound on $\sum_{j=1}^{I+1} Z_j^\de$ in $L^1([0,T] \times \Omega)$ guarantees that the measure of the subset of $\Omega_T$ where $\sum_{j=1}^{I+1} Z_j^\de \geq E$ holds true tends to zero for $E \rightarrow \infty$, uniformly in $\varepsilon$. Therefore, $\varphi_i^E(Z^\de) \neq c_i^\de$ can be true only on a set of points with vanishing measure in the limit $E \rightarrow \infty$, uniformly in $\varepsilon$. For any $\vartheta > 0$ we now find
		\begin{align*}
		& \mathcal L^{d+1} \bigg( \bigg\{ (t,x) \in \Om_T \ : \ |c_i^\de(t,x) - c_i(t,x)| > \vartheta \bigg\} \bigg)                                 \\
		& \quad \leq \mathcal L^{d+1} \bigg( \bigg\{ (t,x) \in \Om_T \ : \ c_i^\de(t,x) \neq \varphi_i^E(Z^\de)(t,x) \bigg\} \bigg)                  \\
		& \qquad + \mathcal L^{d+1} \bigg( \bigg\{ (t,x) \in \Om_T \ : \ |\varphi_i^E(Z^\de)(t,x) - v_i^E(t,x)| > \frac{\vartheta}{2} \bigg\} \bigg) \\
		& \qquad + \mathcal L^{d+1} \bigg( \bigg\{ (t,x) \in \Om_T \ : \ |v_i^E(t,x) - c_i(t,x)| > \frac{\vartheta}{2} \bigg\} \bigg).
		\end{align*}
		The first term on the right-hand side tends to zero for $E \rightarrow \infty$ as discussed just before. The second term tends to zero for fixed $E$ and $\ve \rightarrow 0$ due to the definition of $v_i^E$, whereas the last term converges to zero for $E \rightarrow \infty$ by the definition of $c_i$. This ensures the convergence of $c^\ve$ to $c$ in measure and, hence, convergence a.e.\ for another subsequence.
		Combined with the uniform boundedness of $c_i^\de \log c_i^\de$ in $L^\infty(0, T; L^1(\Omega))$ for every $T>0$, this implies that $c_i^\de$ converges strongly to $c_i$ in $L^p(0,T; L^1(\Omega))$ for any $p\in[1,\infty)$ and $i\in\{1,\dots,I\}$.
		As a result of the strong convergence of $c_i^\de$ to $c_i$ in $L^1(0,T; L^1(\Omega))$, we also obtain the distributional convergence of $\sqrt{c_i^\de}$ to $\sqrt{c_i}$. In combination with the uniform $L^2(0,T; H^1(\Omega))$ bound on $\sqrt{c_i^\de}$, we deduce that a subsequence $\sqrt{c_i^\de}$ weakly converges to $\sqrt{c_i}$ in $L^2(0,T; H^1(\Omega))$.
		Employing the convergence $u^\de \rightarrow u$ in the distributional sense and the uniform bound~\eqref{eq:ut2} on $\partial_t u^\de$ in $L^s(0,T; W^{1, s'}(\Omega)^\ast)$, where $s=\tfrac{2d+2}{2d+1}$,
		we infer that $\partial_t u^\de \stackrel{\ast}{\rightharpoonup} \partial_t u \ \ \mbox{in}\ L^s(0,T; W^{1, s'}(\Omega)^\ast)$ with the limit $\partial_t u$ again satisfying~\eqref{eq:ut2} by weak-$\ast$ lower semi-continuity.
		The convergence results derived above are sufficient to pass to the limit $\ve\rightarrow 0$ in the equation for $u^\de$ to infer
		\begin{equation}
		\label{equweaklimit}
		\int_0^T \int_{\Omega} a(c, u) \nabla u \cdot \nabla \phi \,\dd x\dd t+ \int_0^T \big\langle \partial_t u, \phi \big\rangle \, \dd t = 0,
		\end{equation}
		where the distributional convergence of $a(c^\ve, u^\ve) \nabla u^\ve$ to $a(c, u) \nabla u$ can be obtained as before using Lemma~\ref{l:1}. Using lower semi-continuity arguments and the above convergence results, we infer from~\eqref{equnifbounds} in particular the bounds
		\begin{equation}
		\|\nabla \sqrt{c}\|_{L^2(\Omega_T)} + \Big\| \sqrt{a} \nabla u \Big\|_{L^2(\Omega_T)}
		+ \Big\|\pi_1^\frac12 (\hat \sigma''+\sum_ic_i\tfrac{w_i''}{w_i})\nabla u\Big \|_{L^2(\Omega_T)}
		\le C(\mathrm{data}).\label{eqfinalbounds}
		\end{equation}
		
		It remains to prove the strong convergence
		\begin{equation}\label{eq:2.1}
		\nabla u^\de \rightarrow \nabla u \quad\text{ in }L^2(\Omega_T).
		\end{equation}
		The key ingredient are the following $L^2$-energy identities, valid for a.e.~$T>0$,
		\begin{align}
		& \label{eq:2.wt}
		\frac12 \int_\Omega u^\de(T, \cdot)^2 \, \dd x\dd t- \frac12 \int_\Omega u_0^2 \, \dd x + \int_0^T \int_\Omega a(c^\de, u^\de) |\nabla u^\de|^2 \, \dd x\dd t= 0,
		\\&\label{eq:2.wlt}
		\frac12 \int_\Omega u(T, \cdot)^2 \, \dd x - \frac12 \int_\Omega u_0^2 \, \dd x + \int_0^T \int_\Omega a(c, u) |\nabla u|^2 \, \dd x\dd t = 0.
		\end{align}
		Let us first assume the validity of these equations and show how they entail~\eqref{eq:2.1}.
		
		Since for a.e.~$T>0$ the first term on the LHS of~\eqref{eq:2.wt} converges to the corresponding term in~\eqref{eq:2.wlt}, the above identities imply that
		\begin{equation*}
		\int_0^T \int_\Omega a(c^\de, u^\de) |\nabla u^\de|^2 \, \dd x\dd t \to  \int_0^T \int_\Omega a(c, u) |\nabla u|^2 \, \dd x\dd t.
		\end{equation*}
		Thanks to the weak convergence $\sqrt{a(c^\de, u^\de)} \nabla u^\de \rightharpoonup \sqrt{a(c, u)} \nabla u$ in $L^2(\Omega_T)$, we infer that $\sqrt{a(c^\de, u^\de)} \nabla u^\de $ converges strongly to $\sqrt{a(c, u)} \nabla u$ in $L^2(\Omega_T)$.
		Since $\sqrt{a}\gtrsim 1$ and $(c^\de, u^\de)\to (c,u)$ a.e., this easily implies~\eqref{eq:2.1}.
		
		\medskip
		
		\underline{Proof of identities~\eqref{eq:2.wt} and~\eqref{eq:2.wlt}:}
		we confine ourselves to the  proof of~\eqref{eq:2.wlt}, since~\eqref{eq:2.wt} can be shown along the same lines.
		Recall that $s = \tfrac{2d+2}{2d+1}$ along with the conjugate exponent $s' = 2d+2$. The function $u$ is then known to satisfy
		\begin{align}\label{eq:500.0}
		\int_0^T \langle \tfrac{\dd u}{\dd t}, \phi \rangle_{X^* \times X} \, \dd t + \int_0^T \int_\Omega a(c, u) \nabla u \cdot \nabla \phi \,\dd x\dd t  = 0
		\end{align}
		for all $\phi\in L^{s'}(0,T;X)$ where $X \coleq W^{1, s'}(\Omega)$.
		Loosely speaking, in order to prove identity~\eqref{eq:2.wlt}, we would like to choose $\phi=u$ as a test function in~\eqref{eq:500.0}. However, the function $u$ may not be sufficiently regular for this choice to be admissible, requiring a regularisation and approximation argument.
		For this purpose, we use again the radially symmetric spatial  mollifier $\wt \rho_\eta$  from \eqref{eqdefmollifier} and let $\rho_\eta=\wt\rho_\eta\ast \wt\rho_\eta$. We further take $\zeta\in C^\infty_c([0,\infty))$ such that $\zeta(z)=\tfrac{1}{2}z^2$ if $z\in[0,1]$ and $\zeta(z)=0$ if $z \ge 2$, and then define $\zeta_E(z)=E^2\zeta(\tfrac{z}{E})$ for $E\ge 1$. Observe that $\zeta_E''(z)=\zeta''(\tfrac{z}{E})$ and, hence, $\sup_{E,z}|\zeta_E''(z)|= \|\zeta''\|_{C([0,\infty))}<\infty$.
		As in the proof of~\cite[Lemma~4]{Fischer_2015_renormalized}, we make use of a partition of unity $\chi_k$ on $\ol \Omega$, i.e.\ we consider functions $\chi_k \in C^\infty(\ol \Omega)$ satisfying $\sum_{k=1}^K \chi_k = 1$ on $\ol \Omega$ such that each $\chi=\chi_k$ is either compactly supported in $\Om$, or supported in some coordinate chart $U\subset\ol\Om$ that is relatively open in $\ol\Om$ and has the property that $U\cap \partial\Om$ can be written as the graph of a suitable Lipschitz function~$\tau$.
		
		In the first situation, we have $\supp \chi \subset \Omega^{4\ol\eta}$ for $\ol\eta > 0$ sufficiently small. We then choose the test function
		\begin{align}\label{eq:500}
		\phi_\eta^E \coleq \rho_\eta \ast \big( \zeta_E'(\rho_\eta \ast u) \chi \big) \in L^\infty(0,T; W^{1,\infty}(\Omega)) \subset L^{s'}(0,T;X)
		\end{align}
		in equation~\eqref{eq:500.0} with $\eta \in (0, \ol\eta)$.
		Since $\chi$ and  $\phi_\eta^E$ are compactly supported in $\Om$,  we have
		\begin{equation}\label{eq:parts.en}
		\begin{split}
		\int_0^T \Big\langle \tfrac{\dd}{\dd t} u, \phi_\eta^E \Big \rangle_{X^\ast \times X} \dd t
		&=\int_0^T \int_\Om\tfrac{\dd}{\dd t} (\rho_\eta\ast u)\zeta_E'(\rho_\eta \ast u) \chi  \,\dd x\dd t
		\\&=\int_\Omega \zeta_E(\rho_\eta\ast u(T,\cdot)) \chi \, \dd x - \int_\Omega \zeta_E(\rho_\eta\ast u_0) \chi \, \dd x.
		\end{split}
		\end{equation}
		Notice that the equation for $u$ (i.e.~\eqref{eq:500.0}) ensures sufficient (time) regularity of $\rho_\eta\ast u$ to justify the above lines.
		Hence, by dominated convergence,
		\begin{align}
		\label{eqtimetermbulkconv}
		\int_0^T \Big\langle \tfrac{\dd}{\dd t} u, \phi_\eta^E \Big \rangle_{X^\ast \times X} \dd t \rightarrow \int_\Omega \zeta_E(u(T, \cdot)) \chi \, \dd x - \int_\Omega \zeta_E(u_0) \chi \, \dd x
		\end{align}
		as $\eta \rightarrow 0$.
		
		On the other hand, we assert that
		\begin{align}
		\label{eqgradtermbulkconv}
		\int_0^T \int_\Omega a(c, u) \nabla u \cdot \nabla \phi^E_\eta \, \dd x\dd t \rightarrow \int_0^T \int_\Omega a(c, u) \nabla u \cdot \nabla \big( \zeta_E'(u) \chi \big) \, \dd x\dd t
		\end{align}
		as $\eta \rightarrow 0$. Reformulating the LHS as
		\begin{align*}
		\int_0^T \int_\Omega a(c, u) \nabla u \cdot \nabla \phi^E_\eta \, \dd x\dd t = \int_0^T \int_{\Omega^{3\eta}} J_\eta(t, x) \cdot K_\eta(t, x) \, \dd x\dd t
		\end{align*}
		with $J_\eta(t, x) \coleq \wt \rho_\eta \ast \nabla \big( \zeta_E'(\rho_\eta \ast u) \chi \big)$ and $K_\eta(t, x) \coleq \wt \rho_\eta \ast (a(c, u) \nabla u)$ for all $(t, x) \in (0, T) \times \Omega^{3\eta}$, we can verify \eqref{eqgradtermbulkconv} by using the same arguments as in the proof of Lemma \ref{lemmachainrulelimit}: in the current situation, Step~1 reads
		\begin{align}
		|K_\eta(t,x)| & \leq C \eta^{1-d} \int_{B(x,3\eta)} \Big( |\nabla u|^2 + P(Z) \Big) \, \dd y                                                                                            \\
		& \quad + C \eta^{-\frac{d}{2}} \Big( \int_{B(x,3\eta)} \Big( |\nabla u|^2 + P(Z) \Big) \, \dd y \Big)^\frac12 \inf_{\wt x \in B(x,\eta)} |(\rho_\eta \ast (1+u))(\wt x)|.
		\end{align}
		Step 2 is slightly more complex. Here, we estimate
		\begin{align}
		|J_\eta (t, x)| & \leq \big|\wt \rho_\eta \ast \big( \chi \nabla \zeta_E'(\rho_\eta \ast u) \big) \big| + \big|\wt \rho_\eta \ast \big( \zeta_E'(\rho_\eta \ast u) \nabla \chi \big) \big| \\
		& \leq C
		\min \Big\{ \eta^{-1}, 1 + \eta^{-\frac{d}{2}} \Big( \int_{B(x,3\eta)} |\nabla u|^2 \, \dd y \Big)^\frac12 \Big\},
		\end{align}
		where the first part of the estimate follows as in \cite{ChenJuengel_2019_renormalised}, while the second part is trivially bounded in terms of a constant that can be controlled by the last line. The upper bound from Step 3 is then given by
		\[
		|J_\eta(t,x) \cdot K_\eta(t,x)| \leq C E \eta^{-d} \int_{B(x,3\eta)}\big( 1 + |\nabla u|^2 + P(Z) \big) \, \dd y
		\]
		for all $(t, x) \in (0, T) \times \Omega^{3\eta}$. Defining $J_0(t, x) \coleq \nabla \big( \zeta_E'(u) \chi \big)$ and $K_0(t, x) \coleq a(c, u) \nabla u$, Step 4 reduces to $|J_0(t, x) K_0(t, x)| \leq C (1 + |\nabla u|) a(c, u) |\nabla u| \in L^1((0, T) \times \Omega)$ for all $(t, x) \in (0, T) \times \Omega^{3\eta}$ and Step 5 can be reused in a one-to-one fashion to establish \eqref{eqgradtermbulkconv}.
		
		We are left to deal with the case that $\supp \chi \subset U$ for some coordinate chart $U \subset \ol\Omega$ as above that satisfies $U \cap \partial \Omega \neq \emptyset$. Employing the notation of~\cite[Proof of Lemma~4]{Fischer_2015_renormalized} and using the fact that $\Omega$ is a Lipschitz domain, there exist a set $V \subset \mathbb R_0^+ \times \mathbb R^{d-1}$ (relatively open in $\mathbb R_0^+ \times \mathbb R^{d-1}$)
		and a bi-Lipschitz homeomorphism $\theta: V \rightarrow U$ with the property that after appropriate rotations and translations of $\Omega$ one has
		\begin{align}
		\label{eqdeftheta}
		\theta(y) = (y_1 + \tau(y_2, \dotsc, y_d), y_2, \dots, y_d)
		\end{align}
		for some suitable Lipschitz function $\tau : \mathbb R^{d-1} \rightarrow \mathbb R$ and
		all $y \in V$.
		Note that $\det D\tau = 1$ a.e.\ in $\Omega$. We then let $\Theta: [0, T] \times V \rightarrow [0, T] \times U$, $\Theta(t, y) \coleq (t, \theta(y))$, and introduce the transformed variables
		\begin{align}
		\hat c \coleq c \circ \Theta, \quad \hat u \coleq u \circ \Theta, \quad \hat \chi \coleq \chi \circ \theta.
		\end{align}
		We now extend these transformed variables by mirroring to $[0, T] \times V_a$ and $V_a$, respectively, where $V_a \coleq V \cup V_m$ and $V_m \coleq \{ y \in \mathbb R^d \ | \ (-y_1, y') \in V\}$ using the short hand notation $y' \coleq (y_2, \dotsc, y_d)$. More precisely,
		\begin{align}
		\hat c(t, y_1, y') \coleq \hat c(t, -y_1, y'), \quad \hat u(t, y_1, y') \coleq \hat u(t, -y_1, y'), \quad \hat \chi(y_1, y') \coleq \hat \chi(-y_1, y')
		\end{align}
		for all $(y_1, y') \in V_m$.
		Furthermore, we extend $\theta$ to $V_a$ by imposing the same formula as in \eqref{eqdeftheta} but now valid for all $(y_1, y') \in V_a$.
		Observe that for $\hat\phi:=\phi\circ\Theta$
		\begin{equation*}
		\int_0^T \int_U a(c, u) \nabla u \cdot \nabla \phi \, \dd x\dd t
		= \int_0^T \int_V a(\hat c, \hat u) \big( (D\theta)^{-T} \nabla \hat u \big)\cdot  \big( (D\theta)^{-T} \nabla \hat\phi\big) \,\dd y\dd t,
		\end{equation*}
		and a similar expression holds for the integral involving the time derivative. We will now show that an identity analogous to~\eqref{eq:500.0} holds true for the extended quantities on $V_a$ in order to essentially reduce the problem to the first situation where $\supp\chi\subset\subset\Om$.
		While on the set $V$ one has
		\begin{align}
		D\theta =
		\begin{pmatrix}
		1 & \partial_2 \tau & \dots  & \partial_d \tau \\
		& 1               &        &                 \\
		&                 & \ddots &                 \\
		&                 &        & 1
		\end{pmatrix},
		\qquad
		(D\theta)^{-T} =
		\begin{pmatrix}
		1                &   &        &   \\
		-\partial_2 \tau & 1 &        &   \\
		\vdots           &   & \ddots &   \\
		-\partial_d \tau &   &        & 1
		\end{pmatrix},
		\end{align}
		the corresponding matrices on $V_m$ read
		\begin{align}
		D\theta =
		\begin{pmatrix}
		-1 & \partial_2 \tau & \dots  & \partial_d \tau \\
		& 1               &        &                 \\
		&                 & \ddots &                 \\
		&                 &        & 1
		\end{pmatrix},
		\qquad
		(D\theta)^{-T} =
		\begin{pmatrix}
		-1              &   &        &   \\
		\partial_2 \tau & 1 &        &   \\
		\vdots          &   & \ddots &   \\
		\partial_d \tau &   &        & 1
		\end{pmatrix}.
		\end{align}
		As a consequence, we find
		\begin{equation*}
		\int_0^T \int_U a(c, u) \nabla u \cdot \nabla \phi \, \dd x\dd t
		= \frac{1}{2}\int_0^T \int_{V_a} a(\hat c, \hat u) \big( (D\theta)^{-T} \nabla \hat u \big)\cdot  \big( (D\theta)^{-T} \nabla \hat\phi\big) \,\dd y\dd t,
		\end{equation*}
		by using the symmetry w.r.t.\ $y_1$ of the terms inside the gradients and the structure of the first column of $(D\theta)^{-T}$ on $V$ and $V_m$, respectively.
		Using similar (but simpler) symmetry properties for the integral involving the time derivative, this allows us to infer
		\begin{equation*}
		\int_0^T \langle\tfrac{\dd }{\dd t}\hat u,\hat\phi\rangle \,\dd t
		= \int_0^T \int_{V_a} a(\hat c, \hat u) \big( (D\theta)^{-T} \nabla \hat u \big)\cdot  \big( (D\theta)^{-T} \nabla \hat\phi\big)  \,\dd y\dd t,
		\end{equation*}
		where $\langle\tfrac{\dd }{\dd t}\hat u,\hat\phi\rangle:=\langle\tfrac{\dd }{\dd t}u,\hat\phi(t,(\theta_{|V})^{-1})\rangle+\langle\tfrac{\dd }{\dd t}u,\hat\phi(t,O\circ(\theta_{|V})^{-1})\rangle$, where $O(y)=(-y_1,y').$
		We now choose the (symmetric) test function
		\begin{align}
		\hat\phi :=\hat\phi^E_\eta:= \rho_\eta \ast ( \zeta_E' (\rho_\eta \ast \hat u) \hat \chi ),
		\end{align}
		where $\rho_\eta$, $0<\eta\ll 1$, and $\zeta^E$, $E\gg 1$, are as before.
		The symmetry of $\hat\phi^E_\eta$ follows from the radial symmetry of $\wt\rho_\eta$. Moreover, we have $\supp\hat\chi\subset\subset V_a.$
		We can therefore argue essentially as in the first situation (derivation of~\eqref{eqtimetermbulkconv} and~\eqref{eqgradtermbulkconv}), the main difference being the appearance of the pull-backs of the gradients in the integral on the RHS, which are, however, harmless thanks to the boundedness of $(D\theta)^{-T}$.
		Concerning the term involving $\tfrac{\dd }{\dd t}\hat u$, we omit here a standard approximation procedure to justify the chain rule argument for the time derivative analogous to~\eqref{eq:parts.en} (one can use e.g.\;the methods in~\cite[Chapter~7]{Roubicek_2013_npde}).
		In the limit $\eta\to0$, we then obtain
		\begin{multline*}
		\int_{V_a} \zeta_E(\hat u(T, \cdot)) \hat\chi \, \dd y-\int_{V_a}\zeta_E(\hat u_0) \hat\chi \, \dd y
		\\=  \int_0^T \int_{V_a} a(\hat c, \hat u)\big((D\theta)^{-T} \nabla \hat u\big)\cdot  (D\theta)^{-T}\nabla \big( \zeta_E'(\hat u) \hat \chi \big) \,\dd y\dd t
		\end{multline*}
		or equivalently
		\begin{equation*}
		\int_\Omega \zeta_E(u(T, \cdot)) \chi \, \dd x-\int_\Omega \zeta_E(u_0) \chi \, \dd x
		=  	\int_0^T \int_U a(c, u) \nabla u \cdot \nabla \big( \zeta_E'( u) \chi \big)   \, \dd x\dd t.
		\end{equation*}
		In combination, we infer that
		\begin{align}
		\int_\Omega \zeta_E(u(T, \cdot)) \chi_k \, \dd x - \int_\Omega \zeta_E(u_0) \chi_k \, \dd x + \int_0^T \int_\Omega a(c, u) \nabla u \cdot \nabla \big( \zeta_E'(u) \chi_k \big) \, \dd x\dd t = 0
		\end{align}
		for a.e.\ $T > 0$ and all $k=1,\dots, K$, and thus, upon summation over $k$,
		\begin{align}
		\int_\Omega \zeta_E(u(T, \cdot)) \, \dd x - \int_\Omega \zeta_E(u_0) \, \dd x + \int_0^T \int_\Omega a(c, u) |\nabla u|^2 \zeta_E''(u) \, \dd x\dd t = 0.
		\end{align}
		Lebesgue's dominated convergence theorem finally allows us to pass to the limit $E \rightarrow \infty$ for a.e.\ $T > 0$ to establish \eqref{eq:2.wlt}.
	\end{proof}

	\subsection{Preliminary PDE for $\varphi_i^E(Z)$}
	As it is generally not possible to directly obtain the desired equation for $\xi(Z)$ by passing to the limit $\varepsilon \rightarrow 0$ in the equation for $\xi(Z^\varepsilon)$, where $\xi \in C^\infty([0,\infty)^{I+1})$ has compactly supported derivatives, we first pass to the limit $\varepsilon \rightarrow 0$ in \eqref{eqchainrulewithfinalt}. Thanks in particular to the strong convergence of $\nabla u^\ve$, this leads to an equation for the truncations $\varphi_i^E(Z)$ involving signed Radon measures $\mu_i^E$  that vanish in the limit $E \rightarrow \infty$. 
	The Radon measures $\mu_i^E$ arise due to the terms in eq.~\eqref{eqchainrulewithfinalt} involving squared gradients of the form $\nabla c_l\cdot\nabla c_k$ and the lack of strong convergence in $L^2$ of $\nabla\sqrt{c_i}$ for all $i$.
	
	\begin{lemma}
		\label{lemmalimitepsilon}
		Let $Z$ be the limiting function obtained in Lemma \ref{lemmasubsequence}. Then, for all $E, T>0$ and $\psi \in C^\infty_c([0,T)\times\ol\Omega)$ 
		the function $\varphi_i^E(Z)$ satisfies the identity
		\begin{align}
		- \int_\Omega \varphi_i^E(Z^0) & \psi(0, \cdot) \, \dd x - \int_0^T\!\!\int_\Omega \varphi_i^E(Z) \tfrac{\dd}{\dd t} \psi \, \dd x\dd t  \label{eqchainrulelimit}               \\
		=                              & - \int_0^T\!\!\int_\Omega \psi \, d\mu_i^E(t, x)                                                                                               \\
		& - \sum_{j,k=1}^{I+1} \int_0^T\!\!\int_\Omega \psi \, \partial_j \partial_k \varphi_i^E (Z) A_{j,I+1}(Z)\nabla u \cdot \nabla Z_k \, \dd x\dd t
		\\& - \sum_{j,l=1}^{I+1}\int_0^T\!\!\int_\Omega \partial_j \varphi_i^E (Z)A_{jl}(Z)\nabla Z_l \cdot\nabla\psi \, \dd x\dd t                     \\
		& + \sum_{j=1}^I \int_0^T\!\!\int_\Omega \psi \, \partial_j \varphi_i^E (Z) R_j(Z)\, \dd x\dd t,
		\end{align}
		where $\mu_i^E$ is a sequence of signed Radon measures with the property that
		\begin{equation}
		\label{eqlimitmuie}
		\lim_{E \rightarrow \infty} |\mu_i^E|([0,T) \times \ol \Omega) = 0
		\end{equation}
		for all $T > 0$ and every $i \in \{1, \dotsc, I\}$. 
		Moreover, under the same hypotheses we have
		\begin{multline}
		\label{equweakequation}
		- \int_\Omega Z^0_{I+1} \psi(0, \cdot) \, \dd x - \int_0^T\!\!\int_\Omega Z_{I+1} \tfrac{\dd}{\dd t} \psi \, \dd x \dd t \\
		= -\int_0^T\!\!\int_{\Omega} A_{I+1,I+1}(Z) \nabla Z_{I+1} \cdot \nabla \psi \, \dd x\dd t.
		\end{multline}
	\end{lemma}
	\begin{proof}
		We stress that \eqref{equweakequation} is a reformulation of \eqref{equweaklimit}, which serves as the analogue of \eqref{eqchainrulelimit} for the internal energy but without the need for the functions $\varphi_i^E$.
		
		For proving \eqref{eqchainrulelimit}, we aim to pass to the limit $\ve \rightarrow 0$ in~\eqref{eqchainrulewithfinalt} (or equivalently in~\eqref{eqchainrule} with $\rreg=\ve$) using the convergence properties of $Z^\ve$ to $Z$ established in Lemma~\ref{lemmasubsequence}. The terms on the LHS converge to the LHS of~\eqref{eqchainrulelimit} thanks to dominated convergence. This argument also applies to the last term on the RHS involving the reactions. Next, we make the crucial observation that, thanks to the strong convergence of $\nabla u^\ve\to\nabla u$ in $L^2(\Om_T)$, the first terms on the RHS of~\eqref{eqchainrulewithfinalt} with $l=I+1$ (involving the thermodiffusive cross terms) converge to the second term on the RHS of~\eqref{eqchainrulelimit}.
		Here, one also uses the weak convergence $\nabla \sqrt{c_j^\de} \rightharpoonup \nabla \sqrt{c_j}$ in $L^2(0,T; L^2(\Omega))$, Lemma~\ref{l:efl}, the uniform bounds~\eqref{equnifbounds},
		Lemma \ref{l:weakConv} and the fact that $D \varphi_i^E$ is compactly supported. A similar (but simpler) reasoning allows to pass to the limit in the penultimate line. Regarding the remaining terms in the first line of the RHS of~\eqref{eqchainrulewithfinalt}, we need to take more care since they
		involve products $\nabla c^\de_j \cdot \nabla c^\de_k$ of merely weakly converging functions.
		To treat these terms we define the signed Radon measures
		\begin{align*}
		\mu_{i}^{\dee} & \coleq \sum_{k=1}^{I+1}\sum_{j=1}^I \partial_j \partial_k \varphi_i^E(Z^\de)A_{jj}^\ve(Z^\ve)\nabla c_j^\ve\cdot\nabla Z_k\, \dd x\dd t,
		\end{align*}
		which are easily seen to be uniformly bounded in $\ve\in(0,1]$ as long as $E>0$ is kept fixed.
		A key property is the fact that
		the measures $\mu_{i}^{\dee}$ can also be controlled uniformly in $E$, which will allow us to infer~\eqref{eqlimitmuie}.
		To see this, we note that thanks to the properties \ref{it:C2} and \ref{it:C3} of $\varphi_i^E$,  the pointwise estimate~\eqref{eq:101c} of $A_{jj}^\ve(Z^\ve)\nabla c_j^\ve$, the fact that $P_\ve^\frac{1}{2}(Z^\ve)$ controls $\ve^\frac{1}{2}|\nabla c_k|+|\nabla\sqrt{c_k}|$ and the $\ve$-uniform control of $\nabla u^\ve$ in $L^2(\Om_T)$ (due to~\eqref{equnifbounds} and~\ref{it:H2}),  we have
		\[
		|\mu_{i}^\dee|([0,T) \times \ol \Omega) \lesssim \int_0^T\!\!\int_\Om \big(P_\ve(Z^\ve)+|\nabla u^\ve|^2\big)\,\dd x\dd t\le C(\mathrm{data}).
		\]
		By the weak-$\ast$ compactness of Radon measures, there exists a subsequence $\mu_{i}^\dee$ that converges weak-$\ast$ in measure to some Radon measure $\mu_i^E$ as $\de \rightarrow 0$.
		
		For proving that $\mu_i^E$ vanishes in the limit $E \rightarrow \infty$, we introduce the non-negative measures
		\[
		\nu^{\de,L} \coleq \chi_{\{ |(c^\de, u^\de)| \in [L-1,L) \}} \big(P_\ve(Z^\ve)+|\nabla u^\ve|^2\big) \dd x\dd t,\quad L\in\mathbb{N},
		\]
		which satisfy the bound $\sum_{L=1}^\infty\nu^{\ve,L}([0,T) \times \ol \Omega)\lesssim C(\mathrm{data})$.
		
		As before, we employ \ref{it:C2} to estimate
		\begin{align*}
		|\mu_{i}^\dee|([0,T) \times \ol \Omega)                                                                                                                                       
		& \ \lesssim\sum_{L=1}^\infty \int_0^T\!\!\int_\Om\chi_{\{ |(c^\de, u^\de)| \in [L-1,L) \}} |Z^\ve||D^2\varphi_i^E(Z^\de)| \big(P_\ve(Z^\ve)+|\nabla u^\ve|^2\big)\, \dd x\dd t \\
		& \ \lesssim \sum_{L=1}^\infty \nu^{\de,L}([0,T) \times \ol \Omega) \sup_{|\tilde Z| \in [L-1,L)}|\tilde Z||D^2\varphi_i^E(\tilde Z)|.
		\end{align*}
		Observe that since $D^2\varphi^E_i$ is compactly supported, for fixed $E \in \mathbb N$ only finitely many terms on the RHS are non-zero. This allows us to estimate
		\begin{align*}
		|\mu_i^E|([0,T) \times \ol \Omega) & \leq \liminf_{\de \rightarrow 0} |\mu_{i}^\dee|([0,T) \times \ol \Omega)                                                                                         \\
		& \lesssim \sum_{L=1}^\infty \liminf_{\de \rightarrow 0} \nu^{\de,L}([0,T) \times \ol \Omega) \sup_{{|\tilde Z| \in [L-1,L)}} |\tilde Z| |D^2\varphi_i^E(\tilde Z)|,
		\end{align*}
		where the first inequality is due to the weak-$\ast$ lower semi-continuity of the total variation of Radon measures on an open set.
		
		Next, applying Fatou's Lemma (for the counting measure on $\mathbb N$), we find
		\[
		\sum_{L=1}^\infty \liminf_{\de \rightarrow 0} \nu^{\de,L}([0,T) \times \ol \Omega) \leq \liminf_{\de \rightarrow 0} \sum_{L=1}^\infty \nu^{\de,L}([0,T) \times \ol \Omega) \le C(\mathrm{data}).
		\]
		By using \ref{it:C2}, \ref{it:C7}, the previous estimate, and the dominated convergence theorem (for the counting measure on $\mathbb N$), we are led to
		\begin{align*}
		\limsup_{E \rightarrow \infty} |\mu_i^E|([0,T) \times \ol \Omega)                                                                                                                                                              & \ \lesssim  \sum_{L=1}^\infty \liminf_{\de \rightarrow 0} \nu^{\de,L}([0,T) \times \ol \Omega) \lim_{E \rightarrow \infty} \sup_{{|\tilde Z| \in [L-1,L)}}|\tilde Z| |D^2\varphi_i^E(\tilde Z)| \\
		& \  \lesssim \sum_{L=1}^\infty \liminf_{\de \rightarrow 0} \nu^{\de,L}([0,T) \times \ol \Omega) \cdot 0 \, = \, 0.
		\end{align*}
		This concludes the proof.
	\end{proof}

	\subsection{Proof of the existence of global renormalised solutions}\label{ssec:proof.renorm}
	We are now in a position to deduce the equation for $\xi(Z)$ as stated in Definition \ref{def:RenormSol}.
	In fact, we derive an approximate expression for the weak time derivative of $\xi(\varphi^E(Z), Z_{I+1})$ from which the equation for $\xi(Z)$ then emerges when passing to the limit $E \rightarrow \infty$.
	
	\begin{proof}[Proof of Theorem \ref{theoremexistence}]
		In order to prove that $Z$ is a global renormalised solution of~\eqref{eq:system}, let $T\in(0,\infty)$ be arbitrary. Keeping in mind that $Z$ satisfies \eqref{eqchainrulelimit} and
		\eqref{equweakequation}, the hypotheses of Lemma \ref{lemmachainrulefischer}
		are fulfilled if we define
		\begin{align*}
		v_i \coleq     & \ \varphi_i^E(Z), \quad (v_0)_i \coleq \varphi_i^E(Z^0),     \quad                                                                                             
		v_{I+1} \coleq  \ Z_{I+1}, \quad (v_0)_{I+1} \coleq Z^0_{I+1}, \\
		\nu_i \coleq & -\mu_i^E, \quad \nu_{I+1} \coleq 0, \\
		w_i \coleq   & - \sum_{j,k=1}^{I+1} \partial_j \partial_k \varphi_i^E (Z)A_{j,I+1}(Z)\nabla u \cdot \nabla Z_k + \sum_{j=1}^I \partial_j \varphi_i^E (Z) R_j(Z), \\
		w_{I+1} \coleq & \ 0, \\
		z_i \coleq   & - \sum_{j,l=1}^{I+1} \partial_j \varphi_i^E (Z)A_{jl}(Z)\nabla Z_l, \quad
		z_{I+1} \coleq  \ -A_{I+1,I+1}(Z) \nabla Z_{I+1}.
		\end{align*}
		As a consequence, the function $\xi(\varphi^E(Z), Z_{I+1})$ satisfies for all $\psi \in C^\infty_c([0,T)\times\ol\Omega)$ 
		the estimate
		\begin{align*}
		\bigg| -\int_0^T & \!\!\int_\Omega \xi(\varphi^E(Z), Z_{I+1}) \tfrac{\dd}{\dd t} \psi \,  \dd x\dd t  - \int_\Omega \xi(\varphi^E(Z^0), Z_{I+1}^0) \psi(0, \cdot) \, \dd x                                                           \\
		& \quad + \sum_{i=1}^I\sum_{k,j,l=1}^{I+1} \int_0^T\!\!\int_\Omega \psi \partial_i \partial_k \xi (\varphi^E(Z), Z_{I+1}) \partial_j \varphi_i^E (Z) A_{jl}(Z)\nabla Z_l\cdot \nabla \varphi_k^E(Z) \,  \dd x\dd t \\
		& \quad + \sum_{k=1}^{I+1} \int_0^T\!\!\int_\Omega \psi \partial_{I+1} \partial_k \xi (\varphi^E(Z), Z_{I+1}) A_{I+1,I+1}(Z)\nabla Z_{I+1} \cdot \nabla \varphi_k^E(Z) \,  \dd x\dd t \\
		& \quad + \sum_{i=1}^I\sum_{j,l=1}^{I+1} \int_0^T\!\!\int_\Omega \partial_i \xi (\varphi^E(Z), Z_{I+1}) \partial_j \varphi_i^E (Z)A_{jl}(Z)\nabla Z_l\cdot \nabla \psi \, \dd x\dd t                               \\
		& \quad + \int_0^T\!\!\int_\Omega \partial_{I+1} \xi (\varphi^E(Z), Z_{I+1}) A_{I+1,I+1}(Z)\nabla Z_{I+1} \cdot \nabla \psi \, \dd x\dd t                               \\
		& \quad + \sum_{i=1}^{I}\sum_{j,k=1}^{I+1} \int_0^T\!\!\int_\Omega \psi \partial_i \xi(\varphi^E(Z), Z_{I+1}) \partial_j \partial_k \varphi_i^E (Z)A_{j,I+1}(Z)\nabla u \cdot \nabla Z_k\,  \dd x\dd t               \\
		& \quad - \sum_{i=1}^{I} \sum_{j=1}^I \int_0^T\!\!\int_\Omega \psi \partial_i \xi(\varphi^E(Z), Z_{I+1}) \partial_j \varphi_i^E (Z) R_j(Z) \,\dd x\dd t  \bigg|                                        \\
		& \hspace{4cm} \le C \| \psi \|_{L^\infty} \sup_{\tilde v \in \mathbb R^I} |D\xi(\tilde v)| \, \sum_{i=1}^{I} \| \mu_i^E \|_{\mathcal M([0,T) \times \ol \Omega)}.
		\end{align*}
		In view of~\eqref{eqlimitmuie}, the RHS converges to zero as
		$E \rightarrow \infty$. To establish the convergence of the LHS, we employ the
		pointwise a.e.\ convergence of $\varphi^E_i(Z)$ to $Z_i$, the boundedness of
		$D\varphi_i^E$ (cf. \ref{it:C5}), the compact support of
		$D\xi$, 
		the regularity $\nabla \sqrt{c_j}, \nabla u \in L^2((0,T) \times \Omega)$,
		Lemma~\ref{l:efl}
		and the uniform bounds \eqref{equnifbounds}. As in \cite{Fischer_2015_renormalized} and \cite{ChenJuengel_2019_renormalised}, we utilize the following auxiliary result: there exists a constant $E_0>0$ such that for all $E > E_0$ the relation $\sum_{i=1}^{I+1} Z_i \geq E_0$ implies $D \xi(\varphi^E(Z), Z_{I+1}) = D\xi(Z) = 0$ and $D^2\xi(\varphi^E(Z), Z_{I+1}) = D^2\xi(Z) = 0$. This result ensures that all terms involving derivatives of $\xi$ are zero if $\max_i Z_i$ is larger than $E_0$.
		
		This auxiliary result is easily proven by choosing $E_0>0$ such that $\supp D\xi \subset B_{E_0/\sqrt{I+1}}(0)$. For $E > E_0$ and $\sum_{i=1}^{I+1} Z_i  \geq E_0$, we then get $\sum_{i=1}^{I} \varphi_i^E(Z) + Z_{I+1} \geq E_0$ from \ref{it:C9}. 
		As a consequence, $Z, (\varphi^E(Z), Z_{I+1}) \notin B_{E_0/\sqrt{I+1}}(0)$ and the claim follows.
		
		As $T > 0$ was chosen arbitrarily, we have thus shown that $Z=(c,u)$ is a global renormalised solution. The weak formulation of the equation for $u$ already appeared in \eqref{equweaklimit}, while the conservation of the energy results from the energy conservation of $u^\varepsilon$ and the convergence $u^\varepsilon \rightarrow u$ in $L^2([0,T] \times \Omega)$. Finally, the bounds on the solution follow from Lemma \ref{lemmasubsequence} and \eqref{eqfinalbounds}.
	\end{proof}

	\appendix
	
	\section{Some auxiliary results}\label{sec:app}

	For the reader's convenience, we recall some classical results
	frequently used in this paper.
	
	\begin{lemma}[Gagliardo--Nirenberg inequality]
		\label{l:GN} 
		Let $1\le p<q$, define
		\begin{align}\label{eq:103}
		\theta=\frac{\frac{1}{p}-\frac{1}{q}}{\frac{1}{p}+\frac{1}{d}-\frac{1}{2}},
		\end{align}
		and suppose that $\theta\in(0,1)$.  There exists $C_1\in(0,\infty)$ such that
		for all $z\in H^1(\Omega)$
		\begin{align}\label{eq:106}
		\|z\|_{L^q(\Omega)}\le C_1\|\nabla z\|_{L^2(\Omega)}^\theta\|z\|_{L^p(\Omega)}^{1-\theta}+C_2\|z\|_{L^1(\Omega)}.
		\end{align}
	\end{lemma}
	
	The following lemma is a corollary of the Dunford--Pettis theorem if $p =
	1$. In fact, we only need the simpler version for $p\in(1,\infty)$, which
	follows from standard results on weak convergence.
	
	\begin{lemma}\label{l:weakConv}
		Let $p \in [1, \infty)$ and suppose that $f_j,f:\Omega\to\mathbb{R}$ are
		measurable functions and $g_j,g\in L^p(\Omega)$.  If $f_j\to f$ a.e.\ in
		$\Omega$, $\sup_j\|f_j\|_{L^\infty(\Omega)}<\infty$ and
		$g_j\rightharpoonup g$ in $L^p(\Omega)$, then
		\begin{align*}
		f_j g_j\rightharpoonup fg \quad\text{in } L^p(\Omega).
		\end{align*}
	\end{lemma}
	
	The next observation allows us to deal with coefficients in our system that 
	are singular near $u=0$, which may arise in the case of model~\HH.
	
	\begin{lemma}\label{l:1}
		Suppose that $a_j,V_j, j\in \mathbb{N},$ are measurable functions on a
		bounded domain $\Omega$ satisfying
		\begin{itemize}
			\item $V_j\rightharpoonup V$ in $L^1(\Omega)$
			\item $a_j\to a$ a.e.~in $\Omega$ (and $|a_j|\gtrsim 1$ for all $j$)
			\item $\sup_j\|a_jV_j\|_{L^{1+\epsilon}}<\infty$ for some $\epsilon>0$.
		\end{itemize}
		Then
		$a_jV_j\rightharpoonup aV\text{ in }L^{1+\epsilon}(\Omega).$
	\end{lemma}
	We remark that the hypothesis $|a_j|\gtrsim 1$ can be removed by writing
	$a_jV_j=a_j(\chi_{\tilde\Omega_j}V_j)+(1-\chi_{\tilde\Omega_j})a_jV_j$, where
	$\tilde\Omega_j=\{a_j\ge1\}$ a.e.
	\begin{proof}
		By weak compactness, there exists $X\in L^{1+\epsilon}(\Omega)$ such that,
		along a subsequence, $a_jV_j\rightharpoonup X$ in $L^{1+\epsilon}(\Omega)$.
		Assuming $|a_j|\gtrsim 1$ a.e., we can invoke Lemma~\ref{l:weakConv} to
		deduce that $V_j\rightharpoonup\tfrac{1}{a}X$ in $L^{1+\epsilon}(\Omega)$ and
		hence $X=aV$. The conclusion now follows from the observation that this
		argument applies to any subsequence.
	\end{proof}
	
	\paragraph*{\bf Acknowledgements} 
	M.K. gratefully acknowledges the hospitality of WIAS Berlin, where a major part of the project was carried out. The research stay of M.K. at WIAS Berlin was funded by the Austrian Federal Ministry of Education, Science and Research through a research
	fellowship for graduates of a promotio sub auspiciis.
	The research of A.M. has been partially supported
	by Deutsche Forschungsgemeinschaft (DFG) through the Collaborative Research
	Center SFB 1114 ``\emph{Scaling Cascades in Complex Systems}'' (Project no.\
	235221301),  Subproject C05 ``Effective models for materials and
	interfaces with multiple scales''. J.F. and A.M. are grateful for the hospitality of the Erwin Schr\"odinger Institute in Vienna, where some ideas for this work have been developed. The authors are grateful to two anonymous referees for several helpful comments, in particular for the short proof of estimate \eqref{eq:ERE.Bz}.


\begin{thebibliography}{10}
		
		\bibitem{agh_2002_thermodynamic}
		G.~Albinus, H.~Gajewski, and R.~H\"{u}nlich.
		\newblock Thermodynamic design of energy models of semiconductor devices.
		\newblock {\em Nonlinearity}, 15(2):367--383, 2002.
		
		\bibitem{amann_1992}
		H.~Amann.
		\newblock Nonhomogeneous linear and quasilinear elliptic and parabolic boundary
		value problems.
		\newblock In {\em Function spaces, differential operators and nonlinear
			analysis ({F}riedrichroda, 1992)}, volume 133 of {\em Teubner-Texte Math.},
		pages 9--126. Teubner, Stuttgart, 1993.
		
		\bibitem{ArnoldMarkowichToscani_2000_LargeTimeDriftDiffusionPoisson}
		A.~Arnold, P.~Markowich, and G.~Toscani.
		\newblock On large time asymptotics for drift-diffusion-{P}oisson systems.
		\newblock {\em Transport Theory Statist. Phys.}, 29(3-5):571--581, 2000.
		\newblock Proc. of the {F}ifth {I}nternational {W}orkshop on {M}athematical
		{A}spects of {F}luid and {P}lasma {D}ynamics ({M}aui, {HI}, 1998).
		
		\bibitem{BBGGPV_1995}
		P.~B\'{e}nilan, L.~Boccardo, T.~Gallou\"{e}t, R.~Gariepy, M.~Pierre, and J.~L.
		V\'{a}zquez.
		\newblock An {$L^1$}-theory of existence and uniqueness of solutions of
		nonlinear elliptic equations.
		\newblock {\em Ann. Scuola Norm. Sup. Pisa Cl. Sci. (4)}, 22(2):241--273, 1995.
		
		\bibitem{BG_1992}
		L.~Boccardo and T.~Gallou\"{e}t.
		\newblock Nonlinear elliptic equations with right-hand side measures.
		\newblock {\em Comm. Partial Differential Equations}, 17(3-4):641--655, 1992.
		
		\bibitem{BM_1992}
		L.~Boccardo and F.~Murat.
		\newblock Almost everywhere convergence of the gradients of solutions to
		elliptic and parabolic equations.
		\newblock {\em Nonlinear Anal.}, 19(6):581--597, 1992.
		
		\bibitem{BDPS_2010}
		M.~Burger, M.~Di~Francesco, J.-F. Pietschmann, and B.~Schlake.
		\newblock Nonlinear cross-diffusion with size exclusion.
		\newblock {\em SIAM J. Math. Anal.}, 42(6):2842--2871, 2010.
		
		\bibitem{CanizoDesvillettesFellner_2014_ImprovedDualityEstimates}
		J.~A. Ca\~{n}izo, L.~Desvillettes, and K.~Fellner.
		\newblock Improved duality estimates and applications to reaction-diffusion
		equations.
		\newblock {\em Comm. Partial Differential Equations}, 39(6):1185--1204, 2014.
		
		\bibitem{CGV_2019}
		M.~C. Caputo, T.~Goudon, and A.~F. Vasseur.
		\newblock Solutions of the 4-species quadratic reaction-diffusion system are
		bounded and {$C^\infty$}-smooth, in any space dimension.
		\newblock {\em Anal. PDE}, 12(7):1773--1804, 2019.
		
		\bibitem{CJ_2004_strongcross}
		L.~Chen and A.~J\"{u}ngel.
		\newblock Analysis of a multidimensional parabolic population model with strong
		cross-diffusion.
		\newblock {\em SIAM J. Math. Anal.}, 36(1):301--322, 2004.
		
		\bibitem{CJ_2006_noselfdiff}
		L.~Chen and A.~J\"{u}ngel.
		\newblock Analysis of a parabolic cross-diffusion population model without
		self-diffusion.
		\newblock {\em J. Differential Equations}, 224(1):39--59, 2006.
		
		\bibitem{ChenDausJuengel_2018_global}
		X.~Chen, E.~S. Daus, and A.~J\"{u}ngel.
		\newblock Global existence analysis of cross-diffusion population systems for
		multiple species.
		\newblock {\em Arch. Ration. Mech. Anal.}, 227(2):715--747, 2018.
		
		\bibitem{ChenJuengel_2019_renormalised}
		X.~Chen and A.~J\"{u}ngel.
		\newblock Global renormalized solutions to reaction-cross-diffusion systems
		with self-diffusion.
		\newblock {\em J. Differential Equations}, 267(10):5901--5937, 2019.
		
		\bibitem{CJ_2019_wkstruni}
		X.~Chen and A.~J\"{u}ngel.
		\newblock Weak-strong uniqueness of renormalized solutions to
		reaction-cross-diffusion systems.
		\newblock {\em Math. Models Methods Appl. Sci.}, 29(2):237--270, 2019.
		
		\bibitem{CJL_2014}
		X.~Chen, A.~J\"{u}ngel, and J.-G. Liu.
		\newblock A note on {A}ubin-{L}ions-{D}ubinski\u{\i} lemmas.
		\newblock {\em Acta Appl. Math.}, 133:33--43, 2014.
		
		\bibitem{ChenLiu_2013_globalWeak}
		X.~Chen and J.-G. Liu.
		\newblock Global weak entropy solution to {D}oi-{S}aintillan-{S}helley model
		for active and passive rod-like and ellipsoidal particle suspensions.
		\newblock {\em J. Differential Equations}, 254(7):2764--2802, 2013.
		
		\bibitem{DalMaso_etal_1999_RenormalizedSolutionsMeasureData}
		G.~Dal~Maso, F.~Murat, L.~Orsina, and A.~Prignet.
		\newblock Renormalized solutions of elliptic equations with general measure
		data.
		\newblock {\em Ann. Scuola Norm. Sup. Pisa Cl. Sci. (4)}, 28(4):741--808, 1999.
		
		\bibitem{DGJ_1997}
		P.~Degond, S.~G\'{e}nieys, and A.~J\"{u}ngel.
		\newblock Symmetrization and entropy inequality for general diffusion
		equations.
		\newblock {\em C. R. Acad. Sci. Paris S\'{e}r. I Math.}, 325(9):963--968, 1997.
		
		\bibitem{DesvillettesFellner_2006_ExponentialDecayViaEntropy}
		L.~Desvillettes and K.~Fellner.
		\newblock Exponential decay toward equilibrium via entropy methods for
		reaction-diffusion equations.
		\newblock {\em J. Math. Anal. Appl.}, 319(1):157--176, 2006.
		
		\bibitem{DesvillettesFellner_2007_EntropyMethods}
		L.~Desvillettes and K.~Fellner.
		\newblock Entropy methods for reaction-diffusion systems.
		\newblock In {\em Discrete Contin. Dyn. Syst. (suppl). Dynamical Systems and
			Differential Equations. Proceedings of the 6th AIMS International
			Conference}, pages 304--312, 2007.
		
		\bibitem{DesvillettesFellner_2008_EntropyMethodsAPrioriBounds}
		L.~Desvillettes and K.~Fellner.
		\newblock Entropy methods for reaction-diffusion equations: slowly growing
		a-priori bounds.
		\newblock {\em Rev. Mat. Iberoam.}, 24(2):407--431, 2008.
		
		\bibitem{DesvillettesFellner_2015_DualityMethodsDegenerateDiffusion}
		L.~Desvillettes and K.~Fellner.
		\newblock Duality and entropy methods for reversible reaction-diffusion
		equations with degenerate diffusion.
		\newblock {\em Math. Methods Appl. Sci.}, 38(16):3432--3443, 2015.
		
		\bibitem{Desvillettes_etal_2007_GlobalExistQuadrSystems}
		L.~Desvillettes, K.~Fellner, M.~Pierre, and J.~Vovelle.
		\newblock Global existence for quadratic systems of reaction-diffusion.
		\newblock {\em Adv. Nonlinear Stud.}, 7(3):491--511, 2007.
		
		\bibitem{DesvillettesFellnerTang_2017_ComplexBalance}
		L.~Desvillettes, K.~Fellner, and B.~Q. Tang.
		\newblock Trend to equilibrium for reaction-diffusion systems arising from
		complex balanced chemical reaction networks.
		\newblock {\em SIAM J. Math. Anal.}, 49(4):2666--2709, 2017.
		
		\bibitem{DLM_2014_entropy}
		L.~Desvillettes, T.~Lepoutre, and A.~Moussa.
		\newblock Entropy, duality, and cross diffusion.
		\newblock {\em SIAM J. Math. Anal.}, 46(1):820--853, 2014.
		
		\bibitem{DiFrancescoFellnerMarkowich_2008_EntropyDissipationInhomogeneous}
		M.~Di~Francesco, K.~Fellner, and P.~A. Markowich.
		\newblock The entropy dissipation method for inhomogeneous reaction--diffusion
		systems.
		\newblock {\em Proc. Royal Soc. A}, 464:3272--3300, 2008.
		
		\bibitem{DiPernaLions_1988_FokkerPlanckBoltzmann}
		R.~J. DiPerna and P.-L. Lions.
		\newblock On the {F}okker-{P}lanck-{B}oltzmann equation.
		\newblock {\em Comm. Math. Phys.}, 120(1):1--23, 1988.
		
		\bibitem{DiPernaLions_1989_CauchyProblemBoltzmann}
		R.~J. DiPerna and P.-L. Lions.
		\newblock On the {C}auchy problem for {B}oltzmann equations: global existence
		and weak stability.
		\newblock {\em Ann. of Math. (2)}, 130(2):321--366, 1989.
		
		\bibitem{DiPernaLions_1989_ODETransportSobolevSpaces}
		R.~J. DiPerna and P.-L. Lions.
		\newblock Ordinary differential equations, transport theory and {S}obolev
		spaces.
		\newblock {\em Invent. Math.}, 98(3):511--547, 1989.
		
		\bibitem{Dreher_2008_population}
		M.~Dreher.
		\newblock Analysis of a population model with strong cross-diffusion in
		unbounded domains.
		\newblock {\em Proc. Roy. Soc. Edinburgh Sect. A}, 138(4):769--786, 2008.
		
		\bibitem{DreherJuengel_2012_compactFamilies}
		M.~Dreher and A.~J\"{u}ngel.
		\newblock Compact families of piecewise constant functions in {$L^p(0,T;B)$}.
		\newblock {\em Nonlinear Anal.}, 75(6):3072--3077, 2012.
		
		\bibitem{FellnerTang_2017_DetailBalance}
		K.~Fellner and B.~Q. Tang.
		\newblock Explicit exponential convergence to equilibrium for nonlinear
		reaction-diffusion systems with detailed balance condition.
		\newblock {\em Nonlinear Anal.}, 159:145--180, 2017.
		
		\bibitem{FellnerTang_2018_ConvergenceOfRenormalizedSolutions}
		K.~Fellner and B.~Q. Tang.
		\newblock Convergence to equilibrium of renormalised solutions to nonlinear
		chemical reaction-diffusion systems.
		\newblock {\em Z. Angew. Math. Phys.}, 69(3):Paper No. 54, 30, 2018.
		
		\bibitem{Fischer_2015_renormalized}
		J.~Fischer.
		\newblock Global existence of renormalized solutions to entropy-dissipating
		reaction-diffusion systems.
		\newblock {\em Arch. Ration. Mech. Anal.}, 218(1):553--587, 2015.
		
		\bibitem{Fischer_2017}
		J.~Fischer.
		\newblock Weak-strong uniqueness of solutions to entropy-dissipating
		reaction-diffusion equations.
		\newblock {\em Nonlinear Anal.}, 159:181--207, 2017.
		
		\bibitem{Gajewski_Groeger_1996_ReactionDiffusionElectrically}
		H.~Gajewski and K.~Gr\"{o}ger.
		\newblock Reaction-diffusion processes of electrically charged species.
		\newblock {\em Math. Nachr.}, 177:109--130, 1996.
		
		\bibitem{GGJ_2003}
		G.~Galiano, M.~L. Garz\'{o}n, and A.~J\"{u}ngel.
		\newblock Semi-discretization in time and numerical convergence of solutions of
		a nonlinear cross-diffusion population model.
		\newblock {\em Numer. Math.}, 93(4):655--673, 2003.
		
		\bibitem{GilbargTrudinger_book}
		D.~Gilbarg and N.~S. Trudinger.
		\newblock {\em Elliptic partial differential equations of second order}.
		\newblock Classics in Mathematics. Springer-Verlag, Berlin, 2001.
		\newblock Reprint of the 1998 edition.
		
		\bibitem{GlitzkyGroegerHuenlich_1996_FreeEnergyDissipationRate}
		A.~Glitzky, K.~Gr\"{o}ger, and R.~H\"{u}nlich.
		\newblock Free energy and dissipation rate for reaction diffusion processes of
		electrically charged species.
		\newblock {\em Appl. Anal.}, 60:201--217, 1996.
		
		\bibitem{GlitzkyHuenlich_1997_ElectroReactionDiffusion}
		A.~Glitzky and R.~H\"{u}nlich.
		\newblock Global estimates and asymptotics for electro-reaction-diffusion
		systems in heterostructures.
		\newblock {\em Appl. Anal.}, 66(3-4):205--226, 1997.
		
		\bibitem{hhmm_2018}
		J.~Haskovec, S.~Hittmeir, P.~Markowich, and A.~Mielke.
		\newblock Decay to equilibrium for energy-reaction-diffusion systems.
		\newblock {\em SIAM J. Math. Anal.}, 50(1):1037--1075, 2018.
		
		\bibitem{J_2000_thermo}
		A.~J\"{u}ngel.
		\newblock Regularity and uniqueness of solutions to a parabolic system in
		nonequilibrium thermodynamics.
		\newblock {\em Nonlinear Anal.}, 41(5-6, Ser. A: Theory Methods):669--688,
		2000.
		
		\bibitem{Juengel_2015_boundednes-by-entropy}
		A.~J\"{u}ngel.
		\newblock The boundedness-by-entropy method for cross-diffusion systems.
		\newblock {\em Nonlinearity}, 28(6):1963--2001, 2015.
		
		\bibitem{LM_2017_entropic}
		T.~Lepoutre and A.~Moussa.
		\newblock Entropic structure and duality for multiple species cross-diffusion
		systems.
		\newblock {\em Nonlinear Anal.}, 159:298--315, 2017.
		
		\bibitem{LY_1999_secondLaw}
		E.~H. Lieb and J.~Yngvason.
		\newblock The physics and mathematics of the second law of thermodynamics.
		\newblock {\em Phys. Rep.}, 310(1):96, 1999.
		
		\bibitem{Mielke_2011_generic}
		A.~Mielke.
		\newblock Formulation of thermoelastic dissipative material behavior using
		{GENERIC}.
		\newblock {\em Contin. Mech. Thermodyn.}, 23(3):233--256, 2011.
		
		\bibitem{Mielke_2011_gradientStructure}
		A.~Mielke.
		\newblock A gradient structure for reaction-diffusion systems and for
		energy-drift-diffusion systems.
		\newblock {\em Nonlinearity}, 24(4):1329--1346, 2011.
		
		\bibitem{Mielke_2013_thermomechanical}
		A.~Mielke.
		\newblock Thermomechanical modeling of energy-reaction-diffusion systems,
		including bulk-interface interactions.
		\newblock {\em Discrete Contin. Dyn. Syst. Ser. S}, 6(2):479--499, 2013.
		
		\bibitem{MielkeMittnenzweig_2018_convergence}
		A.~Mielke and M.~Mittnenzweig.
		\newblock Convergence to equilibrium in energy-reaction-diffusion systems using
		vector-valued functional inequalities.
		\newblock {\em J. Nonlinear Sci.}, 28(2):765--806, 2018.
		
		\bibitem{Pierre_2010_GlobalExistenceSurvey}
		M.~Pierre.
		\newblock Global existence in reaction-diffusion systems with control of mass:
		a survey.
		\newblock {\em Milan J. Math.}, 78(2):417--455, 2010.
		
		\bibitem{Roubicek_2013_npde}
		T.~Roub\'{\i}\v{c}ek.
		\newblock {\em Nonlinear partial differential equations with applications},
		volume 153 of {\em International Series of Numerical Mathematics}.
		\newblock Birkh\"{a}user/Springer Basel AG, Basel, second edition, 2013.
		
		\bibitem{Villani_1996_CauchyProblemLandauEquation}
		C.~Villani.
		\newblock On the {C}auchy problem for {L}andau equation: sequential stability,
		global existence.
		\newblock {\em Adv. Differential Equations}, 1(5):793--816, 1996.
		
	\end{thebibliography}

\end{document}